\documentclass[aap,preprint,oneside]{imsart}

\usepackage{amsthm,amsmath}
\usepackage{amssymb,amscd}
\usepackage{mathrsfs,eucal}
\usepackage{enumerate}
\usepackage[all,cmtip]{xy}

\usepackage{times}

\usepackage[margin=1.505in]{geometry}

\startlocaldefs
\def\journal@id{~}
\def\journal@name{~}
\def\journal@url{~}
\endlocaldefs

\newtheorem{theorem}{Theorem}[section]
\newtheorem{cor}[theorem]{Corollary}
\newtheorem{lem}[theorem]{Lemma}
\newtheorem{prop}[theorem]{Proposition}

\theoremstyle{definition}
\newtheorem{defn}[theorem]{Definition}

\theoremstyle{remark}
\newtheorem{remark}[theorem]{Remark}

\numberwithin{equation}{section}

\newcommand{\outint}{\int^*\!}
\newcommand{\card}{\mathop{\mathrm{card}}}
\newcommand{\bbS}{\mathbb{S}}
\newcommand{\bbX}{\mathbb{X}}
\newcommand{\bbY}{\mathbb{Y}}

\newcommand{\vvverto}{\mathopen{|\hspace{-0.12em}|\hspace{-0.12em}|}}
\newcommand{\vvvertc}{\mathclose{|\hspace{-0.12em}|\hspace{-0.12em}|}}
\newcommand{\tnorm}[1]{{\vvverto #1 \vvvertc}}

\begin{document}

\begin{frontmatter}

\title{Comparison Theorems for Gibbs Measures\protect\thanksref{T1}}
\runtitle{Comparison Theorems for Gibbs Measures}
\thankstext{T1}{Supported in part
by NSF grants DMS-1005575 and CAREER-DMS-1148711.}

\begin{aug}
\author{\fnms{Patrick} 
\snm{Rebeschini}\ead[label=e2]{prebesch@princeton.edu}}
\and
\author{\fnms{Ramon} \snm{van Handel}\ead[label=e3]{rvan@princeton.edu}}
\runauthor{Patrick Rebeschini and Ramon van Handel}
\affiliation{Princeton University}
\address{Sherrerd Hall\\
Princeton University \\
Princeton, NJ 08544 \\
USA \\ \printead{e2} \\ \printead{e3}}
\end{aug}

\begin{abstract} 
The Dobrushin comparison theorem is a powerful tool to bound the difference 
between the marginals of high-dimensional probability distributions in terms 
of their local specifications.  Originally introduced to prove uniqueness 
and decay of correlations of Gibbs measures, it has been widely used in 
statistical mechanics as well as in the analysis of algorithms on random 
fields and interacting Markov chains. However, the classical comparison 
theorem requires validity of the Dobrushin uniqueness criterion, essentially 
restricting its applicability in most models to a small subset of the 
natural parameter space.  In this paper we develop generalized Dobrushin 
comparison theorems in terms of influences between blocks of sites, in the 
spirit of Dobrushin-Shlosman and Weitz, that substantially extend the range 
of applicability of the classical comparison theorem.  Our proofs are based 
on the analysis of an associated family of Markov chains.  We develop in 
detail an application of our main results to the analysis of sequential 
Monte Carlo algorithms for filtering in high dimension.
\end{abstract}


\end{frontmatter}

\setcounter{tocdepth}{2}
{\small\tableofcontents}

\section{Introduction}

The canonical description of a high-dimensional random system is provided 
by specifying a probability measure $\rho$ on a (possibly infinite) 
product space $\bbS=\prod_{i\in I}\bbS^i$: each site $i\in I$ represents a 
single degree of freedom, or dimension, of the model.  When $I$ is 
defined as the set of vertices of a graph, the measure $\rho$ defines a 
graphical model or a random field.  Models of this type are ubiquitous in 
statistical mechanics, combinatorics, computer science, statistics, and in 
many other areas of science and engineering.

Let $\rho$ and $\tilde\rho$ be two such models that are defined on the 
same space $\bbS$.  We would like to address the following basic question: 
when is $\tilde\rho$ a good approximation of $\rho$?  Such questions arise 
at a basic level not only in understanding the properties of random 
systems themselves, but also in the analysis of the algorithms that are 
used to investigate and approximate these systems.  Of course, probability 
theory provides numerous methods to evaluate the difference between 
arbitrary probability measures, but the high-dimensional setting brings 
some specific challenges: any approximation of practical utility in high 
dimension must yield error bounds that do not grow, or at least grow 
sufficiently slowly, in the model dimension.  We therefore seek 
quantitative methods that allow to establish dimension-free bounds on 
high-dimensional probability distributions.

A general method to address precisely this problem was developed by 
Dobrushin \cite{Dob70} in the context of statistical mechanics.  In the 
approach pioneered by Dobrushin, Lanford, and Ruelle, an 
infinite-dimensional system of interacting particles is defined by its 
\emph{local} description: for finite sets of sites $J\subset I$, the 
conditional distribution $\rho(dx^J|x^{I\backslash J})$ of the 
configuration in $J$ is specified given that the particles outside $J$ are 
frozen in a fixed configuration.  This local description is a direct 
consequence of the physical parameters of the problem.  The model $\rho$ 
is then defined as a probability measure (called a Gibbs measure) that is 
compatible with the given system of local conditional distributions; see 
section \ref{sec:setting}.  This setting gives rise to many classical 
questions in statistical mechanics \cite{Geo11,Sim93}; for example, the 
Gibbs measure may or may not be unique, reflecting the presence of a phase 
transition (akin to the transition from water to ice at the freezing 
point).  

The Dobrushin comparison theorem \cite[Theorem 3]{Dob70} provides a 
powerful tool to obtain dimension-free estimates on the difference 
between the marginals of Gibbs measures $\rho$ and $\tilde\rho$ in terms 
of the single site conditional distributions 
$\rho(dx^j|x^{I\backslash\{j\}})$ and 
$\tilde\rho(dx^j|x^{I\backslash\{j\}})$.  In its simplified form due to 
F\"ollmer \cite{Fol82}, this result has become standard textbook material, 
cf.\ \cite[Theorem 8.20]{Geo11}, \cite[Theorem V.2.2]{Sim93}. It is widely 
used to establish numerous properties of Gibbs measures, including 
uniqueness, decay of correlations, global Markov properties, and 
analyticity \cite{Geo11,Sim93,Fol88}, as well as functional inequalities 
and concentration of measure properties \cite{GZ03,Kul03,Wu06}, and has 
similarly proved to be useful in the analysis of algorithms on random 
fields and interacting Markov chains \cite{You89,Tat03,DSS12,RvH13}.

Despite this broad array of applications, the range of applicability of 
the Dobrushin comparison theorem proves to be somewhat limited.  This can 
already be seen in the easiest qualitative consequence of this result: the 
comparison theorem implies uniqueness of the Gibbs measure under the 
well-known Dobrushin uniqueness criterion \cite{Dob70}. Unfortunately, 
this criterion is restrictive: even in models where uniqueness can be 
established by explicit computation, the Dobrushin uniqueness criterion 
holds only in a small subset of the natural parameter space (see, e.g., 
\cite{Wei05} for examples).  This suggests that the Dobrushin comparison 
theorem is a rather blunt tool. On the other hand, it is also known that 
the Dobrushin uniqueness criterion can be substantially improved: this was 
accomplished in Dobrushin and Shlosman \cite{DS85} by considering a local 
description in terms of larger blocks $\rho(dx^J|x^{I\backslash J})$ 
instead of the single site specification 
$\rho(dx^j|x^{I\backslash\{j\}})$.  In this manner, it is possible in many 
cases to capture a large part of or even the entire uniqueness region. The 
uniqueness results of Dobrushin and Shlosman were further generalized by 
Weitz \cite{Wei05}, who developed remarkably general combinatorial 
criteria for uniqueness.  However, while the proofs of Dobrushin-Shlosman 
and Weitz also provide some information on decay of correlations, they do 
not provide an analogue of the powerful general-purpose machinery that the 
Dobrushin comparison theorem yields in its more restrictive setting.

The aim of the present paper is to fill this gap. Our main results 
(Theorem \ref{thm:main} and Theorem \ref{thm:oneside}) provide a direct 
generalization of the Dobrushin comparison theorem to the much more 
general setting considered by Weitz \cite{Wei05}, substantially extending 
the range of applicability of the classical comparison theorem.  While the 
classical comparison theorem is an immediate consequence of our main 
result (Corollary \ref{cor:follmer}), the classical proof that is based on 
the ``method of estimates'' does not appear to extend easily beyond the 
single site setting.  We therefore develop a different, though certainly 
related, method of proof that systematically exploits the connection of 
Markov chains.  In particular, our main results are derived from a more 
general comparison theorem for Markov chains that is applied to a suitably 
defined family of Gibbs samplers, cf.\ section \ref{sec:proofs} below.

Our original motivation for developing the generalized comparison theorems 
of this paper was the investigation of algorithms for filtering in high 
dimension.  Filtering---the computation of the conditional distributions 
of a hidden Markov process given observed data---is a problem that arises 
in a wide array of applications in science and engineering.  Modern 
filtering algorithms utilize sequential Monte Carlo methods to efficiently 
approximate the conditional distributions \cite{CMR05}.  Unfortunately, 
such algorithms suffer heavily from the curse of dimensionality, making 
them largely useless in complex data assimilation problems that arise in 
high-dimensional applications such as weather forecasting (the 
state-of-the-art in such applications is still dominated by \emph{ad-hoc} 
methods).  Motivated by such problems, we have begun to investigate in 
\cite{RvH13} a class of regularized filtering algorithms that can, in 
principle, exhibit dimension-free performance in models that possess decay 
of correlations.  For the simplest possible algorithm of this type, 
dimension-free error bounds are proved in \cite{RvH13} by systematic 
application of the Dobrushin comparison theorem.  

In order to ensure decay of correlations, the result of \cite{RvH13} 
imposes a weak interactions assumption that is dictated by the Dobrushin 
comparison theorem.  As will be explained in section \ref{sec:appl}, 
however, this assumption is unsatisfactory already at the 
\emph{qualitative} level: it limits not only the spatial interactions (as 
is needed to ensure decay of correlations) but also the dynamics in time. 
Overcoming this unnatural restriction requires a generalized version of 
the comparison theorem, which provided the motivation for our main 
results.  As an illustration of our main results, and as a problem of 
interest in its own right, the application to filtering algorithms will be 
developed in detail in section \ref{sec:appl}.

The remainder of this paper is organized as follows.  Section 
\ref{sec:mainres} introduces the basic setup and notation to be used 
throughout the paper, and states our main results.  While the comparison 
theorem, being quantitative in nature, is already of significant interest 
in the finite setting $\card I<\infty$ (unlike the qualitative uniqueness 
questions that are primarily of interest when $\card I=\infty$), we will 
develop our main results in a general setting that admits even 
infinite-range interactions. The proofs of the main results are given in 
section \ref{sec:proofs}.  The application to filtering algorithms is 
finally developed in section \ref{sec:appl}.

\section{Main results}
\label{sec:mainres}

\subsection{Setting and notation}
\label{sec:setting}

We begin by introducing the basic setting that 
will be used throughout this section.  

\subsubsection*{Sites and configurations}

Let $I$ be a finite or countably infinite set of \emph{sites}.  Each 
subset $J\subseteq I$ is called a \emph{region}; the set of finite
regions will be denoted as
$$
	\mathcal{I} := \{J\subseteq I:\card J<\infty\}.
$$
To each site $i\in I$ is associated a measurable space $\bbS^i$, the 
\emph{local state space}.  A \emph{configuration} is an assignment 
$x_i\in\bbS^i$ to each site $i\in I$.  The set of all configurations 
$\bbS$, and the set $\bbS^J$ of configurations in a given region 
$J\subseteq I$, are defined as
$$
	\bbS := \prod_{i\in I}\bbS^i,\qquad\qquad
	\bbS^J := \prod_{i\in J}\bbS^i.
$$
For $x=(x_i)_{i\in I}\in\bbS$, we denote by $x^J:=(x_i)_{i\in J}\in\bbS^J$ 
the natural projection on $\bbS^J$. When $J\cap K=\varnothing$, we define 
$z=x^Jy^K\in\bbS^{J\cup K}$ such that $z^J=x^J$ and $z^K=y^K$.

\subsubsection{Local functions}

A function $f:\bbS\to\mathbb{R}$ is said to be \emph{$J$-local} 
if $f(x)=f(z)$ whenever $x^J=z^J$, that is, if $f(x)$ depends on $x^J$ 
only.  The function $f$ is said to be \emph{local} if it is $J$-local for
some finite region $J\in\mathcal{I}$.  When $I$ is a finite set, every 
function is local.  When $I$ is infinite, however, we will frequently 
restrict attention to local functions.  More generally, we will consider 
a class of ``nearly'' local functions to be defined presently.

Given any function $f:\bbS\to\mathbb{R}$, let us define for 
$J\in\mathcal{I}$ and $x\in\bbS$ the $J$-local function
$$
	f_x^J(z) := f(z^Jx^{I\backslash J}).
$$
Then $f$ is called 
\emph{quasilocal} if it can be approximated pointwise by the local 
functions $f_x^J$:
$$
	\lim_{J\in\mathcal{I}}|f_x^J(z)-f(z)|=0\quad
	\mbox{for all }x,z\in\bbS,
$$
where $\lim_{J\in\mathcal{I}}a_J$ denotes the limit of the net 
$(a_J)_{J\in\mathcal{I}}$ where $\mathcal{I}$ is directed by inclusion
$\subseteq$ (equivalently, $a_J\to 0$ if and only if $a_{J_i}\to 0$
for every sequence $J_1,J_2,\ldots\in\mathcal{I}$ such that
$J_1\subseteq J_2\subseteq\cdots$ and $\bigcup_iJ_i=I$).  Let us note 
that this notion is slightly weaker than the conventional notion of 
quasilocality used, for example, in  \cite{Geo11}.

\subsubsection*{Metrics}

In the sequel, we fix for each $i\in I$ a metric 
$\eta_i$ on $\bbS^i$ (we assume throughout that $\eta_i$ is measurable as 
a function on $\bbS^i\times\bbS^i$).  
We will write $\|\eta_i\|=\sup_{x,z}\eta_i(x,z)$.

Given a function $f:\bbS\to\mathbb{R}$ and $i\in I$, we define
$$
	\delta_if :=
	\sup_{x,z\in\bbS:x^{I\backslash\{i\}}=z^{I\backslash\{i\}}}
	\frac{|f(x)-f(z)|}{\eta_i(x_i,z_i)}. 
$$
The quantity $\delta_if$ measures the variability of $f(x)$ with respect 
to the variable $x_i$.

\subsubsection*{Matrices}

The calculus of possibly infinite nonnegative matrices will appear 
repeatedly in the sequel.  Given matrices $A=(A_{ij})_{i,j\in I}$ and 
$B=(B_{ij})_{i,j\in I}$ with nonnegative entries $A_{ij}\ge 0$ and 
$B_{ij}\ge 0$, the matrix product is defined as usual by
$$
	(AB)_{ij} = \sum_{k\in I}A_{ik}B_{kj}.
$$
This quantity is well defined as the terms in the sum are all nonnegative, 
but $(AB)_{ij}$ may possibly take the value $+\infty$.  As long as we 
consider only nonnegative matrices, all the usual rules of matrix 
multiplication extend to infinite matrices provided that we allow entries 
with the value $+\infty$ and that we use the convention $+\infty\cdot 0=0$ 
(this follows from the Fubini-Tonelli theorem, cf.\ \cite[Chapter 4]{Dud02}).
In particular, the matrix powers $A^k$, $k\ge 1$ are well defined, and we define
$A^0=I$ where $I:=(\mathbf{1}_{i=j})_{i,j\in I}$ denotes the identity 
matrix.  We will write $A<\infty$ if the nonnegative matrix $A$ satisfies
$A_{ij}<\infty$ for every $i,j\in I$.

\subsubsection{Kernels, covers, local structure}

Recall that a \emph{transition kernel} $\gamma$ from a measurable space
$(\Omega,\mathcal{F})$ to a measurable space $(\Omega',\mathcal{F}')$
is a map $\gamma:\Omega\times\mathcal{F}'\to\mathbb{R}$ such that
$\omega\mapsto\gamma_\omega(A)$ is a measurable function for each 
$A\in\mathcal{F}'$ and $\gamma_\omega(\cdot)$ is a probability 
measure for each $\omega\in\Omega$, cf.\ \cite{Kal02}.
Given a probability measure $\mu$ on
$\Omega$ and function $f$ on $\Omega'$, we define as usual
the probability measure $(\mu\gamma)(A) = \int \mu(d\omega)\gamma_\omega(A)$
on $\Omega'$ and function $(\gamma f)(\omega)=\int 
\gamma_\omega(d\omega')f(\omega')$ on $\Omega$.  A transition kernel 
$\gamma$ between product spaces is called \emph{quasilocal} if 
$\gamma f$ is quasilocal for every bounded and measurable quasilocal 
function $f$.

Our interest throughout this paper is in models of random configurations, 
described by a probability measure $\mu$ on $\bbS$.  We would like to 
understand the properties of such models based on their \emph{local} 
structure.  A natural way to express the local structure in a finite 
region $J\in\mathcal{I}$ is to consider the conditional distribution 
$\gamma^J_x(dz^J)=\mu(dz^J|x^{I\backslash J})$ of the configuration in $J$ 
given a fixed configuration $x^{I\backslash J}$ for the sites outside $J$: 
conceptually, $\gamma^J$ describes how the sites in $J$ ``interact'' with 
the sites outside $J$.
The conditional distribution $\gamma^J$ is a transition kernel from $\bbS$ 
to $\bbS^J$.  To obtain a complete local description of the model, we must 
consider a class of finite regions $J$ that covers the entire set of sites 
$I$.  Let us call a collection of regions 
$\mathcal{J}\subseteq\mathcal{I}$ a \emph{cover} of $I$ if every site 
$i\in I$ is contained in at least one element of $\mathcal{J}$ (note that,
by definition, a cover contains only finite regions).
Given any cover $\mathcal{J}$, the collection 
$(\gamma^J)_{J\in\mathcal{J}}$ provides a local description of the model.

In fact, our main results will hold in a somewhat more general setting 
than is described above.  Let $\mu$ be a probability measure on $\bbS$ and 
$\gamma^J$ be transition kernel from $\bbS$ to $\bbS^J$.  We say that 
$\mu$ is \emph{$\gamma^J$-invariant} if for every bounded measurable 
function $f$
$$
	\int \mu(dx)\,f(x) = \int 
	\mu(dx)\,\gamma^J_x(dz^J)\,f(z^Jx^{I\backslash J});
$$
by a slight abuse of notation, we will also write $\mu f=\mu\gamma^Jf^J$. 
This means that if the configuration $x$ is drawn according to $\mu$, then 
its distribution is left unchanged if we replace the configuration $x^J$ 
inside the region $J$ by a random sample from the distribution 
$\gamma^J_x$, keeping the configuration $x^{I\backslash J}$ outside $J$ 
fixed.
Our main results will be formulated in terms of a collection of transition 
kernels $(\gamma^J)_{J\in\mathcal{J}}$ such that $\mathcal{J}$ is a cover 
of $I$ and such that $\mu$ is $\gamma^J$-invariant for every 
$J\in\mathcal{J}$.
If we choose $\gamma^J_x(dz^J) = \mu(dz^J|x^{I\backslash J})$ as above, 
then the $\gamma^J$-invariance of $\mu$ holds by construction 
\cite[Theorem 6.4]{Kal02}; however, any family of $\gamma^J$-invariant 
kernels will suffice for the validity of our main results.

\begin{remark}
The idea that the collection $(\gamma^J)_{J\in\mathcal{J}}$ provides a 
natural description of high-dimensional probability distributions is 
prevalent in many applications.  In fact, in statistical mechanics, the 
model is usually \emph{defined} in terms of such a family.  To this end, 
one fixes a priori a family of transition kernels 
$(\gamma^J)_{J\in\mathcal{I}}$, called a \emph{specification}, that 
describes the local structure of the model. The definition of $\gamma^J$ 
is done directly in terms of the parameters of the problem (the potentials 
that define the physical interactions, or the local constraints that 
define the combinatorial structure). A measure $\mu$ on $\bbS$ is called a 
\emph{Gibbs measure} for the given specification if 
$\mu(dz^J|x^{I\backslash J})= \gamma^J_x(dz^J)$ for every 
$J\in\mathcal{I}$.  The existence of a Gibbs measure allows to define the 
model $\mu$ in terms of the specification.  It may happen that there are 
multiple Gibbs measures for the same specification: the significance of 
this phenomenon is the presence of a \emph{phase transition}, akin to the 
transition of water from liquid to solid at the freezing point.
As the construction of Gibbs measures from specifications is not essential
for the validity or applicability of our results (cf.\ section 
\ref{sec:appl} below), we omit further details.  
We refer to \cite{Geo11,Sim93,Wei05} for extensive discussion, examples, 
and references.
\end{remark}

\subsection{Main result}
\label{sec:main}

Let $\rho$ and $\tilde\rho$ be probability measures on the space of 
configurations $\bbS$.  Our main result, Theorem \ref{thm:main} below, 
provides a powerful tool to obtain quantitative bounds on the difference 
between $\rho$ and $\tilde\rho$ in terms of their local structure.  
Before we can state our results, we must first introduce some basic 
notions.  Our terminology is inspired by Weitz \cite{Wei05}.

As was explained above, the local description of a 
probability measure $\rho$ on $\bbS$ will be provided in terms of a family 
of transition kernels.  We formalize this as follows.

\begin{defn}
\label{defn:localupd}
A \emph{local update rule} for $\rho$ is a collection 
$(\gamma^J)_{J\in\mathcal{J}}$ where $\mathcal{J}$ is a cover of $I$, 
$\gamma^J$ is a transition kernel from $\bbS$ to $\bbS^J$ and $\rho$ is 
$\gamma^J$-invariant for every $J\in\mathcal{J}$.
\end{defn}

In order to compare two measures $\rho$ and $\tilde\rho$ on the basis of 
their local update rules $(\gamma^J)_{J\in\mathcal{J}}$ and 
$(\tilde\gamma^J)_{J\in\mathcal{J}}$, we must quantify two separate 
effects.  On the one hand, we must understand how the two models differ 
locally: that is, we must quantify how $\gamma_x^J$ and $\tilde\gamma_x^J$ 
differ when acting on the same configuration $x$.    On the other hand, we 
must understand how perturbations to the local update rule in different 
regions interact: to this end, we will quantify the extent to which 
$\gamma_x^J$ and $\gamma_z^J$ differ for different configurations $x,z$.
Both effects will be addressed by introducing a suitable family of 
couplings.  Recall that a probability measure $Q$ on a product space 
$\Omega\times\Omega$ is called a \emph{coupling} of probability measures 
$\mu,\nu$ on $\Omega$ if its marginals coincide with $\mu,\nu$, that is, 
$Q({}\cdot{}\times\Omega)=\mu$ and $Q(\Omega\times{}\cdot{})=\nu$.

\begin{defn}
\label{defn:coupledupd}
A \emph{coupled update rule} for 
$(\rho,\tilde\rho)$ is a collection $(\gamma^J,\tilde\gamma^J,Q^J,
\hat Q^J)_{J\in\mathcal{J}}$, where
$\mathcal{J}$ is a cover of $I$, such that the following properties hold:
\begin{enumerate}
\item $(\gamma^J)_{J\in\mathcal{J}}$ and $(\tilde\gamma^J)_{J\in\mathcal{J}}$
are local update rules for $\rho$ and $\tilde\rho$, respectively.
\item $Q^J_{x,z}$ is a coupling of $\gamma^J_x,\gamma^J_z$ for every
$J\in\mathcal{J}$ and $x,z\in\bbS$ with $\card\{i:x_i\ne z_i\}=1$.
\item $\hat Q^J_x$ is a coupling of $\gamma^J_x,\tilde\gamma^J_x$ for
every $J\in\mathcal{J}$ and $x\in\bbS$.
\end{enumerate}
\end{defn}

We can now state our main result.  The proof will be given in sections 
\ref{sec:step1}--\ref{sec:step3}.

\begin{theorem}
\label{thm:main}
Let $\mathcal{J}$ be a cover of $I$, let $(w_J)_{J\in\mathcal{J}}$ be 
a family of strictly positive weights, and let 
$(\gamma^J,\tilde\gamma^J,Q^J,\hat Q^J)_{J\in\mathcal{J}}$ be a coupled 
update rule for $(\rho,\tilde\rho)$.  Define for $i,j\in I$
\begin{align*}
	W_{ij} &:= \mathbf{1}_{i=j}\sum_{J\in\mathcal{J}:i\in J}w_J,\\
	R_{ij} &:= \sup_{
	\substack{x,z\in\bbS : \\ x^{I\backslash\{j\}} = z^{I\backslash \{j\}}}
	} \frac{1}{\eta_j(x_j,z_j)}\sum_{J\in\mathcal{J}:i\in J}
	w_J\, Q^J_{x,z}\eta_i, \\
	a_j &:= \sum_{J\in\mathcal{J}:j\in J}w_J\outint\tilde\rho(dx)\,
	\hat Q^J_x\eta_j.
\end{align*}
Assume that $\gamma^J$ is quasilocal for every $J\in\mathcal{J}$, and that
\begin{equation}
\label{eq:follmer}
	W_{ii}\le 1\quad\mbox{and}\quad
	\lim_{n\to\infty} \sum_{j\in I}(I-W+R)^n_{ij}\,
	(\rho\otimes\tilde\rho)\eta_j=0
	\quad\mbox{for all }i\in I.
\end{equation}
Then we have
$$
	|\rho f-\tilde\rho f| \le
	\sum_{i,j\in I}\delta_if\,D_{ij}\,W_{jj}^{-1}a_j
	\quad\mbox{where}\quad
	D:=\sum_{n=0}^\infty (W^{-1}R)^n,
$$
for any bounded and measurable quasilocal function $f$ such that
$\delta_if<\infty$ for all $i\in I$.
\end{theorem}

\begin{remark}
While it is essential in the proof that $\gamma^J$ and $\tilde\gamma^J$ 
are transition kernels, we do not require that $Q^J$ and $\hat Q^J$ are 
transition kernels in Definition \ref{defn:coupledupd}, that is, the 
couplings $Q^J_{x,z}$ and $\hat Q^J_x$ need not be measurable as 
functions of $x,z$.  It is for this reason that the coefficients $a_j$ are 
defined in terms of an outer integral 
rather than an ordinary integral \cite{VW96}:
$$
	\outint f(x)\,\rho(dx) := 
	\inf\bigg\{
	\int g(x)\,\rho(dx):f\le g,~g\mbox{ is measurable}
	\bigg\}.
$$
When $x\mapsto\hat Q^J_x\eta_j$ is measurable this 
issue can be disregarded.  In practice measurability will hold in all but 
pathological cases, but may not always be trivial to prove. We therefore 
allow for nonmeasurable couplings for sake of technical convenience, so 
that it is not necessary to check measurability of the coupled updates 
when applying Theorem \ref{thm:main}.
\end{remark}

We will presently formulate a number of special cases and extensions of 
Theorem \ref{thm:main} that may be useful in different settings.  A 
detailed application is developed in section \ref{sec:appl}.

\subsection{The classical comparison theorem}

The original comparison theorem of Dobrushin \cite[Theorem 3]{Dob70} and 
its commonly used formulation due to F\"ollmer \cite{Fol82} correspond to 
the special case of Theorem \ref{thm:main} where the cover 
$\mathcal{J}=\mathcal{J}_{\mathrm{s}}:=\{\{i\}:i\in I\}$ consists of 
single sites.  For example, the main result of \cite{Fol82} follows 
readily from Theorem \ref{thm:main} under a mild regularity assumption.  
To formulate it, recall that the \emph{Wasserstein distance} 
$d_\eta(\mu,\nu)$ between probability measures $\mu$ and $\nu$ on a 
measurable space $\Omega$ with respect to a measurable metric $\eta$ is 
defined as
$$
	d_\eta(\mu,\nu) := \inf_{
	\substack{Q(\mbox{}\cdot\mbox{}\times\Omega)=\mu\\
	Q(\Omega\times\mbox{}\cdot\mbox{})=\nu}}  Q\eta,
$$
where the infimum is taken over probability measures $Q$ on 
$\Omega\times\Omega$ with the given marginals $\mu$ and $\nu$.
We now obtain the following classical result (cf.\ 
\cite{Fol82} and \cite[Remark 2.17]{Fol88}).

\begin{cor}[\rm \cite{Fol82}]
\label{cor:follmer}
Assume $\bbS^i$ is Polish and $\eta_i$ is lower-semicontinuous
for all $i\in I$.
Let $(\gamma^{\{i\}})_{i\in I}$ and $(\tilde\gamma^{\{i\}})_{i\in I}$
be local update rules for $\rho$ and $\tilde\rho$, respectively, and let
$$
	C_{ij} := \sup_{
        \substack{x,z\in\bbS : \\ x^{I\backslash\{j\}} = z^{I\backslash \{j\}}}
        } 
	\frac{d_{\eta_i}(\gamma^{\{i\}}_x,\gamma^{\{i\}}_z)}{\eta_j(x_j,z_j)},
	\qquad
	b_j := \outint\tilde\rho(dx)\,d_{\eta_j}(\gamma^{\{j\}}_x,
	\tilde\gamma^{\{j\}}_x).
$$
Assume that $\gamma^{\{i\}}$ is quasilocal for every $i\in I$, and that
$$
	\lim_{n\to\infty}\sum_{j\in I}C_{ij}^n
	(\rho\otimes\tilde\rho)\eta_j=0
	\quad\mbox{for all }i\in I.
$$
Then we have
$$
	|\rho f-\tilde\rho f| \le
	\sum_{i,j\in I}\delta_if\,D_{ij}\,b_j
	\quad\mbox{where}\quad
	D:=\sum_{n=0}^\infty C^n,
$$
for any bounded and measurable quasilocal function $f$ such that
$\delta_if<\infty$ for all $i\in I$.
\end{cor}

If $Q^{\{i\}}_{x,z}$ and $\hat Q^{\{i\}}_x$ are
minimizers in the definition of 
$d_{\eta_i}(\gamma^{\{i\}}_x,\gamma^{\{i\}}_z)$ and 
$d_{\eta_i}(\gamma^{\{i\}}_x,\tilde\gamma^{\{i\}}_x)$, respectively, 
and if we let $\mathcal{J}=\mathcal{J}_{\rm s}$ and $w_{\{i\}}=1$ for 
all $i\in I$, then Corollary \ref{cor:follmer} follows immediately
from Theorem \ref{thm:main}.  For simplicity, we have imposed the mild
topological regularity assumption on $\bbS^i$ and $\eta_i$ to ensure the 
existence of minimizers \cite[Theorem 4.1]{Vil09} (when minimizers do not 
exist, it is possible with some more work to obtain a similar result by 
using near-optimal couplings in Theorem \ref{thm:main}).
Let us note that when $\eta_i(x,z)=\mathbf{1}_{x\ne z}$ is the trivial 
metric, the Wasserstein distance reduces to the total variation distance
$$
	d_\eta(\mu,\nu) = \frac{1}{2}\|\mu-\nu\| :=
	\frac{1}{2}\sup_{f:\|f\|\le 1}|\mu f-\nu f|
	\quad\mbox{when}
	\quad\eta(x,z)=\mathbf{1}_{x\ne z},
$$
and an optimal coupling exists in any measurable space \cite[p.\ 472]{Dob70}.
Thus in this case no regularity assumptions are needed, and
Corollary \ref{cor:follmer} reduces to the textbook version 
of the comparison theorem that appears, e.g., in \cite[Theorem 
8.20]{Geo11} or \cite[Theorem V.2.2]{Sim93}.  

While the classical comparison theorem of Corollary \ref{cor:follmer} 
follows from our main result, it should be emphasized that the single site 
assumption $\mathcal{J}=\mathcal{J}_{\rm s}$ is a significant restriction. 
The general statement of Theorem \ref{thm:main} constitutes a crucial 
improvement that substantially extends the range of applicability of the 
comparison method, as we will see below and in section 
\ref{sec:appl}.  Let us also note that the proofs in \cite{Dob70,Fol82}, 
based on the ``method of estimates,'' do not appear to extend easily 
beyond the single site setting.  We use a different (though related) method 
of proof that systematically exploits the connection with Markov chains.

\subsection{Alternative assumptions}
\label{sec:altcond}

The key assumption of Theorem \ref{thm:main} is (\ref{eq:follmer}).  
The aim of the present section is to obtain a number of useful 
alternatives to assumption (\ref{eq:follmer}) that are easily verified in 
practice.

We begin by defining the notion of a tempered measure \cite[Remark 2.17]{Fol88}.

\begin{defn}
A probability measure $\mu$ on $\bbS$ is called \emph{$x^\star$-tempered} 
if
$$
	\sup_{i\in I}\int\mu(dx)\,\eta_i(x_i,x^\star_i)<\infty.
$$
In the sequel $x^\star\in\bbS$ will be considered fixed and $\mu$ will
be called \emph{tempered}.
\end{defn}

It is often the case in practice that the collection of metrics is 
uniformly bounded, that is, $\sup_i\|\eta_i\|<\infty$.  In this case, every 
probability measure on $\bbS$ is trivially tempered.  However, the 
restriction to tempered measures may be essential when the spaces $\bbS^i$ 
are noncompact (see, for example, \cite[section 5]{Dob70} for a simple but 
illuminating example).

Let us recall that a norm $\|\cdot\|$ defined on an algebra of square 
(possibly infinite) matrices is called a \emph{matrix norm} if 
$\|AB\|\le\|A\|\,\|B\|$.  We also recall that the matrix norms 
$\|\cdot\|_\infty$ and $\|\cdot\|_1$ are defined for nonnegative matrices 
$A=(A_{ij})_{i,j\in I}$ as
$$
	\|A\|_\infty := \sup_{i\in I}\sum_{j\in I} A_{ij},\qquad\qquad
	\|A\|_1 := \sup_{j\in I}\sum_{i\in I} A_{ij}.
$$
The following result collects various useful alternatives to 
(\ref{eq:follmer}).  It is proved in section \ref{sec:coruniq}.

\begin{cor}
\label{cor:uniq}
Suppose that $\rho$ and $\tilde\rho$ are tempered.
Then the conclusion of Theorem \ref{thm:main} remains valid when the 
assumption (\ref{eq:follmer}) is replaced by one of the following:
\begin{enumerate}
\item $\card I<\infty$ and $D<\infty$.
\item $\card I<\infty$, $R<\infty$,
and $\|(W^{-1}R)^n\|<1$ for some matrix norm
$\|\cdot\|$ and $n\ge 1$.
\item $\sup_iW_{ii}<\infty$ and $\|W^{-1}R\|_\infty<1$.
\item $\sup_iW_{ii}<\infty$, $\|RW^{-1}\|_\infty<\infty$, and
$\|(RW^{-1})^n\|_\infty<1$ for some $n\ge 1$.
\item $\sup_iW_{ii}<\infty$, $\sum_i\|\eta_i\|<\infty$, and
$\|RW^{-1}\|_1<1$.
\item $\sup_iW_{ii}<\infty$,
there exists a metric $m$ on $I$ such that 
$\sup\{m(i,j):R_{ij}>0\}<\infty$ and $\sup_i\sum_j 
e^{-\beta m(i,j)}<\infty$ for all $\beta>0$, and $\|RW^{-1}\|_1<1$.
\end{enumerate}
\end{cor}

The conditions of Corollary \ref{cor:uniq} are closely related to the 
uniqueness problem for Gibbs measures.  Suppose that the collection of 
quasilocal transition kernels $(\gamma^J)_{J\in\mathcal{J}}$ is a local 
update rule for $\rho$. It is natural to ask whether $\rho$ is the 
\emph{unique} measure that admits $(\gamma^J)_{J\in\mathcal{J}}$ as a 
local update rule (see the remark at the end of section 
\ref{sec:setting}). We now observe that uniqueness is a necessary 
condition for the conclusion of Theorem \ref{thm:main}.  Indeed, let 
$\tilde\rho$ be another measure that admits the same local update rule. If 
(\ref{eq:follmer}) holds, we can apply Theorem \ref{thm:main} with 
$\tilde\gamma^J=\gamma^J$ and $a_j=0$ to conclude that $\tilde\rho=\rho$. 
In particular, $\sum_j(I-W+R)^n_{ij}\to 0$ in Theorem \ref{thm:main}
evidently implies uniqueness in the class of tempered measures.

Of course, the point of Theorem \ref{thm:main} is that it provides a 
quantitative tool that goes far beyond qualitative uniqueness questions.  
It is therefore interesting to note that this single result nonetheless 
captures many of the uniqueness conditions that are used in the literature.  
In Corollary \ref{cor:uniq}, Condition 3 is precisely the ``influence on a 
site'' condition of Weitz \cite[Theorem 2.5]{Wei05} (our setting is even 
more general in that we do not require bounded-range interactions as is 
essential in \cite{Wei05}). Conditions 5 and 6 constitute a slight 
strengthening (see below) of the ``influence of a site'' condition of 
Weitz \cite[Theorem 2.7]{Wei05} under summable metric or subexponential 
graph assumptions, in the spirit of the classical uniqueness condition of 
Dobrushin and Shlosman \cite{DS85}.  In the finite setting with single site 
updates, Condition 2 is in the spirit of \cite{DGJ09} and Condition 4 
is in the spirit of \cite{DGJ08}.

On the other hand, we can now see that Theorem \ref{thm:main} provides a 
crucial improvement over the classical comparison theorem.  The single 
site setting of Corollary \ref{cor:follmer} corresponds essentially to the 
original Dobrushin uniqueness regime \cite{Dob70}.  It is well known that 
this setting is restrictive, in that it captures only a small part of the 
parameter space where uniqueness of Gibbs measures holds.  It is precisely 
for this reason that Dobrushin and Shlosman introduced their improved 
uniqueness criterion in terms of larger blocks \cite{DS85}, which in many 
cases allows to capture a large part of or even the entire uniqueness 
region; see \cite[section 5]{Wei05} for examples.  The generalized 
comparison Theorem \ref{thm:main} in terms of larger blocks can therefore 
be fruitfully applied to a much larger and more natural class of models 
than the classical comparison theorem.  This point will be further 
emphasized in the context of the application that will be developed in 
detail in section \ref{sec:appl}.

\begin{remark}
The ``influence of a site'' condition $\|RW^{-1}\|_1<1$ that appears in 
Corollary \ref{cor:uniq} is slightly stronger than the corresponding 
condition of Dobrushin-Shlosman \cite{DS85} and Weitz \cite[Theorem 
2.7]{Wei05}. Writing out the definition of $R$, we find that our condition 
reads
$$
	\|RW^{-1}\|_1 =
	\sup_{j\in I}W_{jj}^{-1}\sum_{i\in I} 
	\sup_{
	\substack{x,z\in\bbS : \\ x^{I\backslash\{j\}} = z^{I\backslash \{j\}}}
	} \frac{1}{\eta_j(x_j,z_j)}\sum_{J\in\mathcal{J}:i\in J}
	w_J\, Q^J_{x,z}\eta_i<1,
$$
while the condition of \cite[Theorem 2.7]{Wei05} (which extends the 
condition of \cite{DS85}) reads
$$
	\sup_{j\in I}W_{jj}^{-1}
	\sup_{
	\substack{x,z\in\bbS : \\ x^{I\backslash\{j\}} = z^{I\backslash \{j\}}}
	} \frac{1}{\eta_j(x_j,z_j)}
	\sum_{i\in I} 
	\sum_{J\in\mathcal{J}:i\in J}
	w_J\, Q^J_{x,z}\eta_i<1.
$$
The latter is slightly weaker as the sum over sites $i$ appears 
inside the supremum over configurations $x,z$.  While the distinction 
between these conditions is inessential in many applications, 
there do exist situations in which the weaker condition yields an 
essential improvement, see, e.g., \cite[section 5.3]{Wei05}.  In such 
problems, Theorem \ref{thm:main} is not only limited by the stronger
uniqueness condition but could also lead to poor quantitative bounds, as 
the comparison bound is itself expressed in terms of the uniform influence 
coefficients $R_{ij}$.

It could therefore be of interest to develop comparison theorems that are 
able to exploit the finer structure that is present in the weaker 
uniqueness condition.  In fact, the proof of Theorem \ref{thm:main} 
already indicates a natural approach to such improved bounds.  However, 
the resulting comparison theorems are necessarily \emph{nonlinear} in that 
the action of the matrix $R$ is replaced by a nonlinear operator 
$\mathsf{R}$. The nonlinear expressions are somewhat difficult to handle 
in practice, and as we do not at present have a compelling application for 
such bounds we do not pursue this direction here.  However, for 
completeness, we will briefly sketch at the end of section \ref{sec:step2} 
how such bounds can be obtained. 
\end{remark}

\subsection{A one-sided comparison theorem}

As was discussed in section \ref{sec:setting}, it is natural in many 
applications to describe high-dimensional probability distributions in 
terms of local conditional probabilities of the form 
$\mu(dz^J|x^{I\backslash J})$.  This is in essence a static picture, where 
we describe the behavior of each local region $J$ given that the 
configuration of the remaining sites $I\backslash J$ is frozen.  In models 
that possess dynamics, this description is not very natural. In this 
setting, each site $i\in I$ occurs at a given time $\tau(i)$, and its 
state is only determined by the configuration of sites $j\in I$ in the 
past and present $\tau(j)\le\tau(i)$, but not by the future.  For example, 
the model might be defined as a high-dimensional Markov chain whose 
description is naturally given in terms of one-sided conditional 
probabilities (see, e.g., \cite{Fol79} and the application in section 
\ref{sec:appl}). It is therefore interesting to note that the original 
comparison theorem of Dobrushin \cite{Dob70} is actually more general than 
Corollary \ref{cor:follmer} in that it is applicable both in the static 
and dynamic settings.  We presently develop an analogous generalization to 
Theorem \ref{thm:main}.

For the purposes of this section, we assume that we are given a function 
$\tau:I\to\mathbb{Z}$ that assigns to each site $i\in I$ an integer index 
$\tau(i)$.  We define
$$
	I_{\le k}:=\{i\in I:\tau(i)\le k\},\qquad\qquad
	\bbS_{\le k}:=\bbS^{I_{\le k}},
$$
and for any probability measure 
$\rho$ on $\bbS$ we denote by $\rho_{\le k}$ the marginal distribution on 
$\bbS_{\le k}$.

\begin{defn}
\label{defn:onesidelocalupd}
A \emph{one-sided local update rule} for $\rho$ is a collection 
$(\gamma^J)_{J\in\mathcal{J}}$ where
\begin{enumerate}
\item $\mathcal{J}$ is a cover of $I$ such that $\min_{i\in J}\tau(i)=
\max_{i\in J}\tau(i)=:\tau(J)$ for every $J\in\mathcal{J}$.
\item $\gamma^J$ is a transition kernel from $\bbS_{\le\tau(J)}$ to 
$\bbS^J$.
\item $\rho_{\le\tau(J)}$ is $\gamma^J$-invariant for every $J\in\mathcal{J}$.
\end{enumerate}
\end{defn}

The canonical example of a one-sided local update rule is to consider the 
one-sided conditional distributions 
$\gamma^J_x(dz^J)=\rho(dz^J|x^{I_{\le\tau(J)}\backslash J})$.  This 
situation is particularly useful in the investigation of interacting 
Markov chains, cf.\ \cite{Dob70,Fol79}, where $\tau(j)$ denotes the time 
index of the site $j$ and we condition only on the past and present, but 
not on the future.

\begin{defn}
\label{defn:onesidecoupledupd}
A \emph{one-sided coupled update rule} for 
$(\rho,\tilde\rho)$ is a collection of transition kernels 
$(\gamma^J,\tilde\gamma^J,Q^J,
\hat Q^J)_{J\in\mathcal{J}}$ such that the following hold:
\begin{enumerate}
\item $(\gamma^J)_{J\in\mathcal{J}}$ and $(\tilde\gamma^J)_{J\in\mathcal{J}}$
are one-sided local update rules for $\rho$ and $\tilde\rho$, 
respectively.
\item $Q^J_{x,z}$ is a coupling of $\gamma^J_x,\gamma^J_z$ for
$J\in\mathcal{J}$ and $x,z\in\bbS_{\le\tau(J)}$ with $\card\{i:x_i\ne 
z_i\}=1$.
\item $\hat Q^J_x$ is a coupling of $\gamma^J_x,\tilde\gamma^J_x$ for
$J\in\mathcal{J}$ and $x\in\bbS_{\le\tau(J)}$.
\end{enumerate}
\end{defn}

We can now state a one-sided counterpart to Theorem \ref{thm:main}.

\begin{theorem}
\label{thm:oneside}
Let $(\gamma^J,\tilde\gamma^J,Q^J,\hat Q^J)_{J\in\mathcal{J}}$ be a 
one-sided coupled update rule for $(\rho,\tilde\rho)$, and let
$(w_J)_{J\in\mathcal{J}}$ be a family of strictly positive weights.
Define the matrices $W$ and $R$ and the vector $a$ as in Theorem 
\ref{thm:main}.  Assume that $\gamma^J$ is quasilocal for every 
$J\in\mathcal{J}$, that
\begin{equation}
\label{eq:oneside}
	\sum_{j\in I}D_{ij}\,
	(\rho\otimes\tilde\rho)\eta_j<\infty
	\quad\mbox{for all }i\in I
	\qquad\mbox{where}\qquad
	D:=\sum_{n=0}^\infty (W^{-1}R)^n,
\end{equation}
and that (\ref{eq:follmer}) holds.  Then we have
$$
	|\rho f-\tilde\rho f| \le
	\sum_{i,j\in I}\delta_if\,D_{ij}\,W_{jj}^{-1}a_j
$$
for any bounded and measurable quasilocal function $f$ such that
$\delta_if<\infty$ for all $i\in I$.
\end{theorem}

Let us remark that the result of Theorem \ref{thm:oneside} is formally the 
same as that of Theorem \ref{thm:main}, except that we have changed the 
nature of the update rules used in the definition of the coefficients.  
We also require a further assumption (\ref{eq:oneside}) in addition
to assumption (\ref{eq:follmer}) of Theorem \ref{thm:main}, but this is 
not restrictive in practice: in particular, it is readily verified that 
the conclusion of Theorem \ref{thm:oneside} also holds under any of the 
conditions of Corollary \ref{cor:uniq}.

\section{Proofs}
\label{sec:proofs}

\subsection{General comparison principle}
\label{sec:step1}

The proof of Theorem \ref{thm:main} is derived from a general comparison 
principle for Markov chains that will be formalized in this section.  The 
basic idea behind this principle is to consider two transition kernels $G$ 
and $\tilde G$ on $\bbS$ such that $\rho G=G$ and $\tilde\rho\tilde G= 
\tilde\rho$.  One should think of $G$ as the transition kernel of a Markov 
chain that admits $\rho$ as its invariant measure, and similarly for 
$\tilde G$.  The comparison principle of this section provides a general 
method to bound the difference between the invariant measures $\rho$ and 
$\tilde\rho$ in terms of the transition kernels $G$ and $\tilde G$.  In 
the following sections, we will apply this principle to a specific 
choice of $G$ and $\tilde G$ that is derived from the coupled update rule.

We begin by introducing a standard notion in the analysis of 
high-dimensional Markov chains, cf.\ \cite{Fol79} (note that our indices 
are reversed as compared to the definition in \cite{Fol79}).

\begin{defn}
\label{def:wasserstein}
$(V_{ij})_{i,j\in I}$ is called a \emph{Wasserstein matrix} for a 
transition kernel $G$ on $\bbS$ if
$$
	\delta_j Gf \le \sum_{i\in I}\delta_if\,V_{ij}
$$
for every $j\in I$ and bounded and measurable quasilocal function $f$.
\end{defn}

We now state our general comparison principle.  

\begin{prop}
\label{prop:markovcomp}
Let $G$ and $\tilde G$ be transition kernels on $\bbS$ such that
$\rho G=\rho$ and $\tilde\rho\tilde G=\tilde\rho$, and let $Q_x$ be a 
coupling between the measures $G_x$ and $\tilde G_x$ for every $x\in\bbS$.
Assume that $G$ is quasilocal, and let $V$ be a Wasserstein matrix for 
$G$.  Then we have
$$
	|\rho f-\tilde\rho f| \le
	\sum_{i,j\in I}\delta_if\,N^{(n)}_{ij}
	\outint \tilde\rho(dx)\,Q_x\eta_j +
	\sum_{i,j\in I}\delta_if\,V^n_{ij}\,(\rho\otimes\tilde\rho)\eta_j,
$$
where we defined
$$
	N^{(n)}:=\sum_{k=0}^{n-1} V^k,
$$
for any bounded and measurable quasilocal function $f$ and $n\ge 1$.
\end{prop}

Theorem \ref{thm:main} will be derived from this result.  Roughly 
speaking, we will design the transition kernel $G$ such that $V=I-W+R$ is 
a Wasserstein matrix; then assumption (\ref{eq:follmer}) implies that the 
second term in Proposition \ref{prop:markovcomp} vanishes as $n\to\infty$, 
and the result of Theorem \ref{thm:main} reduces to some matrix algebra 
(as will be explained below, however, a more complicated argument 
is needed to obtain Theorem \ref{thm:main} in full generality).

To prove Proposition \ref{prop:markovcomp} we require a simple lemma.

\begin{lem}
\label{lem:coupling}
Let $Q$ be a coupling of probability measures $\mu,\nu$ on $\bbS$.  Then
$$
	|\mu f-\nu f|\le \sum_{i\in I}\delta_if\,Q\eta_i
$$
for every bounded and measurable quasilocal function $f$.
\end{lem}

\begin{proof}
Let $J\in\mathcal{I}$.  Enumerate its elements arbitrarily as
$J=\{j_1,\ldots,j_r\}$, and define $J_k=\{j_1,\ldots,j_k\}$ for
$1\le k\le r$ and $J_0=\varnothing$. Then we can evidently estimate
$$
	|f_x^J(z)-f_x^J(\tilde z)| \le
	\sum_{k=1}^r|f_x^J(z^{J_k}\tilde z^{J\backslash J_k})-
	f_x^J(z^{J_{k-1}}\tilde z^{J\backslash J_{k-1}})| \le
	\sum_{j\in J}\delta_jf\,\eta_j(z_j,\tilde z_j).
$$
As $f$ is quasilocal, we can let $J\uparrow I$ to obtain
$$
	|f(z)-f(\tilde z)| \le \sum_{i\in I}\delta_if\,\eta_i(z_i,\tilde z_i).
$$
The result follows readily as $|\mu f-\nu f|\le \int
|f(z)-f(\tilde z)|\,Q(dz,d\tilde z)$.
\end{proof}

We now proceed to the proof of Proposition \ref{prop:markovcomp}.

\begin{proof}[Proof of Proposition \ref{prop:markovcomp}]
We begin by writing
\begin{align*}
	|\rho f-\tilde\rho f| &=
	|\rho G^nf -\tilde\rho\tilde G^nf| \\
	&\le
	\sum_{k=0}^{n-1}|\tilde\rho\tilde G^{n-k-1}G^{k+1}f-
	\tilde\rho\tilde G^{n-k}G^kf| 
	+ |\rho G^nf-\tilde\rho G^nf| \\
	&=
	\sum_{k=0}^{n-1}|\tilde\rho G G^{k}f-
	\tilde\rho\tilde G G^kf| 
	+ |\rho G^nf-\tilde\rho G^nf|.
\end{align*}
As $G$ is assumed quasilocal, $G^kf$ is quasilocal, and thus
Lemma \ref{lem:coupling} yields
\begin{align*}
	|\tilde\rho G G^{k}f-\tilde\rho\tilde G G^kf|
	&\le
	\int\tilde\rho(dx)\,|G_x G^{k}f-\tilde G_x G^kf| \\
	&\le
	\outint\tilde\rho(dx)
	\sum_{j\in I}\delta_jG^kf\,Q_x\eta_j\\
	&\le
	\sum_{i,j\in I}\delta_if\,V^k_{ij}\outint\tilde\rho(dx)\,Q_x\eta_j.
\end{align*}
Similarly, as $\rho\otimes\tilde\rho$ is a coupling of $\rho,\tilde\rho$, 
we obtain by Lemma \ref{lem:coupling} 
$$
	|\rho G^nf-\tilde\rho G^nf| \le
	\sum_{j\in I}\delta_jG^nf\,(\rho\otimes\tilde\rho)\eta_j \le
	\sum_{i,j\in I}\delta_if\,V^n_{ij}\,(\rho\otimes\tilde\rho)\eta_j.
$$
Thus the proof is complete.
\end{proof}

\subsection{Gibbs samplers}
\label{sec:step2}

To put Proposition \ref{prop:markovcomp} to good use, we must construct 
transition kernels $G$ and $\tilde G$ for which $\rho$ and $\tilde\rho$ 
are invariant, and that admit tractable estimates for the quantities in 
the comparison theorem in terms of the coupled update rule 
$(\gamma^J,\tilde\gamma^J,Q^J,\hat Q^J)_{J\in\mathcal{J}}$ and the weights 
$(w_J)_{J\in\mathcal{J}}$.  To this end, we will use a standard 
construction called the \emph{Gibbs sampler}: in each time step, we draw a 
region $J\in\mathcal{J}$ with probability $v_J\propto w_J$, and then apply 
the transition kernel $\gamma^J$ to the current configuration. This 
readily defines a transition kernel $G$ for which $\rho$ is $G$-invariant 
(as $\rho$ is $\gamma^J$-invariant for every $J\in\mathcal{J}$).  The 
construction for $\tilde G$ is identical.  As will be explained below, 
this is not the most natural construction for the proof of our main 
result; however, it will form the basis for further computations.

We fix throughout this section a coupled update rule
$(\gamma^J,\tilde\gamma^J,Q^J,\hat Q^J)_{J\in\mathcal{J}}$ for 
$(\rho,\tilde\rho)$ and weights $(w_J)_{J\in\mathcal{J}}$ satisfying
the assumptions of Theorem \ref{thm:main}.  Let 
$\mathbf{v}=(v_J)_{J\in\mathcal{J}}$ be a sequence of nonnegative weights 
such that $\sum_Jv_J\le 1$.  We define the Gibbs samplers
\begin{align*}
	G_x^{\mathbf{v}}(A) & :=
	\left(
	1-\sum_{J\in\mathcal{J}}v_J\right)\mathbf{1}_A(x)+
	\sum_{J\in\mathcal{J}}
	v_J \int \mathbf{1}_A(z^Jx^{I\backslash J})\,\gamma^J_x(dz^J),\\
	\tilde G_x^{\mathbf{v}}(A) & :=
	\left(
	1-\sum_{J\in\mathcal{J}}v_J\right)\mathbf{1}_A(x)+
	\sum_{J\in\mathcal{J}}
	v_J \int \mathbf{1}_A(z^Jx^{I\backslash J})\,\tilde\gamma^J_x(dz^J).
\end{align*}
Evidently $G^{\mathbf{v}}$ and $\tilde G^{\mathbf{v}}$ are transition 
kernels on $\bbS$, and $\rho G^{\mathbf{v}}=\rho$ and $\tilde\rho\tilde 
G^{\mathbf{v}}=\tilde\rho$ by construction.  To apply Proposition 
\ref{prop:markovcomp}, we must establish some basic properties.

\begin{lem}
\label{lem:quasilocal}
Assume that $\gamma^J$ is quasilocal for every $J\in\mathcal{J}$.
Then $G^{\mathbf{v}}$ is quasilocal.
\end{lem}

\begin{proof}
Let $f:\bbS\to\bbS$ be a bounded and measurable quasilocal function.
It evidently suffices to show that $\gamma^Jf^J$ is quasilocal for every
$J\in\mathcal{J}$.  To this end, let us fix $J\in\mathcal{J}$, 
$x,z\in\bbS$, and $J_1,J_2,\ldots\in\mathcal{I}$ such that $J_1\subseteq 
J_2\subseteq\cdots$ and $\bigcup_iJ_i=I$.  Then we have
$$
	\gamma^J_{z^{J_i}x^{I\backslash J_i}}\xrightarrow{i\to\infty}\gamma^J_z
	\quad\mbox{setwise}
$$
as $\gamma^J$ is quasilocal. On the other hand, we have
$$
	f^J_{z^{J_i}x^{I\backslash J_i}}\xrightarrow{i\to\infty}f^J_z
	\quad\mbox{pointwise}
$$
as $f$ is quasilocal.
Thus by \cite[Proposition 18, p.\ 270]{Roy88} we obtain
$$
	\gamma^J_{z^{J_i}x^{I\backslash J_i}}f^J_{z^{J_i}x^{I\backslash J_i}}
	\xrightarrow{i\to\infty}
	\gamma^J_{z}f^J_{z}.
$$
As the choice of $x,z$ and $(J_i)_{i\ge 1}$ is arbitrary,
the result follows.
\end{proof}

\begin{lem}
\label{lem:wasserstein}
Assume that $\gamma^J$ is quasilocal for every $J\in\mathcal{J}$, and
define
\begin{align*}
	W_{ij}^{\mathbf{v}}
	 &:= \mathbf{1}_{i=j}\sum_{J\in\mathcal{J}:i\in J}v_J,\\
	R_{ij}^{\mathbf{v}} &:= \sup_{
	\substack{x,z\in\bbS : \\ x^{I\backslash\{j\}} = z^{I\backslash \{j\}}}
	} \frac{1}{\eta_j(x_j,z_j)}\sum_{J\in\mathcal{J}:i\in J}
	v_J\, Q^J_{x,z}\eta_i.
\end{align*}
Then $V^{\mathbf{v}}=I-W^{\mathbf{v}}+R^{\mathbf{v}}$ is a Wasserstein 
matrix for $G^{\mathbf{v}}$.
\end{lem}

\begin{proof}
Let $f:\bbS\to\bbS$ be a bounded and measurable quasilocal function, and 
let $x,z\in\bbS$ be configurations that differ at a single site
$\card\{i\in I:x_i\ne z_i\}=1$.  Note that
$$
	\gamma_x^Jf_x^J = (\gamma_x^J\otimes\delta_{x^{I\backslash J}})f,
	\qquad\quad
	\gamma_z^Jf_z^J = (\gamma_z^J\otimes\delta_{z^{I\backslash J}})f.
$$
As $Q_{x,z}^J$ is a coupling of $\gamma^J_x$ 
and $\gamma^J_z$ by construction, the measure 
$Q_{x,z}^J\otimes\delta_{x^{I\backslash J}}\otimes\delta_{z^{I\backslash 
J}}$ is a coupling of $\gamma_x^J\otimes\delta_{x^{I\backslash J}}$ and
$\gamma_z^J\otimes\delta_{z^{I\backslash J}}$.  Thus Lemma 
\ref{lem:coupling} yields
\begin{align*}
	|\gamma_x^Jf_x^J-\gamma_z^Jf_z^J| 
	&\le
	\sum_{i\in I}\delta_i f\,
	(Q_{x,z}^J\otimes\delta_{x^{I\backslash J}}
	\otimes\delta_{z^{I\backslash J}})\eta_i \\
	&=
	\sum_{i\in J}\delta_i f\,
	Q_{x,z}^J\eta_i +
	\sum_{i\in I\backslash J}\delta_i f\,\eta_i(x_i,z_i).
\end{align*}
In particular, we obtain
\begin{align*}
	&|G^{\mathbf{v}}f(x)-G^{\mathbf{v}}f(z)| \le
	\Bigg(1-\sum_{J\in\mathcal{J}}v_J\Bigg)
	|f(x)-f(z)| +
	\sum_{J\in\mathcal{J}}v_J\,|\gamma_x^Jf_x^J-\gamma_z^Jf_z^J| \\
	&\quad\le
	\Bigg(1-\sum_{J\in\mathcal{J}}v_J\Bigg)
	\sum_{i\in I}\delta_if\,\eta_i(x_i,z_i) 
	+
	\sum_{J\in\mathcal{J}}v_J\Bigg(
	\sum_{i\in J}\delta_i f\,
	Q_{x,z}^J\eta_i +
	\sum_{i\in I\backslash J}\delta_if\,\eta_i(x_i,z_i)
	\Bigg) \\
	&\quad =
	\sum_{i\in I}
	\delta_if\,\{1-W_{ii}^{\mathbf{v}}\}\,\eta_i(x_i,z_i)
	+
	\sum_{i\in I}
	\delta_i f
	\sum_{J\in\mathcal{J}:i\in J}v_J\,Q_{x,z}^J\eta_i.
	\phantom{\Bigg\}}	
\end{align*}
Now suppose that $x^{I\backslash\{j\}}=z^{I\backslash\{j\}}$ (and $x\ne 
z$).  Then by definition
$$
	\sum_{J\in\mathcal{J}:i\in J}v_J\,Q_{x,z}^J\eta_i \le
	R_{ij}^{\mathbf{v}}\,\eta_j(x_j,v_j),	
$$
and we obtain
$$
	\frac{|G^{\mathbf{v}}f(x)-G^{\mathbf{v}}f(z)|}
	{\eta_j(x_j,z_j)} \le
	\delta_jf\,\{1-W_{jj}^{\mathbf{v}}\}
	+
	\sum_{i\in I}
	\delta_i f\,
	R_{ij}^{\mathbf{v}}.
$$
Thus $V^{\mathbf{v}}=I-W^{\mathbf{v}}+R^{\mathbf{v}}$ satisfies
Definition \ref{def:wasserstein}.
\end{proof}

Using Lemmas \ref{lem:quasilocal} and \ref{lem:wasserstein}, we can now 
apply Proposition \ref{prop:markovcomp}.

\begin{cor}
\label{cor:precomp}
Assume that $\gamma^J$ is quasilocal for every $J\in\mathcal{J}$.
Then
$$
	|\rho f-\tilde\rho f| \le
	\sum_{i,j\in I}\delta_if\,N^{\mathbf{v}(n)}_{ij}\,
	a_j^{\mathbf{v}} +
	\sum_{i,j\in I}\delta_if\,(I-W^{\mathbf{v}}+R^{\mathbf{v}})^n_{ij}
	\,(\rho\otimes\tilde\rho)\eta_j
$$
for every $n\ge 1$ and bounded and measurable quasilocal function $f$,
where
$$
	N^{\mathbf{v}(n)}:=\sum_{k=0}^{n-1} 
	(I-W^{\mathbf{v}}+R^{\mathbf{v}})^k
$$
and the coefficients $(a_j^{\mathbf{v}})_{j\in I}$ are defined by
$a_j^{\mathbf{v}} := \sum_{J\in\mathcal{J}:j\in J}v_J\int^*\tilde\rho(dx)\,
\hat Q_x^J\eta_j$.
\end{cor}

\begin{proof}
Let $G=G^{\mathbf{v}}$, $\tilde G=\tilde G^{\mathbf{v}}$, 
$V=I-W^{\mathbf{v}}+R^{\mathbf{v}}$ in Proposition 
\ref{prop:markovcomp}.  The requisite assumptions are verified by Lemmas 
\ref{lem:quasilocal} and \ref{lem:wasserstein}, and
it remains to show that there exists a coupling 
$Q_x$ of $G_x$ and $\tilde G_x$ such that
$\int^*\tilde\rho(dx)\,Q_x\eta_j\le a_j$ for every $j\in I$.  But choosing 
$$
	Q_xg :=
	\Bigg(
	1-\sum_{J\in\mathcal{J}}v_J\Bigg)
	g(x,x)+
	\sum_{J\in\mathcal{J}}
	v_J \int\hat Q_x^J(dz^J,d\tilde z^J)\,
	g(z^Jx^{I\backslash J},\tilde z^Jx^{I\backslash J}),
$$
it is easily verified that $Q_x$ satisfies the necessary properties.
\end{proof}

In order for the construction of the Gibbs sampler to make sense, the 
weights $v_J$ must be probabilities.  This imposes the requirement
$\sum_Jv_J\le 1$.  If we were to assume that $\sum_Jw_J\le 1$, we could 
apply Corollary \ref{cor:precomp} with $v_J=w_J$.  Then assumption 
(\ref{eq:follmer}) guarantees that the second term in Corollary 
\ref{cor:precomp} vanishes as $n\to\infty$, which yields
$$
	|\rho f-\tilde\rho f| \le
	\sum_{i,j\in I}\delta_if\,N_{ij}\,
	a_j\quad\mbox{with}\quad
	N := \sum_{k=0}^\infty (I-W+R)^k.
$$
The proof of Theorem \ref{thm:main} would now be complete after we
establish the identity
$$
	N = \sum_{k=0}^\infty (I-W+R)^k =
	\sum_{k=0}^\infty (W^{-1}R)^k\,W^{-1} = DW^{-1}.
$$
This straightforward matrix identity will be proved in the next 
section.  The assumption that the weights $w_J$ are summable is 
restrictive, however, when $I$ is infinite: in Theorem \ref{thm:main} we 
only assume that $W_{ii}\le 1$ for all $i$, so we evidently cannot set
$v_J=w_J$.

When the weights $w_j$ are not summable, it is not natural to interpret 
them as probabilities.  In this setting, a much more natural construction 
would be to consider a \emph{continuous time} counterpart of the Gibbs 
sampler called \emph{Glauber dynamics}.  To define this process, one 
attaches to each region $J\in\mathcal{J}$ an independent Poisson process 
with rate $w_J$, and applies the transition kernel $\gamma^J$ at every 
jump time of the corresponding Poisson process.  Thus $w_J$ does not 
represent the probability of selecting the region $J$ in one time step, 
but rather the frequency with which region $J$ is selected in continuous 
time.  Once this process has been defined, one would choose the transition 
kernel $G$ to be the transition semigroup of the continuous time process 
on any fixed time interval.  Proceeding with this construction we expect, 
at least formally, to obtain Theorem \ref{thm:main} under the stated 
assumptions.

Unfortunately, there are nontrivial technical issues involved in 
implementing this approach: it is not evident \emph{a priori} that the 
continuous time construction defines a well-behaved Markov semigroup, so 
that it is unclear when the above program can be made rigorous.  The 
existence of a semigroup has typically been established under more 
restrictive assumptions than we have imposed in the present setting 
\cite{Lig05}. In order to circumvent such issues, we will proceed by an 
alternate route.  Formally, the Glauber dynamics can be obtained by an 
appropriate scaling limit of discrete time Gibbs samplers.  We will also 
utilize this scaling, but instead of applying Proposition 
\ref{prop:markovcomp} to the limiting dynamics we will take the scaling 
limit directly in Corollary \ref{cor:precomp}.  Thus, while our intuition 
comes from the continuous time setting, we avoid some technicalities 
inherent in the construction of the limit dynamics.  Instead, we now face 
the problem of taking limits of powers of infinite matrices.  The 
requisite matrix algebra will be worked out in the following section.

\begin{remark}
Let us briefly sketch how the previous results can be sharpened to obtain 
a \emph{nonlinear} comparison theorem that could lead to sharper bounds in 
some situations.  Assume for simplicity that $\sum_Jw_J\le 1$.  Then 
$V=I-W+R$ is a Wasserstein matrix for $G$ by Lemma \ref{lem:wasserstein}.
Writing out the definitions, this means
$\delta(Gf)\le \delta(f)V$ where
$$
	(\beta V)_j
	=
	\sum_{i\in I}\beta_i
	\sup_{
	\substack{x,z\in\bbS : \\ x^{I\backslash\{j\}} = z^{I\backslash \{j\}}}
	} 
	\left\{
	\mathbf{1}_{i=j}\left(1-\sum_{J:i\in J}w_J\right)
	+
	\frac{1}{\eta_j(x_j,z_j)}\sum_{J:i\in J}
	w_J\, Q^J_{x,z}\eta_i
	\right\}
$$
(here we interpret $\beta=(\beta_i)_{i\in I}$ and 
$\delta(f)=(\delta_if)_{i\in I}$ as row vectors).  However, from the 
proof of Lemma \ref{lem:wasserstein} we even obtain the sharper bound 
$\delta(Gf)\le\mathsf{V}[\delta(f)]$ where
$$
	\mathsf{V}[\beta]_j
	:=
	\sup_{
	\substack{x,z\in\bbS : \\ x^{I\backslash\{j\}} = z^{I\backslash \{j\}}}
	} 
	\sum_{i\in I}\beta_i
	\left\{
	\mathbf{1}_{i=j}\left(1-\sum_{J:i\in J}w_J\right)
	+
	\frac{1}{\eta_j(x_j,z_j)}\sum_{J:i\in J}
	w_J\, Q^J_{x,z}\eta_i
	\right\}
$$
is defined with the supremum over configurations outside the sum.
The nonlinear operator $\mathsf{V}$ can now be used much in the same way 
as the Wasserstein matrix $V$.  In particular, following the identical 
proof as for Proposition \ref{prop:markovcomp}, we immediately obtain
$$
	|\rho f-\tilde\rho f| \le
	\sum_{j\in I}\sum_{k=0}^{n-1}\mathsf{V}^k[\delta(f)]_j
	\outint\tilde\rho(dx)\,Q_x\eta_j +
	\sum_{j\in I}\mathsf{V}^n[\delta(f)]_j\,
	(\rho\otimes\tilde\rho)\eta_j,
$$
where $\mathsf{V}^k$ denotes the $k$th iterate of the nonlinear 
operator $\mathsf{V}$.  Proceeding along these lines, one can develop 
nonlinear comparison theorems under Dobrushin-Shlosman type conditions 
(see the discussion in section \ref{sec:altcond}).  The nonlinear 
expressions are somewhat difficult to handle, however, and we do not 
develop this idea further in this paper.
\end{remark}

\subsection{Proof of Theorem \ref{thm:main}}
\label{sec:step3}

Throughout this section, we work under the assumptions of Theorem 
\ref{thm:main}.  The main idea of the proof is the following continuous
scaling limit of Corollary \ref{cor:precomp}.

\begin{prop}
\label{prop:contcomp}
Let $t>0$.  Define the matrices
$$
	N := 
	\sum_{k=0}^\infty 
	(I-W+R)^k, \qquad\quad
	V^{[t]} :=
	\sum_{k=0}^\infty
	\frac{t^ke^{-t}}{k!}\,
	(I-W+R)^k.
$$
Then we have, under the assumptions of Theorem \ref{thm:main},
$$
	|\rho f-\tilde\rho f| \le
	\sum_{i,j\in I}\delta_if\,N_{ij}\,
	a_j +
	\sum_{i,j\in I}\delta_if\,V^{[t]}_{ij}
	\,(\rho\otimes\tilde\rho)\eta_j
$$
for every bounded and measurable quasilocal function $f$ such 
that $\delta_if<\infty$ for all $i\in I$.
\end{prop}

\begin{proof} 
Without loss of generality, we will assume throughout the proof that $f$ 
is a local function (so that only finitely many $\delta_if$ are nonzero). 
The extension to quasilocal $f$ follows readily by applying the local 
result to $f_x^J$ and letting $J\uparrow I$ as in the proof of Lemma 
\ref{lem:coupling}.

As the cover $\mathcal{J}$ is at most countable (because
$\mathcal{I}$ is countable), we can enumerate its elements
arbitrarily as $\mathcal{J}=\{J_1,J_2,\ldots\}$. 
Define the weights 
$\mathbf{v}^{r}=(v^{r}_J)_{J\in\mathcal{J}}$ as
$$
	v^{r}_J :=
	\left\{
	\begin{array}{ll}
	w_J & \mbox{when }J=J_k\mbox{ for }k\le r,\\
	0 & \mbox{otherwise}.
	\end{array}
	\right.
$$
For every $r\in\mathbb{N}$, the weight vector $u\mathbf{v}^r$ evidently 
satisfies $\sum_J uv^r_J\le 1$ for all $u>0$ sufficiently small (depending 
on $r$).  The main idea of the proof is to apply Corollary 
\ref{cor:precomp} to the weight vector $\mathbf{v}=(t/n)\mathbf{v}^r$, 
then let $n\to\infty$, and finally $r\to\infty$.

Let us begin by considering the second term in Corollary 
\ref{cor:precomp}.  We can write
\begin{align*}
	(I-W^{(t/n)\mathbf{v}^r}+R^{(t/n)\mathbf{v}^r})^n &=
	\bigg(\bigg(1-\frac{t}{n}\bigg)I+
	\frac{t}{n}(I-W^{\mathbf{v}^r}+R^{\mathbf{v}^r})\bigg)^n \\
	&=
	\sum_{k=0}^n {n\choose k}\bigg(1-\frac{t}{n}\bigg)^{n-k}
	\bigg(\frac{t}{n}\bigg)^k\,
	(I-W^{\mathbf{v}^r}+R^{\mathbf{v}^r})^k \\
	&= \mathbf{E}[(I-W^{\mathbf{v}^r}+R^{\mathbf{v}^r})^{Z_n}],
	\phantom{\sum^n}
\end{align*}
where we defined the Binomial random variables
$Z_n\sim\mathrm{Bin}(n,t/n)$.  The random variables $Z_n$ converge weakly 
as $n\to\infty$ to the Poisson random variable $Z_\infty\sim\mathrm{Pois}(t)$.
To take the limit of the above expectation, we need a simple estimate that 
will be useful in the sequel.

\begin{lem}
\label{lem:exponentialest}
Let $(c_j)_{j\in I}$ be any nonnegative vector.  Then
$$
	\sum_{j\in I}(I-W^{\mathbf{v}^r}+R^{\mathbf{v}^r})^k_{ij}\,c_j 
	\le 2^k\max_{0\le\ell\le k}
	\sum_{j\in I}(I-W+R)^\ell_{ij}\,c_j
$$
for every $i\in I$ and $k\ge 0$.
\end{lem}

\begin{proof}
As $R^{\mathbf{v}}$ is nondecreasing in $\mathbf{v}$ we obtain the
elementwise estimate
$$
	I-W^{\mathbf{v}^r}+R^{\mathbf{v}^r} \le
	I+R \le I + (I-W+R),
$$
where we have used $W_{ii}\le 1$.  We therefore have
$$
	\sum_{j\in I}(I-W^{\mathbf{v}^r}+R^{\mathbf{v}^r})^k_{ij}\,c_j 
	\le \sum_{j\in I}(I+\{I-W+R\})^k_{ij}\,c_j =
	\sum_{\ell=0}^k {k\choose\ell}
	\sum_{j\in I}(I-W+R)^\ell_{ij}\,c_j,
$$
and the proof is easily completed.
\end{proof}

Define the random variables
$$
	X_n = g(Z_n)\quad\mbox{with}\quad
	g(k) = 
	\sum_{i,j\in I}\delta_if\,
	(I-W^{\mathbf{v}^r}+R^{\mathbf{v}^r})^{k}_{ij}\,
	(\rho\otimes\tilde\rho)\eta_j.
$$
Then $X_n\to X_\infty$ weakly by the continuous mapping theorem.
On the other hand, applying Lemma \ref{lem:exponentialest} with 
$c_j=(\rho\otimes\tilde\rho)\eta_j$ we estimate $g(k)\le C2^k$
for some finite constant $C<\infty$ and all $k\ge 0$, where we have used 
assumption (\ref{eq:follmer}) and that $f$ is local.  As
$$
	\limsup_{u\to\infty}\sup_{n\ge 1}
	\mathbf{E}[2^{Z_n}\mathbf{1}_{2^{Z_n}\ge u}] \le
	\lim_{u\to\infty}u^{-1}\sup_{n\ge 1}\mathbf{E}[4^{Z_n}]  =
	\lim_{u\to\infty}u^{-1}e^{3t}=0,
$$
it follows that the random variables $(X_n)_{n\ge 1}$ are uniformly 
integrable.  We therefore conclude that 
$\mathbf{E}[X_n]\to\mathbf{E}[X_\infty]$ as $n\to\infty$
(cf.\ \cite[Lemma 4.11]{Kal02}).  In particular,
$$
	\lim_{n\to\infty}
	\sum_{i,j\in I}\delta_if\,
	(I-W^{(t/n)\mathbf{v}^r}+R^{(t/n)\mathbf{v}^r})^{n}_{ij}\,
	(\rho\otimes\tilde\rho)\eta_j =
	\sum_{i,j\in I}\delta_if\,
	V^{r[t]}_{ij}\,
	(\rho\otimes\tilde\rho)\eta_j,
$$
where
$$
	V^{r[t]} = 
	\sum_{k=0}^\infty\frac{t^ke^{-t}}{k!}
	(I-W^{\mathbf{v}^r}+R^{\mathbf{v}^r})^{k}.
$$
We now let $r\to\infty$.  Note that $W^{\mathbf{v}^r}\uparrow W$ 
and $R^{\mathbf{v}^r}\uparrow R$ elementwise and, arguing as in the 
proof of Lemma \ref{lem:exponentialest}, we have 
$I-W^{\mathbf{v}^r}+R^{\mathbf{v}^r}\le I+(I-W+R)$ elementwise where
$$
	\sum_{k=0}^\infty
	\sum_{i,j\in I}
	\frac{t^ke^{-t}}{k!}\,
	\delta_if\,
	\{I+(I-W+R)\}^{k}_{ij}\,
	(\rho\otimes\tilde\rho)\eta_j \le
	e^t
	\sup_{\ell\ge 0}
	\sum_{i,j\in I}\delta_if\,
	(I-W+R)^{\ell}_{ij}\,
	(\rho\otimes\tilde\rho)\eta_j
$$
is finite by assumption (\ref{eq:follmer}) and as $f$ is local.
We therefore obtain
$$
	\lim_{r\to\infty}
	\lim_{n\to\infty}
	\sum_{i,j\in I}\delta_if\,
	(I-W^{(t/n)\mathbf{v}^r}+R^{(t/n)\mathbf{v}^r})^{n}_{ij}\,
	(\rho\otimes\tilde\rho)\eta_j =
	\sum_{i,j\in I}\delta_if\,
	V^{[t]}_{ij}\,
	(\rho\otimes\tilde\rho)\eta_j
$$
by dominated convergence.
That is, the second term in Corollary \ref{cor:precomp} with 
$\mathbf{v}=(t/n)\mathbf{v}^r$ converges as $n\to\infty$ and $r\to\infty$ 
to the second term in statement of the present result.

It remains to establish the corresponding conclusion for the first term in
Corollary \ref{cor:precomp}, which proceeds much along the same lines.
We begin by noting that
\begin{align*}
	&\frac{1}{n}\sum_{k=0}^{n-1}
	(I-W^{(t/n)\mathbf{v}^r}+R^{(t/n)\mathbf{v}^r})^k =
	\frac{1}{n}\sum_{k=0}^{n-1}
	\bigg(\bigg(1-\frac{t}{n}\bigg)I+
	\frac{t}{n}(I-W^{\mathbf{v}^r}+R^{\mathbf{v}^r})\bigg)^k \\
	&\qquad=
	\frac{1}{n}
	\sum_{k=0}^{n-1}
	\sum_{\ell=0}^k {k\choose\ell}\bigg(1-\frac{t}{n}\bigg)^{k-\ell}
	\bigg(\frac{t}{n}\bigg)^\ell\,
	(I-W^{\mathbf{v}^r}+R^{\mathbf{v}^r})^\ell \\
	&\qquad=
	\sum_{\ell=0}^{n-1}
	p_\ell^{(n)}
	(I-W^{\mathbf{v}^r}+R^{\mathbf{v}^r})^\ell,
\end{align*}
where we have defined
$$
	p_\ell^{(n)} = 
	\frac{1}{n}
	\sum_{k=\ell}^{n-1}
	{k\choose\ell}\bigg(1-\frac{t}{n}\bigg)^{k-\ell}
	\bigg(\frac{t}{n}\bigg)^\ell =
	\frac{1}{t}
	\int_{\ell t/n}^t
	{\lfloor sn/t\rfloor\choose\ell}
	\bigg(1-\frac{t}{n}\bigg)^{\lfloor sn/t\rfloor-\ell}
	\bigg(\frac{t}{n}\bigg)^\ell\,ds
$$
for $\ell<n$.  An elementary computation yields
$$
	\sum_{\ell=0}^{n-1}p_\ell^{(n)}=1
	\qquad\mbox{and}\qquad
	p_\ell^{(n)}\xrightarrow{n\to\infty}
	p_\ell^{(\infty)} = \frac{1}{t}\int_0^t\frac{s^\ell e^{-s}}{\ell!}\,ds.
$$
We can therefore introduce $\{0,1,\ldots\}$-valued random variables $Y_n$ 
with $\mathbf{P}[Y_n=\ell]=p_\ell^{(n)}$ for $\ell<n$, and we have shown
above that $Y_n\to Y_\infty$ weakly and that
$$
	\frac{1}{n}\sum_{k=0}^{n-1}
	(I-W^{(t/n)\mathbf{v}^r}+R^{(t/n)\mathbf{v}^r})^k =
	\mathbf{E}[(I-W^{\mathbf{v}^r}+R^{\mathbf{v}^r})^{Y_n}].
$$
The first term in Corollary 
\ref{cor:precomp} with $\mathbf{v}=(t/n)\mathbf{v}^r$ can be written as
$$
	\sum_{i,j\in I}\delta_if 
	\sum_{k=0}^{n-1}(I-W^{(t/n)\mathbf{v}^r}+R^{(t/n)\mathbf{v}^r})^k_{ij}\,
	a_j^{(t/n)\mathbf{v}^r} =
	t\,\mathbf{E}[h(Y_n)],
$$
where we have defined
$$
	h(k)=
	\sum_{i,j\in I}\delta_if \,
	(I-W^{\mathbf{v}^r}+R^{\mathbf{v}^r})^{k}_{ij}\,
	a_j^{\mathbf{v}^r}.
$$
We now proceed essentially as above.   We can assume without loss of
generality that
$$
	\sup_{\ell\ge 0}
	\sum_{i,j\in I}\delta_if \,
	(I-W+R)^{\ell}_{ij}\,
	a_j<\infty, 
$$ 
as otherwise the right-hand side in the statement of the present result is 
infinite and the estimate is trivial.  It consequently follows from Lemma 
\ref{lem:exponentialest} that $h(k)\le C2^k$ for some finite constant 
$C<\infty$ and all $k\ge 0$.  A 
similar computation as was done above shows that $(h(Y_n))_{n\ge 0}$ is 
uniformly integrable, and therefore 
$\mathbf{E}[h(Y_n)]\to\mathbf{E}[h(Y_\infty)]$.  In particular, the first 
term in Corollary \ref{cor:precomp} with $\mathbf{v}=(t/n)\mathbf{v}^r$ 
converges as $n\to\infty$ to
$$
	\lim_{n\to\infty}
	\sum_{i,j\in I}\delta_if 
	\sum_{k=0}^{n-1}(I-W^{(t/n)\mathbf{v}^r}+R^{(t/n)\mathbf{v}^r})^k_{ij}\,
	a_j^{(t/n)\mathbf{v}^r} =
	\sum_{i,j\in I}\delta_if \,
	N^r_{ij}\,
	a_j^{\mathbf{v}^r},
$$
where
$$
	N^r = \sum_{k=0}^\infty \int_0^t \frac{s^ke^{-s}}{k!}\,ds~
	(I-W^{\mathbf{v}^r}+R^{\mathbf{v}^r})^{k}.
$$
Similarly, letting $r\to\infty$ and repeating exactly the arguments used 
above for the second term of Corollary \ref{cor:precomp}, we obtain by 
dominated convergence
$$
	\lim_{r\to\infty}
	\lim_{n\to\infty}
	\sum_{i,j\in I}\delta_if 
	\sum_{k=0}^{n-1}(I-W^{(t/n)\mathbf{v}^r}+R^{(t/n)\mathbf{v}^r})^k_{ij}\,
	a_j^{(t/n)\mathbf{v}^r} =
	\sum_{i,j\in I}\delta_if \,
	\tilde N_{ij}\,
	a_j,
$$
where
$$
	\tilde N = \sum_{k=0}^\infty \int_0^t \frac{s^ke^{-s}}{k!}\,ds~
	(I-W+R)^{k}.
$$
To conclude, we have shown that applying Corollary
\ref{cor:precomp} to the weight vector $\mathbf{v}=(t/n)\mathbf{v}^r$ and 
taking the limit as $n\to\infty$ and $r\to\infty$, respectively, yields 
the estimate
$$
	|\rho f-\tilde\rho f| \le
	\sum_{i,j\in I}\delta_if\,\tilde N_{ij}\,
	a_j +
	\sum_{i,j\in I}\delta_if\,V^{[t]}_{ij}
	\,(\rho\otimes\tilde\rho)\eta_j.
$$
It remains to note that $t^ke^{-t}/k!$ is the density of a Gamma 
distribution (with shape $k+1$ and scale $1$), so 
$\int_0^ts^ke^{-s}/k!\,ds\le 1$ and thus $\tilde N\le N$
elementwise.
\end{proof}

We can now complete the proof of Theorem \ref{thm:main}.

\begin{proof}[Proof of Theorem \ref{thm:main}]
Once again, we will assume without loss of generality that $f$ 
is a local function (so that only finitely many $\delta_if$ are nonzero). 
The extension to quasilocal $f$ follows readily by localization as in the 
proof of Lemma \ref{lem:coupling}.

We begin by showing that the second term in Proposition 
\ref{prop:contcomp} vanishes as $t\to\infty$.  Indeed, 
for any $n\ge 0$, we can evidently estimate the second term as
\begin{align*}
	&\sum_{k=0}^\infty \frac{t^ke^{-t}}{k!}
	\sum_{i,j\in I}\delta_if\,
	(I-W+R)^k_{ij}
	\,(\rho\otimes\tilde\rho)\eta_j  
	\\ &\qquad
	\mbox{}\le
	\sup_{\ell\ge 0}
	\sum_{i,j\in I}\delta_if\,
	(I-W+R)^\ell_{ij}
	\,(\rho\otimes\tilde\rho)\eta_j 
	~\sum_{k=0}^n \frac{t^ke^{-t}}{k!}
	\\ &\qquad\quad
	\mbox{}+
	\sup_{\ell>n}
	\sum_{i,j\in I}\delta_if\,
	(I-W+R)^\ell_{ij}
	\,(\rho\otimes\tilde\rho)\eta_j.	
\end{align*}
By assumption (\ref{eq:follmer}) and as $f$ is local, the two terms on the 
right vanish as $t\to\infty$ and $n\to\infty$, respectively.
Thus second term in Proposition \ref{prop:contcomp} vanishes as $t\to\infty$. 

We have now proved the estimate
$$
	|\rho f-\tilde\rho f| \le
	\sum_{i,j\in I}\delta_if\,N_{ij}\,
	a_j.
$$
To complete the proof of Theorem \ref{thm:main}, it remains to establish 
the identity $N=DW^{-1}$.  This is an exercise in matrix 
algebra.  By the definition of the matrix product, we have
$$
	(I-W+R)^p = \sum_{k=0}^p
	\sum_{\substack{n_0,\ldots,n_k\ge 0\\n_0+\cdots+n_k=p-k}}
	(I-W)^{n_k}R\cdots (I-W)^{n_1}R(I-W)^{n_0}.
$$
We can therefore write
\begin{align*}
	&\sum_{p=0}^\infty (I-W+R)^p \\
	&\quad\mbox{}= 
	\sum_{k=0}^\infty
	\sum_{n_0,\ldots,n_k\ge 0}
	\sum_{p=0}^\infty
	\mathbf{1}_{n_0+\cdots+n_k=p-k}
	\mathbf{1}_{k\le p}
	(I-W)^{n_k}R\cdots (I-W)^{n_1}R(I-W)^{n_0} \\
	&\quad\mbox{}=
	\sum_{k=0}^\infty
	\sum_{n_0,\ldots,n_k\ge 0}
	(I-W)^{n_k}R\cdots (I-W)^{n_1}R(I-W)^{n_0} \\
	&\quad\mbox{}=
	\sum_{k=0}^\infty (W^{-1}R)^kW^{-1},
\end{align*}
where we have used that $W^{-1}=\sum_{n=0}^\infty (I-W)^n$ as
$W$ is diagonal with $0<W_{ii}\le 1$.
\end{proof}

\subsection{Proof of Corollary \ref{cor:uniq}}
\label{sec:coruniq}

Note that $\sup_iW_{ii}<\infty$ in all parts of Corollary \ref{cor:uniq} 
(either by assumption or as $\card I<\infty$).  Moreover, it is easily 
seen that all parts of Corollary \ref{cor:uniq} as well as the conclusion 
of Theorem \ref{thm:main} are unchanged if all the weights are multiplied 
by the same constant.  We may therefore assume without loss of generality 
that $\sup_iW_{ii}\le 1$.

Next, we note that as $\rho$ and $\tilde\rho$ are tempered, we have
$$
	\sup_{i\in I}\mbox{} (\rho\otimes\tilde\rho)\eta_i \le
	\sup_{i\in I} \rho\,\eta_i(\mbox{}\cdot\mbox{},x^\star_i) +
	\sup_{i\in I} \tilde\rho\,\eta_i(x^\star_i,\mbox{}\cdot\mbox{})
	<\infty
$$
by the triangle inequality.  To verify (\ref{eq:follmer}), it therefore
suffices to show that
\begin{equation}
\label{eq:unifollmer}
	\lim_{k\to\infty}\sum_{j\in I}(I-W+R)^k_{ij}=0
	\quad\mbox{for all }i\in I.
\end{equation}
We now proceed to verify this condition in the different cases of 
Corollary \ref{cor:uniq}.

\begin{proof}[Proof of Corollary \ref{cor:uniq}(1)]
It was shown at the end of the proof of Theorem \ref{thm:main} that
$$
	\sum_{k=0}^\infty (I-W+R)^k = \sum_{k=0}^\infty (W^{-1}R)^kW^{-1}
	= DW^{-1}.
$$
As $W^{-1}$ has finite entries, $D<\infty$ certainly implies that 
$(I-W+R)^k\to 0$ as $k\to\infty$ elementwise.  But this trivially yields 
(\ref{eq:unifollmer}) when $\card I<\infty$.
\end{proof}

\begin{proof}[Proof of Corollary \ref{cor:uniq}(2)]
Note that we can write
$$
	D = \sum_{k=0}^\infty (W^{-1}R)^k =
	\sum_{p=0}^{n-1}(W^{-1}R)^p
	\sum_{k=0}^\infty (W^{-1}R)^{nk}.
$$
Therefore, if $R<\infty$ and $\|(W^{-1}R)^n\|<1$, we can estimate
$$
	\|D\| \le
	\Bigg\|\sum_{p=0}^{n-1}(W^{-1}R)^p\Bigg\|
	\sum_{k=0}^\infty \|(W^{-1}R)^{n}\|^k < \infty.
$$
Thus $D<\infty$ and we conclude by the previous part.
\end{proof}

\begin{proof}[Proof of Corollary \ref{cor:uniq}(3)]
We give a simple probabilistic proof (a more complicated matrix-analytic 
proof could be given along the lines of \cite[Theorem 3.21]{RFS08}).
Let $P=W^{-1}R$. As $\|P\|_\infty<1$, the infinite matrix $P$ is 
substochastic.  Thus $P$ is the transition probability matrix of a killed 
Markov chain $(X_n)_{n\ge 0}$ such that $\mathbf{P}[X_n=j|X_{n-1}=i]=P_{ij}$
and $\mathbf{P}[X_n\mbox{ is dead}| X_{n-1}=i]=1-\sum_jP_{ij}$ (once the 
chain dies, it stays dead).  Denote by $\zeta=\inf\{n:X_n\mbox{ is 
dead}\}$ the killing time of the chain.  Then we obtain
$$
	\mathbf{P}[\zeta>n|X_0=i] = 
	\mathbf{P}[X_n\mbox{ is not dead}|X_0=i] =
	\sum_{j\in I}P^n_{ij} \le
	\|P^n\|_\infty \le \|P\|_\infty^n.
$$
Therefore, as $\|P\|_\infty<1$, we find by letting $n\to\infty$ that 
$\mathbf{P}[\zeta=\infty|X_0=i]=0$.
That is, the chain dies eventually with unit probability for any initial 
condition.

Now define $\tilde P=I-W+R=I-W+WP$.  As $\sup_iW_{ii}\le 1$, the matrix 
$\tilde P$ is also substochastic and corresponds to the following 
transition mechanism.  If $X_{n-1}=i$, then at time $n$ we flip a biased 
coin that comes up heads with probability $W_{ii}$.  In case of heads we 
make a transition according to the matrix $P$, but in case of tails we 
leave the current state unchanged.  From this description, it is evident
that we can construct a Markov chain $(\tilde X_n)_{n\ge 0}$ with
transition matrix $\tilde P$ by modifying the chain $(X_n)_{n\ge 0}$ as 
follows.  Conditionally on $(X_n)_{n\ge 0}$, draw independent random 
variables $(\xi_n)_{n\ge 0}$ such that $\xi_n$ is geometrically 
distributed with parameter $W_{X_nX_n}$.  Now define the process $(\tilde 
X_n)_{n\ge 0}$ such that it stays in state $X_0$ for the first $\xi_0$ 
time steps, then is in state $X_1$ for the next $\xi_1$ time steps, etc.
By construction, the resulting process is Markov with transition 
matrix $\tilde P$.  Moreover, as $\zeta<\infty$ a.s., we have
$\tilde\zeta:=\inf\{n:\tilde X_n\mbox{ is dead}\}<\infty$ a.s.\ also.  Thus
$$
	\lim_{n\to\infty}\sum_{j\in I}(I-W+R)^n_{ij} =
	\lim_{n\to\infty}\mathbf{P}[\tilde\zeta>n|X_0=i] = 0
$$
for every $i\in I$.  We have therefore established (\ref{eq:unifollmer}).
\end{proof}

\begin{proof}[Proof of Corollary \ref{cor:uniq}(4)]
We begin by writing as above
$$
	\sum_{k=0}^\infty (I-W+R)^k = \sum_{k=0}^\infty (W^{-1}R)^kW^{-1} 
	= \sum_{k=0}^\infty W^{-1}(RW^{-1})^k,
$$
where the last identity is straightforward.  Arguing as in 
Corollary \ref{cor:uniq}(2), we obtain
\begin{align*}
	W_{ii}\sum_{k=0}^\infty \sum_{j\in I}(I-W+R)^k_{ij} &=
	\sum_{j\in I}
	\sum_{k=0}^\infty (RW^{-1})^k_{ij} \le
	\Bigg\|
	\sum_{k=0}^\infty (RW^{-1})^k\Bigg\|_\infty \\
	&\le
	\sum_{p=0}^{n-1}\|RW^{-1}\|_\infty^p
	\sum_{k=0}^\infty \|(RW^{-1})^{n}\|_\infty^k < \infty.
\end{align*}
It follows immediately that (\ref{eq:unifollmer}) holds.
\end{proof}

\begin{proof}[Proof of Corollary \ref{cor:uniq}(5)]
Note that
$$
	\sum_{j\in I}(RW^{-1})^k_{ij}\|\eta_j\| \le
	\|(RW^{-1})^k\|_1\sum_{j\in I}\|\eta_j\| \le
	\|RW^{-1}\|_1^k\sum_{j\in I}\|\eta_j\|.
$$
Thus $\sum_j\|\eta_j\|<\infty$ and $\|RW^{-1}\|_1<1$ yield
$$
	\sum_{k=0}^\infty \sum_{j\in I}(I-W+R)^k_{ij}\|\eta_j\|
	= W^{-1}_{ii}\sum_{k=0}^\infty 
	\sum_{j\in I}(RW^{-1})^k_{ij}\|\eta_j\|<\infty,
$$
which evidently implies
$$
	\lim_{k\to\infty}
	\sum_{j\in I}(I-W+R)^k_{ij}(\rho\otimes\tilde\rho)\eta_j=0
	\quad\mbox{for all }i\in I.
$$
We have therefore established (\ref{eq:follmer}).
\end{proof}

\begin{proof}[Proof of Corollary \ref{cor:uniq}(6)]
Let $r=\sup\{m(i,j):R_{ij}>0\}$ (which is finite by assumption), and 
choose $\beta>0$ such that $\|RW^{-1}\|_1<e^{-\beta r}$.  Then we can 
estimate
$$
	\|RW^{-1}\|_{1,\beta m} :=
	\sup_{j\in I}\sum_{i\in I} e^{\beta m(i,j)}(RW^{-1})_{ij}
	\le e^{\beta r}\|RW^{-1}\|_1<1.
$$
As $m$ is a pseudometric, it satisfies the triangle inequality and it is 
therefore easily seen that $\|\cdot\|_{1,\beta m}$ is a matrix norm.  In 
particular, we can estimate
$$
	e^{\beta m(i,j)}(RW^{-1})^n_{ij} \le
	\|(RW^{-1})^n\|_{1,\beta m} \le \|RW^{-1}\|_{1,\beta m}^n
$$
for every $i,j\in I$.  But then
$$
	\|(RW^{-1})^n\|_\infty =
	\sup_{i\in I}\sum_{j\in I}(RW^{-1})^n_{ij} \le
	\|RW^{-1}\|_{1,\beta m}^n
	\sup_{i\in I}\sum_{j\in I}e^{-\beta m(i,j)} < \infty
$$
for all $n$.
We therefore have $\|RW^{-1}\|_\infty<\infty$, and we can choose
$n$ sufficiently large that $\|(RW^{-1})^n\|_\infty<1$.  The
conclusion now follows from Corollary \ref{cor:uniq}(4).
\end{proof}

\subsection{Proof of Theorem \ref{thm:oneside}}

In the case of one-sided local updates, the measure $\rho_{\le k}$ is 
$\gamma^J$-invariant for $\tau(J)=k$ (but not for $\tau(J)<k$).  The proof 
of Theorem \ref{thm:oneside} therefore proceeds by induction on $k$.  In 
each stage of the induction, we apply the logic of Theorem \ref{thm:main} 
to the partial local updates $(\gamma^J)_{J\in\mathcal{J}:\tau(J)=k}$, and 
use the induction hypothesis to estimate the remainder term.  

Throughout this section, we work in the setting of Theorem 
\ref{thm:oneside}.  Define
$$
	I_{\le k}:= \{i\in I:\tau(i)\le k\},\qquad\quad
	I_{k}:= \{i\in I:\tau(i)=k\}.
$$
Note that we can assume without loss of generality that $R_{ij}=0$ 
whenever $\tau(j)>\tau(i)$.  Indeed, the local update rule $\gamma^J_x$ 
does not depend on $x_j$ for $\tau(j)>\tau(J)$, so we can trivially choose 
the coupling $Q^J_{x,z}$ for $x^{I\backslash\{j\}}= z^{I\backslash\{j\}}$ 
such that $Q^J_{x,z}\eta_i=0$ for all $i\in J$.  On the other hand, the 
choice $R_{ij}=0$ evidently yields the smallest bound in Theorem 
\ref{thm:oneside}.  In the sequel, we will always assume that $R_{ij}=0$ 
whenever $\tau(j)>\tau(i)$.

The key induction step is formalized by the following result.

\begin{prop}
\label{prop:onesideinduction}
Assume (\ref{eq:follmer}).  Let $(\beta_i)_{i\in I_{\le k-1}}$ be 
nonnegative weights such that
$$
	|\rho_{\le k-1}g-\tilde\rho_{\le k-1}g| \le
	\sum_{i\in I_{\le k-1}}\delta_ig\,\beta_i
$$
for every bounded measurable quasilocal function 
$g$ on $\bbS_{\le k-1}$ so that $\delta_ig<\infty$ $\forall i$.
Then
$$
	|\rho_{\le k}f-\tilde\rho_{\le k}f| \le
	\sum_{j\in I_{\le k-1}}\Bigg\{\delta_jf +
	\sum_{i,l\in I_k}
	\delta_if\,D_{il}\,(W^{-1}R)_{lj}\Bigg\}\,\beta_j
	+
	\sum_{i,j\in I_k} \delta_if\,D_{ij}\,W_{jj}^{-1}a_j
$$
for every bounded measurable quasilocal function 
$f$ on $\bbS_{\le k}$ so that $\delta_if<\infty$ $\forall i$.
\end{prop}

\begin{proof}
We fix throughout the proof a bounded and measurable local function 
$f:\bbS_{\le k}\to\mathbb{R}$ such that $\delta_if<\infty$ for all 
$i\in I_{\le k}$.  The extension of the conclusion to quasilocal functions 
$f$ follows readily by localization as in the proof of Lemma 
\ref{lem:coupling}.

We denote by $G^{\mathbf{v}}$ and $\tilde G^{\mathbf{v}}$ the Gibbs 
samplers as defined in section \ref{sec:step2}.  Let us enumerate the 
partial cover $\{J\in\mathcal{J}:\tau(J)=k\}$ as $\{J_1,J_2,\ldots\}$, and 
define the weights $\mathbf{v}^r$ as in the proof of Proposition 
\ref{prop:contcomp}.  By the definition of the one-sided local update 
rule, $\rho_{\le k}$ is $G^{u\mathbf{v}^r}$-invariant and $\tilde\rho_{\le k}$
is $\tilde G^{u\mathbf{v}^r}$-invariant for every $r,u$ such that 
$\sum_Juv_J^r\le 1$.  Thus
$$
	|\rho_{\le k}f-\tilde\rho_{\le k}f| \le
	\sum_{i,j\in I_{\le k}}
		\delta_if\,N_{ij}^{u\mathbf{v}^r(n)}\,a_j^{u\mathbf{v}^r}
	+ |\rho_{\le k}(G^{u\mathbf{v}^r})^nf-
	\tilde\rho_{\le k}(G^{u\mathbf{v}^r})^nf|
$$
as in the proof of Corollary \ref{cor:precomp}, with the only distinction 
that we refrain from using the Wasserstein matrix to expand the second 
term in the proof of Proposition \ref{prop:markovcomp}.  We now use the 
induction hypothesis to obtain an improved estimate for the second term.

\begin{lem}
\label{lem:couplingplusplus}
We can estimate
$$
	|\rho_{\le k}g-\tilde\rho_{\le k}g| \le
	\sum_{i\in I_{\le k-1}}\delta_ig\,\beta_i +
	3\sum_{i\in I_k}\delta_ig\,(\rho\otimes\tilde\rho)\eta_i
$$
for any bounded and measurable quasilocal function
$g:\bbS_{\le k}\to\mathbb{R}$ such that $\delta_ig<\infty$
$\forall i$.
\end{lem}

\begin{proof}
For any $x\in\bbS_{\le k}$ we can estimate
$$
	|\rho_{\le k} g-\tilde\rho_{\le k} g| \le
	|\rho_{\le k-1} \hat g_x-
	\tilde\rho_{\le k-1} \hat g_x| +
	|\rho_{\le k}(g-\hat g_x)| +
	|\tilde\rho_{\le k}(g-\hat g_x)|,
$$
where we defined $\hat g_x(z):=g(z^{I_{\le k-1}}x^{I_k})$.  By Lemma 
\ref{lem:coupling} we have
$$
	|g(z)-\hat g_x(z)| \le
	\sum_{i\in I_k} \delta_ig\,\eta_i(z_i,x_i).
$$
We can therefore estimate using the induction hypothesis and the triangle
inequality
$$
	|\rho_{\le k} g-\tilde\rho_{\le k} g| \le
	\sum_{i\in I_{\le k-1}}\delta_ig\,\beta_i +
	\sum_{i\in I_k} \delta_ig\,
	\{\rho\eta_i(\mbox{}\cdot\mbox{},\tilde x_i) +
	\eta_i(\tilde x_i,x_i) +
	\tilde\rho\eta_i(\mbox{}\cdot\mbox{},x_i)\}
$$
for all $x,\tilde x\in\bbS_{\le k}$.  Now integrate this expression
with respect to $\rho(dx)\,\tilde\rho(d\tilde x)$.
\end{proof}

To lighten the notation somewhat we will write $\mathbf{v}=u\mathbf{v}^r$ 
until further notice.
Note that by construction $a_j^{\mathbf{v}}=0$ whenever $\tau(j)<k$, 
while $R_{ij}^{\mathbf{v}}=0$ whenever $\tau(j)>\tau(i)$ by assumption.
Thus we obtain using Lemma \ref{lem:couplingplusplus} and Lemma
\ref{lem:wasserstein}
\begin{align*}
	|\rho_{\le k}f-\tilde\rho_{\le k}f| \le
	\sum_{i,j\in I_k}
		\delta_if\,N_{ij}^{\mathbf{v}(n)}\,a_j^{\mathbf{v}}
	&+ 
	3\sum_{i,j\in I_k}\delta_if\,
	(I-W^{\mathbf{v}}+R^{\mathbf{v}})^n_{ij}
	\,(\rho\otimes\tilde\rho)\eta_j \\
	&+
	\sum_{i\in I_{\le k}}
	\sum_{j\in I_{\le k-1}}
	\delta_if\,(I-W^{\mathbf{v}}+R^{\mathbf{v}})^n_{ij}\,\beta_j,
\end{align*}
provided that 
$\sum_i\delta_if\,(I-W^{\mathbf{v}}+R^{\mathbf{v}})^n_{ij}<\infty$
for all $j$.

Next, note that as $v_J=0$ for $\tau(J)<k$, we have 
$R^{\mathbf{v}}_{ij}=W^{\mathbf{v}}_{ij}=0$
for $i\in I_{\le k-1}$.  Thus
$$
	V^{\mathbf{v}} =
	I-W^{\mathbf{v}}+R^{\mathbf{v}}
	=
	\left(
	\begin{array}{cc}
	\check V^{\mathbf{v}} & \check R^{\mathbf{v}} \\
	0 & I 
	\end{array}
	\right),
$$
where $\check V^{\mathbf{v}}:=(V^{\mathbf{v}}_{ij})_{i,j\in I_k}$
and $\check R^{\mathbf{v}}:=(R^{\mathbf{v}}_{ij})_{i\in I_k,j\in 
I_{\le k-1}}$.  In particular,
$$
	(I-W^{\mathbf{v}}+R^{\mathbf{v}})^n
	=
	\left(
	\begin{array}{cc}
	(\check V^{\mathbf{v}})^n & 
	\sum_{k=0}^{n-1}(\check V^{\mathbf{v}})^k
	\check R^{\mathbf{v}} \\
	0 & I 
	\end{array}
	\right).
$$
Moreover, as $R_{ij}^{\mathbf{v}}=0$ whenever $\tau(j)>\tau(i)$, we 
evidently have $(\check 
V^{\mathbf{v}})^k_{ij}=(V^{\mathbf{v}})^k_{ij}$ for $i,j\in I_k$.
Substituting into the above expression, we obtain
\begin{align*}
	|\rho_{\le k}f-\tilde\rho_{\le k}f| \le
	\sum_{i,j\in I_k}
		\delta_if\,N_{ij}^{\mathbf{v}(n)}\,a_j^{\mathbf{v}}
	&+ 
	3\sum_{i,j\in I_k}\delta_if\,
	(I-W^{\mathbf{v}}+R^{\mathbf{v}})^n_{ij}
	\,(\rho\otimes\tilde\rho)\eta_j \\
	&+
	\sum_{j\in I_{\le k-1}}
	\Bigg\{
	\delta_jf+
	\sum_{i,l\in I_k}
	\delta_if\,N_{il}^{\mathbf{v}(n)}\,R^{\mathbf{v}}_{lj}
	\Bigg\}
	\,\beta_j
\end{align*}
provided that 
$\sum_i\delta_if\,(I-W^{\mathbf{v}}+R^{\mathbf{v}})^n_{ij}<\infty$
for all $j$.  But the latter is easily verified using (\ref{eq:follmer}) 
and Lemma \ref{lem:exponentialest}, as $f$ is
local and $\delta_if<\infty$ for all $i$ by assumption.

The remainder of the proof now proceeds precisely as in the proof of  
Proposition \ref{prop:contcomp} and Theorem \ref{thm:main}.  We set 
$\mathbf{v}=(t/n)\mathbf{v}^r$, let $n\to\infty$ and then $r\to\infty$.  
The arguments for the first two terms are identical to the proof of 
Proposition \ref{prop:contcomp}, while the argument for the third term is 
essentially identical to the argument for the first term.  The proof is 
then completed as in the proof of Theorem \ref{thm:main}.  We leave the 
details for the reader. 
\end{proof}

We now proceed to complete the proof of Theorem \ref{thm:oneside}.

\begin{proof}[Proof of Theorem \ref{thm:oneside}]
Consider first the case that $k_-:=\inf_{i\in I}\tau(i)>-\infty$.
In this setting, we say that the comparison theorem holds for a given
$k\ge k_-$ if we have
$$
	|\rho_{\le k}f-\tilde\rho_{\le k}f| \le
	\sum_{i,j\in I_{\le k}}\delta_if\,D_{ij}\,W_{jj}^{-1}a_j
$$
for every bounded measurable quasilocal function $f$ on $\bbS_{\le k}$ 
such that $\delta_if<\infty$ $\forall i$.  We can evidently apply Theorem 
\ref{thm:main} to show that the comparison theorem holds for $k_-$.  We 
will now use Proposition \ref{prop:onesideinduction} to show that if the 
comparison theorem holds for $k-1$, then it holds for $k$ also.  Then the 
comparison theorem holds for every $k\ge k_-$ by induction, so the 
conclusion of Theorem \ref{thm:oneside} holds whenever $f$ is a local 
function.  The extension to quasilocal $f$ follows readily by localization 
as in the proof of Lemma \ref{lem:coupling}.

We now complete the induction step.  When the comparison theorem holds for 
$k-1$ (the induction hypothesis), we can apply Proposition 
\ref{prop:onesideinduction} with
$$
	\beta_i = \sum_{j\in I_{\le k-1}}D_{ij}\,W_{jj}^{-1}a_j.
$$
This gives
\begin{multline*}
	|\rho_{\le k}f-\tilde\rho_{\le k}f| \le
	\sum_{j,q\in I_{\le k-1}}
	\sum_{i,l\in I_k}
	\delta_if\,D_{il}\,(W^{-1}R)_{lq}\,
	D_{qj}\,W_{jj}^{-1}a_j \\ \mbox{}
	+
	\sum_{i,j\in I_{\le k-1}}\delta_if \, D_{ij}\,W_{jj}^{-1}a_j
	+
	\sum_{i,j\in I_k} \delta_if\,D_{ij}\,W_{jj}^{-1}a_j
\end{multline*}
for every bounded measurable quasilocal function 
$f$ on $\bbS_{\le k}$ so that $\delta_if<\infty$ $\forall i$. 
To complete the proof, it therefore suffices to show that we have
$$
	D_{ij}=
	\sum_{q\in I_{\le k-1}}
	\sum_{l\in I_k}
	D_{il}\,(W^{-1}R)_{lq}\,D_{qj}
	\qquad\mbox{for }i\in I_k,~j\in I_{\le k-1}.
$$
To see this, note that as $R_{ij}=0$ for $\tau(i)<\tau(j)$, we can write
\begin{align*}
	D_{ij}&=\sum_{p=1}^\infty 
	\sum_{\substack{j_1,\ldots,j_{p-1}\in I :\mbox{}\\
	\tau(j)\le\tau(j_1)\le\cdots\le\tau(j_{p-1})\le k}}
	(W^{-1}R)_{ij_{p-1}} \cdots
	(W^{-1}R)_{j_2j_1} 
	(W^{-1}R)_{j_{1}j} \\
	&=\sum_{p=1}^\infty ~
	\sum_{n=1}^p ~
	\sum_{l\in I_k} ~
	\sum_{q\in I_{\le k-1}}
 	(W^{-1}R)^{n-1}_{il}
	(W^{-1}R)_{lq} 
	(W^{-1}R)_{qj}^{p-n}
\end{align*}
for $i\in I_k$ and $j\in I_{\le k-1}$, where we have used that whenever
$\tau(j_1)\le\cdots\le\tau(j_{p-1})\le k$ there exists $1\le n\le p$ 
such that $j_1,\ldots,j_{p-n}\in I_{\le k-1}$ and 
$j_{p-n+1},\ldots,j_{p-1}\in I_k$.  Rearranging the last expression yields 
the desired identity for $D_{ij}$, completing the proof for the case 
$k_->-\infty$ (note that in this case the additional assumption 
(\ref{eq:oneside}) was not needed).

We now turn to the case that $k_-=-\infty$.  Let us say that 
$(\beta_i)_{i\in I_{\le k}}$ is a $k$-estimate if
$$
	|\rho_{\le k}g-\tilde\rho_{\le k}g| \le
	\sum_{i\in I_{\le k}}\delta_ig\,\beta_i
$$
for every bounded measurable quasilocal function $g$ on $\bbS_{\le k}$
such that $\delta_ig<\infty$ $\forall i$.  Then the conclusion of
Proposition \ref{prop:onesideinduction} can be reformulated as follows:
if $(\beta_i)_{i\in I_{\le k-1}}$ is a $(k-1)$-estimate, then
$(\beta_i')_{i\in I_{\le k}}$ is a $k$-estimate with
$\beta_i'=\beta_i$ for $i\in I_{\le k-1}$ and
$$
	\beta_i' = 
	\sum_{j\in I_{\le k-1}}
	\sum_{l\in I_k}
	D_{il}\,(W^{-1}R)_{lj}\,\beta_j
	+
	\sum_{j\in I_k} D_{ij}\,W_{jj}^{-1}a_j
$$
for $i\in I_k$.  We can therefore repeatedly apply Proposition 
\ref{prop:onesideinduction} to extend an initial estimate.  In particular, 
if we fix $k\in\mathbb{Z}$ and $n\ge 1$, and if $(\beta_i)_{i\in I_{\le 
k-n}}$ is a $(k-n)$-estimate, then we can obtain a $k$-estimate 
$(\beta'_i)_{i\in I_{\le k}}$ by iterating Proposition 
\ref{prop:onesideinduction} $n$ times.  We claim that
$$
	\beta_i' = 
	\sum_{s=k-n+1}^{k-r}
	\Bigg\{
	\sum_{j\in I_{\le k-n}}
	\sum_{l\in I_s}
	D_{il}\,(W^{-1}R)_{lj}\,\beta_j
	+
	\sum_{j\in I_s} D_{ij}\,W_{jj}^{-1}a_j\Bigg\}
$$
for $0\le r\le n-1$ and $i\in I_{k-r}$.  To see this, we proceed again by 
induction.  As $(\beta_i)_{i\in I_{\le k-n}}$ is a $(k-n)$-estimate,
the expression is valid for $r=n-1$ by Proposition
\ref{prop:onesideinduction}.  Now suppose the expression is valid for all
$u<r\le n-1$.  Then we obtain
\begin{align*}
	\beta_i' &= 
	\sum_{j\in I_{\le k-n}}
	\sum_{l\in I_{k-u}}
	D_{il}\,(W^{-1}R)_{lj}\,\beta_j
	+
	\sum_{j\in I_{k-u}} D_{ij}\,W_{jj}^{-1}a_j \\
	&\qquad\mbox{}+
	\sum_{s=k-n+1}^{k-u-1}
	\sum_{j\in I_{s}}
	\sum_{l\in I_{k-u}}
	\sum_{t=k-n+1}^{s}
	\sum_{q\in I_{\le k-n}}
	\sum_{p\in I_t}
	D_{il}\,(W^{-1}R)_{lj}\,D_{jp}\,(W^{-1}R)_{pq}\,\beta_q \\
	&\qquad\mbox{}+
	\sum_{s=k-n+1}^{k-u-1}
	\sum_{j\in I_{s}}
	\sum_{l\in I_{k-u}}
	\sum_{t=k-n+1}^{s}
	\sum_{q\in I_t} 
	D_{il}\,(W^{-1}R)_{lj}\,D_{jq}\,W_{qq}^{-1}a_q
\end{align*}
for $i\in I_{k-u}$ by Proposition \ref{prop:onesideinduction}.
Rearranging the sums yields
\begin{align*}
	\beta_i' &= 
	\sum_{j\in I_{\le k-n}}
	\sum_{l\in I_{k-u}}
	D_{il}\,(W^{-1}R)_{lj}\,\beta_j
	+
	\sum_{j\in I_{k-u}} D_{ij}\,W_{jj}^{-1}a_j \\
	&\qquad\mbox{}+
	\sum_{t=k-n+1}^{k-u-1} \Bigg\{
	\sum_{q\in I_{\le k-n}}
	\sum_{p\in I_t}
	\bar D_{ip}\,(W^{-1}R)_{pq}\,\beta_q 
	+ 
	\sum_{p\in I_t} 
	\bar D_{ip}\,W_{pp}^{-1}a_p\Bigg\},
\end{align*}
for $i\in I_{k-u}$, where we have defined
$$
	\bar D_{ij} :=
	\sum_{\ell=s}^{t-1}
	\sum_{q\in I_\ell}
	\sum_{l\in I_t}
	D_{il}\,(W^{-1}R)_{lq}\,D_{qj}
$$
whenever $i\in I_t$ and $j\in I_s$ for $s<t$.  But as $D_{qj}=0$ when 
$\tau(q)<\tau(j)$, we have
$$
	\bar D_{ij} =
	\sum_{q\in I_{\le t-1}}
	\sum_{l\in I_t}
	D_{il}\,(W^{-1}R)_{lq}\,D_{qj} = D_{ij}
	\qquad\mbox{for }i\in I_t,~j\in I_{\le t-1}
$$
using the identity used in the proof for the case $k_->-\infty$, and the 
claim follows.

We can now complete the proof for the case $k_-=-\infty$.
It suffices to prove the theorem for a given local function $f$
(the extension to quasilocal $f$ follows readily as in the 
proof of Lemma \ref{lem:coupling}).  Let us therefore fix a $K$-local 
function $f$ for some $K\in\mathcal{I}$, and let $k=\max_{i\in K}\tau(i)$
and $n \ge 1$.  By Lemma \ref{lem:coupling}, we find that $(\beta_i)_{i\in 
I_{\le k-n}}$ is trivially a $(k-n)$-estimate if we set 
$\beta_i=(\rho\otimes\tilde\rho)\eta_i$ for $i\in I_{\le k-n}$.
We therefore obtain
$$
	|\rho f-\tilde\rho f| \le
	\sum_{i,j\in I}\delta_if\,D_{ij}\,W_{jj}^{-1}a_j
	+\sum_{i\in I}\sum_{j\in I_{\le k-n}}
	\delta_if\,D_{ij}\,(\rho\otimes\tilde\rho)\eta_j
$$
from the $k$-estimate $(\beta_i')_{i\in I_{\le k}}$ derived above,
where we have used that $DW^{-1}R\le D$.  But as $f$ is local and 
$\delta_if<\infty$ for all $i$ by assumption, the second term vanishes as 
$n\to\infty$ by assumption (\ref{eq:oneside}).  This completes the proof 
for the case $k_-=-\infty$.
\end{proof}

\section{Application: particle filters}
\label{sec:appl}

Our original motivation for developing the generalized comparison theorems 
of this paper was the investigation of algorithms for filtering in high 
dimension.  In this section we will develop one such application in 
detail. Our result answers a question raised in \cite{RvH13}, and also 
serves as a concrete illustration of the utility of the generalized 
comparison theorems.

\subsection{Introduction and main result}

Let $(X_n,Y_n)_{n\ge 0}$ be a Markov chain.  We interpret $X_n$ as the 
unobserved component of the model and $Y_n$ as the observed component.  A 
fundamental problem in this setting is to track the state of the 
unobserved component $X_n$ given the history of observed data 
$Y_1,\ldots,Y_n$.  Such problems are ubiquitous in a wide variety of 
applications, ranging from classical tracking problems in navigation and 
robotics to large-scale forecasting problems such as weather prediction, 
and are broadly referred to as \emph{filtering} or \emph{data 
assimilation} problems.

In principle, the optimal solution to the tracking problem is
provided by the \emph{filter}
$$
	\pi_n := \mathbf{P}[X_n\in\mbox{}\cdot\mbox{}|Y_1,\ldots,Y_n]. 
$$ 
If the conditional distribution $\pi_n$ can be computed, in yields not 
only a least mean square estimate of the unobserved state $X_n$ but also a 
complete representation of the uncertainty in this estimate.  
Unfortunately, when $X_n$ takes values in a continuous 
(or finite but large)
state space, the 
filter is rarely explicitly computable and approximations become 
necessary.  In practice, filtering is widely implemented by a class of 
sequential Monte Carlo algorithms called \emph{particle filters}, cf.\ 
\cite{CMR05}, that have been very successful in classical applications. 

The major problem with particle filtering algorithms is that they 
typically require an exponential number of samples in the model dimension. 
Such algorithms are therefore largely useless in high-dimensional problems 
that arise in complex applications such as weather forecasting.  We refer 
to \cite{RvH13} for a detailed discussion of these issues and for further 
references.  In many applications, the high-dimensional nature of the 
problem is due to the presence of spatial degrees of freedom: $X_n$ and 
$Y_n$ at each time $n$ are themselves random fields that evolve 
dynamically over time.  In practice, such models are typically expected to 
exhibit decay of correlations.  We have started in \cite{RvH13} to explore 
the possibility that such properties could be exploited to beat the curse 
of dimensionality by including a form of spatial localization in the 
filtering algorithm.  In particular, the initial analysis in \cite{RvH13} 
has yielded dimension-free error bounds for the simplest possible class of 
local particle filtering algorithms, called \emph{block particle filters}, 
under strong (but dimension-free) model assumptions that ensure the 
presence of decay of correlations.

It should be noted that the block particle filtering algorithm exhibits 
some significant drawbacks that could potentially be resolved by using 
more sophisticated algorithms. These issues are discussed at length in 
\cite{RvH13}, but are beyond the scope of this paper.  In the sequel, we 
will reconsider the same algorithm that was introduced in \cite{RvH13}, 
but we provide an improved analysis of its performance on the basis of 
Theorem \ref{thm:main}.  We will see that the use of Theorem 
\ref{thm:main} already yields a \emph{qualitative} improvement over the 
main result of \cite{RvH13}.

In the remainder of this section we recall the setting of \cite{RvH13} and 
state our main result on block particle filters.  The following sections 
are devoted to the proofs.

\subsubsection*{Dynamical model}

Let $(X_n,Y_n)_{n\ge 0}$ be a Markov chain that takes values in the
product space $\bbX\times\bbY$, and whose transition probability $P$ 
can be factored as
$$
	P((x,y),A) = \int \mathbf{1}_A(x',y')\,
	p(x,x')\,g(x',y')\,\psi(dx')\,\varphi(dy').
$$
Such processes are called \emph{hidden Markov models}
\cite{CMR05}.  As a consequence of this definition, the unobserved process 
$(X_n)_{n\ge 0}$ is a Markov chain in its own right with
transition density $p$ (with respect to the reference measure
$\psi$), while each observation $Y_n$ is a noisy function of $X_n$
only with observation density $g$ (with respect to the reference
measure $\varphi$).

Our interest is in high-dimensional hidden Markov models that possess 
spatial structure.  To this end, we introduce a finite undirected graph 
$G=(V,E)$ that determines the spatial degrees of freedom of the model.
The state $(X_n,Y_n)$ at each time $n$ is itself a random field 
$(X_n^v,Y_n^v)_{v\in V}$ indexed by the vertices of the graph $G$.  
In particular, we choose
$$
	\bbX = \prod_{v\in V}\bbX^v,\qquad
	\bbY = \prod_{v\in V}\bbY^v,\qquad
	\psi = \bigotimes_{v\in V}\psi^v,\qquad
	\varphi = \bigotimes_{v\in V}\varphi^v,
$$
where $(\bbX^v,\psi^v)$ and $(\bbY^v,\varphi^v)$ are measure spaces for 
every $v\in V$.  To define the dynamics of the model, we introduce for 
each $v\in V$ a local transition density $p^v:\bbX\times\bbX^v\to\mathbb{R}_+$
and local observation density $g^v:\bbX^v\times\bbY^v\to\mathbb{R}_+$, and 
we set
$$
	p(x,z) = \prod_{v\in V}p^v(x,z^v),\qquad\qquad
	g(x,y) = \prod_{v\in V}g^v(x^v,y^v).
$$
Therefore, each observation $Y_n^v$ at location $v$ is a noisy function of 
the unobserved state $X_n^v$ at location $v$, and the current state 
$X_n^v$ is determined by the configuration $X_{n-1}$ at the previous time 
step.  We will assume throughout that the dynamics is \emph{local} in the 
sense that $p^v(x,z^v)=p^v(\tilde x,z^v)$ whenever $x^{N(v)}=\tilde 
x^{N(v)}$, where $N(v) = \{v'\in V:d(v,v')\le r\}$ denotes a neighborhood 
of radius $r$ around the vertex $v$ with respect to the graph distance 
$d$.  That is, the state $X_n^v$ at time $n$ and vertex $v$ depends only 
on the past $X_0,\ldots,X_{n-1}$ through the states $X_{n-1}^{N(v)}$ in an 
$r$-neighborhood of $v$ in the previous time step; the interaction radius 
$r$ is fixed throughout.  The dependence structure of our general model is 
illustrated schematically in Figure \ref{fig:lfilt} (in the simplest case 
of a linear graph $G$ with $r=1$).
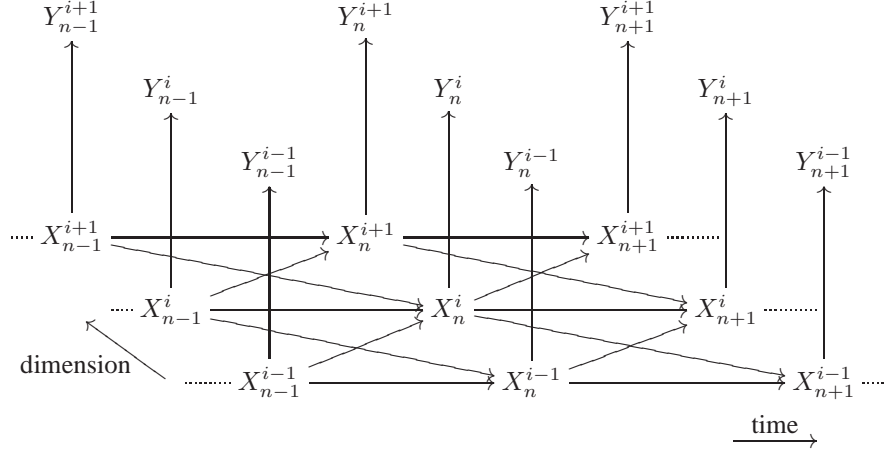
\begin{figure}
{\small
$$\xymatrixrowsep{.8pc}\xymatrixcolsep{.6pc}%
\xymatrix{
 & Y_{n-1}^{i+1} &&& Y_n^{i+1} &&& Y_{n+1}^{i+1} & \\
 && Y_{n-1}^{i} &&& Y_n^{i} &&& Y_{n+1}^{i} & \\
 &&& Y_{n-1}^{i-1} &&& Y_n^{i-1} &&& Y_{n+1}^{i-1} & \\
 \ar@{.}[r] & X_{n-1}^{i+1} \ar[uuu] \ar[rrr] \ar[drrrr] 
        &&& X_n^{i+1} \ar[uuu] \ar[rrr] \ar[drrrr] &&& X_{n+1}^{i+1} 
\ar[uuu] \ar@{.}[r] &  \\
 & ~~~~~~~~~~ \ar@{.}[r] & X_{n-1}^i \ar[uuu] \ar[rrr] \ar[urr] \ar[drrrr]
        &&& X_n^i \ar[uuu] \ar[rrr] \ar[urr] \ar[drrrr]
 &&& X_{n+1}^i \ar[uuu] \ar@{.}[r] &  \\
 && \ar[ul]^{\mbox{dimension}} ~~
 \ar@{.}[r] & X_{n-1}^{i-1} \ar[uuu] \ar[rrr] \ar[urr]
        &&& X_n^{i-1} \ar[uuu] \ar[rrr] \ar[urr]
 &&& X_{n+1}^{i-1} \ar[uuu] \ar@{.}[r] &  \\
&&&&&&&& \ar[r]^{\mbox{time}} &
}$$
}
\caption{Dependency graph of a high-dimensional filtering model
of the type considered in section \ref{sec:appl}.\label{fig:lfilt}}
\end{figure}

\subsubsection*{Block particle filters}

An important property of the filter 
$\pi_n=\mathbf{P}[X_n\in\mbox{}\cdot\mbox{}|Y_1,\ldots,Y_n]$ is that it 
can be computed recursively.  To this end, define for every 
probability measure $\rho$ on $\bbX$ the probability measures 
$\mathsf{P}\rho$ and $\mathsf{C}_n\rho$ as follows:
$$
	\mathsf{P}\rho(dx') := \psi(dx')\int p(x,x')\,\rho(dx),
	\qquad
	\mathsf{C}_n\rho(dx) := \frac{g(x,Y_n)\,\rho(dx)}{\int
	g(z,Y_n)\,\rho(dz)}.
$$
Then it is an elementary consequence of Bayes' formula that \cite{CMR05}
$$
	\pi_n = \mathsf{F}_n\pi_{n-1}:=\mathsf{C}_n\mathsf{P}\pi_{n-1}
	\quad\mbox{for every }n\ge 1,
$$
where the initial condition is given by
$\pi_0=\mu:=\mathbf{P}[X_0\in\mbox{}\cdot\mbox{}]$.  In the sequel we 
will often write $\pi_n^\mu:=\mathsf{F}_n\cdots\mathsf{F}_1\mu$ to 
indicate explicitly the initial condition of the filtering recursion.

To obtain approximate algorithms, we insert additional steps in the 
filtering recursion that enable a tractable implementation.  The classical 
particle filtering algorithm inserts a random sampling step $\mathsf{S}^N$
in the filtering recursion that replaces the current filtering 
distribution by the empirical measure of $N$ independent samples: that is, 
$$
	\mathsf{S}^N\rho := \frac{1}{N}\sum_{i=1}^N\delta_{x(i)}
	\quad\mbox{where}\quad
	(x(i))_{i=1,\ldots,N}\mbox{ are i.i.d.\ samples}\sim\rho.
$$
This gives rise to a sequential Monte Carlo algorithm that maintains at 
each time $N$ approximate samples from the current filtering distribution;
see \cite{CMR05,RvH13} and the references therein.  As $N\to\infty$, the 
sampling error vanishes by the law of large numbers and the particle 
filter converges to the exact filter.  Unfortunately, the number of 
samples $N$ needed to achieve a fixed error is typically exponential in 
the model dimension $\card V$.

To alleviate the curse of dimensionality we must localize the algorithm so 
that it can benefit from the decay of correlations of the underlying 
model.  The simplest possible algorithm of this type inserts an additional 
localization step in the filtering recursion in the following manner.
Fix a partition $\mathcal{K}$ of the vertex set $V$ into nonoverlapping 
blocks, and define for any probability measure $\rho$ on $\bbX$ the
blocking operator $\mathsf{B}\rho$ as
$$
	\mathsf{B}\rho := \bigotimes_{K\in\mathcal{K}}
	\mathsf{B}^K\rho,
$$
where we denote by $\mathsf{B}^J\rho$ for $J\subseteq V$ the marginal of 
$\rho$ on $\bbX^J$.  That is, the blocking 
operation forces the underlying measure to be independent across different 
blocks in the partition $\mathcal{K}$.  The \emph{block particle filter} 
$\hat\pi_n^\mu$ is now defined by the recursion
$$
	\hat\pi_n^\mu := \mathsf{\hat F}_n\cdots
	\mathsf{\hat F}_1\mu,\qquad\qquad
	\mathsf{\hat F}_k := \mathsf{C}_k\mathsf{B}\mathsf{S}^N\mathsf{P}.
$$
This algorithm is very straightforward to implement in practice, cf.\ 
\cite{RvH13}.

\subsubsection*{Analysis of \cite{RvH13}}

We first introduce some notation.
Recall that in our model, each vertex $v$ interacts only with vertices in 
an $r$-neighborhood $N(v)$ in the previous time step, where the 
interaction radius $r$ is fixed throughout.  Given $J\subseteq V$, define 
the $r$-inner boundary as
$$
	\partial J:=\{v\in J:N(v)\not\subseteq J\}.
$$
Thus $\partial J$ is the subset of vertices in $J$ that can interact 
with vertices outside $J$ in one time step of the dynamics.  We also 
define the quantities
\begin{align*}
	|\mathcal{K}|_\infty &:= \max_{K\in\mathcal{K}}\card K,\\
	\Delta &:= \max_{v\in V}\card\{v'\in V:d(v,v')\le r\},\\
	\Delta_\mathcal{K} &:= \max_{K\in\mathcal{K}}
		\card\{K'\in\mathcal{K}:d(K,K')\le r\},
\end{align*}
where $d(J,J'):=\min_{v\in J}\min_{v'\in J'}d(v,v')$ for 
$J,J'\subseteq V$.  Thus $|\mathcal{K}|_\infty$ is the maximal size of a 
block, while $\Delta$ ($\Delta_\mathcal{K}$) is the maximal number of 
vertices (blocks) that interact with a single vertex (block) in one time 
step.  It should be noted that $r,\Delta,\Delta_\mathcal{K}$ are 
\emph{local} quantities that depend on the geometry but not on the size of 
the graph $G$.  We finally introduce, for each $J\subseteq I$, the 
following local distance between random measures $\rho$ and $\rho'$:
$$
	\tnorm{\rho-\rho'}_J :=
	\sup_{f\in\mathcal{X}^J:|f|\le 1}\mathbf{E}[
	|\rho f-\rho'f|^2]^{1/2}
$$
where $\mathcal{X}^J$ denotes the set of $J$-local measurable 
functions $f$ on $\bbX$.  For simplicity, we write
$\pi_n^x:=\pi_n^{\delta_x}$ and $\hat\pi_n^x:=\hat\pi_n^{\delta_x}$
when the filtering recursions are initialized at a point $x\in\bbX$.

We can now recall the main result of \cite{RvH13}.  

\begin{theorem}[\rm Theorem 2.1 in \cite{RvH13}]
\label{thm:RvH13}
There exists a constant $0<\varepsilon_0<1$, depending only on the local 
quantities $\Delta$ and $\Delta_\mathcal{K}$, such that the following 
holds.

Suppose there exist $\varepsilon_0<\varepsilon<1$ and $0<\kappa<1$ such 
that
$$
	\varepsilon\le p^v(x,z^v)\le\varepsilon^{-1},\qquad
	\kappa\le g^v(x^v,y^v)\le\kappa^{-1}\qquad
	\forall\,v\in V,~x,z\in\bbX,~y\in\bbY.
$$
Then for every $n\ge 0$, $\sigma\in\bbX$, $K\in\mathcal{K}$ and 
$J\subseteq K$ we have
$$
	\tnorm{\pi_n^\sigma-\hat\pi_n^\sigma}_J \le
	\alpha\card J\bigg[
	e^{-\beta_1d(J,\partial K)}+
	\frac{e^{\beta_2|\mathcal{K}|_\infty}}{N^{\frac{1}{2}}}
	\bigg],
$$
where the constants $0<\alpha,\beta_1,\beta_2<\infty$ depend only
on $\varepsilon$, $\kappa$, $r$, $\Delta$, and $\Delta_\mathcal{K}$.
\end{theorem}

The error bound in Theorem \ref{thm:RvH13} contains two terms.  The first 
term quantifies the bias introduced by the localization $\mathsf{B}$, 
which decreases when we take larger blocks.  The second term quantifies 
the variance introduced by the sampling $\mathsf{S}^N$, which decreases 
with increasing sample size $N$ but grows exponentially in the block size.  
Traditional particle filtering algorithms correspond to the choice of a 
single block $\mathcal{K}=\{V\}$, and in this case the error grows 
exponentially in the dimension $\card V$. To avoid this curse of 
dimensionality, we must tune the block size so as to optimize the tradeoff 
between bias and variance.  As all the constants in Theorem 
\ref{thm:RvH13} depend only on local quantities, the optimal block size 
results in a dimension-free error bound.  We refer to \cite{RvH13} for a 
full discussion.

\subsubsection*{Main result}

The intuition behind the block particle filtering algorithm is that the 
localization controls the sampling error (as it replaces the model 
dimension $\card V$ by the block size $|\mathcal{K}|_\infty$), while the 
decay of correlations property of the model controls the localization 
error (as it ensures that the effect of the localization decreases in the 
distance to the block boundary).  This intuition is clearly visible in the 
conclusion of Theorem \ref{thm:RvH13}. It is however not automatically the 
case that our model does indeed exhibit decay of correlations: when there 
are strong interactions between the vertices, phase transitions can arise 
and the decay of correlations can fail much as for standard models in 
statistical mechanics \cite{LMS90}, in which case we cannot expect to 
obtain dimension-free performance for the block particle filter.  Such 
phenomena are ruled out in Theorem \ref{thm:RvH13} by the assumption that 
$\varepsilon\le p^v\le\varepsilon^{-1}$ for $\varepsilon>\varepsilon_0$, 
which ensures that the interactions in our model are sufficiently weak.

It is notoriously challenging to obtain sharp quantitative results for 
interacting models, and it is unlikely that one could obtain realistic 
values for the constants in Theorem \ref{thm:RvH13} at the level of 
generality considered here.  More concerning, however, is that the weak 
interaction assumption of Theorem \ref{thm:RvH13} is already 
unsatisfactory at the \emph{qualitative} level.  Note that there is no 
interaction between the vertices in the extreme case $\varepsilon=1$; the 
assumption $\varepsilon>\varepsilon_0$ should be viewed as a perturbation 
of this situation (i.e., weak interactions).  However, setting 
$\varepsilon=1$ turns off not only the interaction between different 
vertices, but also the interaction between the same vertex at different 
times: in this setting the dynamics of the model become trivial. In 
contrast, one would expect that it is only the strength of the spatial 
interactions, and not the local dynamics, that is relevant for 
dimension-free errors, so that Theorem \ref{thm:RvH13} places an 
unnatural restriction on our understanding of block particle filters.

Our main result resolves this qualitative deficiency of Theorem 
\ref{thm:RvH13}.  Rather than assuming $p^v(x,z^v)\approx 1$ as in Theorem 
\ref{thm:RvH13}, we will assume only that the spatial interactions are 
weak in the sense that $p^v(x,z^v)\approx q^v(x^v,z^v)$, where the 
transition density $q^v$ describes the local dynamics at the vertex 
$v$ in the absence of interactions.

\begin{theorem}
\label{thm:block}
For any $0<\delta<1$ there exists $0<\varepsilon_0<1$, 
depending only on $\delta$ and $\Delta$, such that 
the following holds.
Suppose there exist $\varepsilon_0<\varepsilon<1$ and $0<\kappa<1$ so
that
\begin{align*}
	\varepsilon q^v(x^v,z^v) &\le p^v(x,z^v)\le
	\varepsilon^{-1} q^v(x^v,z^v), \\
	\delta &\le q^v(x^v,z^v)\le \delta^{-1},\\
	\kappa &\le g^v(x^v,y^v)\le\kappa^{-1}
\end{align*}
for every $v\in V$, $x,z\in\bbX$, $y\in\bbY$, where 
$q^v:\bbX^v\times\bbX^v\to\mathbb{R}_+$ is a transition density with
respect to $\psi^v$.
Then for every $n\ge 0$, $\sigma\in\bbX$, $K\in\mathcal{K}$ and 
$J\subseteq K$ we have
$$
	\tnorm{\pi_n^\sigma-\hat\pi_n^\sigma}_J \le
	\alpha\card J\bigg[
	e^{-\beta_1d(J,\partial K)}+
	\frac{e^{\beta_2|\mathcal{K}|_\infty}}{N^\gamma}
	\bigg],
$$
where $0<\gamma\le\frac{1}{2}$ and $0<\alpha,\beta_1,\beta_2<\infty$ 
depend only on $\delta$, $\varepsilon$, $\kappa$, $r$, $\Delta$, and 
$\Delta_\mathcal{K}$.
\end{theorem}

In Theorem \ref{thm:block}, the parameter $\varepsilon$ controls the 
spatial correlations while the parameter $\delta$ controls the temporal 
correlations (in contrast to Theorem \ref{thm:RvH13}, where both are 
controlled simultaneously by $\varepsilon$).  The key point is that 
$\delta$ can be arbitrary, and only $\varepsilon$ must lie above the 
threshold $\varepsilon_0$.  That the threshold $\varepsilon_0$ depends on 
$\delta$ is natural: the more ergodic the dynamics, the more spatial 
interactions can be tolerated without losing decay of correlations.

The proof of Theorem \ref{thm:RvH13} was based on repeated application of 
the classical Dobrushin comparison theorem (Corollary \ref{cor:follmer}). 
While there are some significant differences between the details of the 
proofs, the essential improvement that makes it possible to prove Theorem 
\ref{thm:block} is that we can now exploit the generalized comparison 
theorem (Theorem \ref{thm:main}), which enables us to treat the spatial 
and temporal degrees of freedom on a different footing.

\subsubsection*{Organization of the proof}

As in \cite{RvH13}, we consider three recursions
$$
	\pi_n^\mu := \mathsf{F}_n\cdots\mathsf{F}_1\mu,\qquad\quad
	\tilde\pi_n^\mu := \mathsf{\tilde F}_n\cdots\mathsf{\tilde F}_1\mu,\qquad\quad
	\hat\pi_n^\mu := \mathsf{\hat F}_n\cdots\mathsf{\hat F}_1\mu,
$$
where $\mathsf{F}_n:=\mathsf{C}_n\mathsf{P}$, $\mathsf{\tilde 
F}_n:=\mathsf{C}_n\mathsf{BP}$, and $\mathsf{\hat 
F}_n:=\mathsf{C}_n\mathsf{BS}^N\mathsf{P}$. The filter $\pi_n^\mu$ and the 
block particle filter $\hat\pi_n^\mu$ were already defined above. The
block filter $\tilde\pi_n^\mu$ is intermediate: it inserts only the 
localization but not the sampling step in the filtering recursion.  
This allows to decompose the approximation error into two terms, one due 
to localization and one due to sampling
$$
	\tnorm{\pi_n^\mu-\hat\pi_n^\mu}_J \le
	\tnorm{\pi_n^\mu-\tilde\pi_n^\mu}_J +
	\tnorm{\tilde\pi_n^\mu-\hat\pi_n^\mu}_J
$$
by the triangle inequality.  In the proof of Theorem
\ref{thm:block}, each of the terms on the right will be considered 
separately.  The first term, which quantifies the bias due to the 
localization, will be bounded in section \ref{sec:bias}.  The second term, 
which quantifies the sampling variance, will be bounded in section 
\ref{sec:variance}.  Combining these two bounds completes the proof.

\subsection{Bounding the bias}
\label{sec:bias}

The goal of this section is to bound the bias term
$\|\pi_n^\sigma-\tilde\pi_n^\sigma\|_J$, where we denote by
$$
	\|\mu-\nu\|_J := \sup_{f\in\mathcal{X}^J:|f|\le 1}|\mu f-\nu f|
$$
the local total variation distance on the set of sites $J$.  [Note that 
$\|\mu-\nu\|_J\le K$ for some $K\in\mathbb{R}$ evidently implies
$\tnorm{\mu-\nu}_J\le K$; the random measure norm $\tnorm{\cdot}_J$ will 
be essential to bound the sampling error, but is irrelevant for 
the bias term.]

Let us first give an informal outline of the ideas behind the proof of the 
bias bound.  While the filter $\pi_n^\sigma$ is itself a high-dimensional 
distribution (defined on the set of sites $V$), we do not know how to 
obtain a tractable local update rule for it.  We therefore cannot 
apply Theorem \ref{thm:main} directly.  Instead, we will consider the 
\emph{smoothing} distribution
$$
	\rho = \mathbf{P}^\sigma[X_1,\ldots,X_n\in\mbox{}\cdot\mbox{}|
	Y_1,\ldots,Y_n],
$$
defined on the extended set of sites $I=\{1,\ldots,n\}\times V$ and 
configuration space $\bbS=\bbX^{n}$.  As $(X_k^v,Y_k^v)_{(k,v)\in I}$ is 
a Markov random field (cf.\ Figure \ref{fig:lfilt}), we can read off a 
local update rule for $\rho$ from the model definition.  At the same time, 
as $\pi_n^\sigma=\mathbf{P}^\sigma[X_n\in\mbox{}\cdot\mbox{}|Y_1,\ldots,Y_n]$ 
is a marginal of $\rho$, we immediately obtain estimates for 
$\pi_n^\sigma$ from estimates for $\rho$.

This basic idea relies on the probabilistic definition of the filter as a 
conditional distribution of a Markov random field: the filtering recursion 
(which was only introduced for computational purposes) plays no role in 
the analysis.  The block filter $\tilde\pi_n^\sigma$, on the other hand, 
is \emph{defined} in terms of a recursion and does not have an intrinsic 
probabilistic interpretation.  In order to handle the block filter, we 
will artificially cook up a probability measure $\mathbf{\tilde P}$ 
on $\bbS$ such that the block filter satisfies $\tilde\pi^\sigma_n = 
\mathbf{\tilde P}[X_n\in\mbox{}\cdot\mbox{}|Y_1,\ldots,Y_n]$, and set 
$$
	\tilde\rho = 
	\mathbf{\tilde P}[X_1,\ldots,X_n\in\mbox{}\cdot\mbox{}|
	Y_1,\ldots,Y_n].
$$
This implies in particular that
$$
	\|\pi_n^\sigma-\tilde\pi_n^\sigma\|_J =
	\|\rho-\tilde\rho\|_{\{n\}\times J},
$$
and we can now bound the bias term by applying Theorem \ref{thm:main}.

To apply the comparison theorem we must choose a good cover 
$\mathcal{J}$.  It is here that the full flexibility of Theorem 
\ref{thm:main}, as opposed to the classical comparison theorem, comes into 
play.  If we were to apply Theorem \ref{thm:main} with the singleton cover 
$\mathcal{J}_{\rm s}=\{\{i\}:i\in I\}$, we would recover the result of 
Theorem \ref{thm:RvH13}: in this case both the spatial and temporal 
interactions must be weak in order to ensure that 
$D=\sum_n(W^{-1}R)^n<\infty$. To avoid this problem, we work instead with 
larger blocks in the temporal direction.  That is, our blocks 
$J\in\mathcal{J}$ will have the form $J=\{k+1,\ldots,k+q\}\times\{v\}$ for 
an appropriate choice of the block length $q$.  The local update 
$\gamma^J_x$ now behaves as $q$ time steps of an ergodic Markov chain in 
$\bbX^v$: the temporal interactions decay geometrically with $q$, and can 
therefore be made arbitrarily small even if the interaction in one time 
step is arbitrarily strong.  On the other hand, when we increase $q$ there 
will be more nonzero terms in the matrix $W^{-1}R$.  We must therefore 
ultimately tune the block length $q$ appropriately to obtain the result of 
Theorem \ref{thm:block}.

\begin{remark}
The approach used here to bound the bias directly using the comparison 
theorem is different than the one used in \cite{RvH13}, which exploits the 
recursive property of the filter.  The latter approach has a broader 
scope, as it does not rely on the ability to express the approximate 
filter as the marginal of a random field as we do above: this could be 
essential for the analysis of more sophisticated algorithms that do not 
admit such a representation.  For the purposes of this paper, however, the 
present approach provides an alternative and somewhat shorter proof that 
is well adapted to the analysis of block particle filters.
\end{remark}

\begin{remark}
The problem under investigation is based on an interacting
Markov chain model, and is therefore certainly dynamical in nature.  
Nonetheless, our proofs use Theorem \ref{thm:main} and not the 
one-sided Theorem \ref{thm:oneside}.  If we were to approximate the 
dynamics of the Markov chain $X_n$ itself, it would be much 
more convenient to apply Theorem \ref{thm:oneside} as the model is 
already defined in terms of one-sided conditional distributions 
$p(x,z)\,\psi(dz)$.  Unfortunately, when we condition on the observations 
$Y_n$, the one-sided conditional distributions take a complicated form 
that incorporates all the information in the future observations, whereas
conditioning on all variables outside a block $J\in\mathcal{J}$ gives rise 
to relatively tractable expressions.  For this reason, the static 
``space-time'' picture remains the most convenient approach for the 
investigation of high-dimensional filtering problems.
\end{remark}

We now turn to the details of the proof.  We first state the 
main result of this section.

\begin{theorem}[\rm Bias term]
\label{thm:bias}
Suppose there exist $0<\varepsilon,\delta<1$ such that 
\begin{align*}
	\varepsilon q^v(x^v,z^v) &\le p^v(x,z^v)\le
	\varepsilon^{-1}q^v(x^v,z^v), \\
	\delta &\le q^v(x^v,z^v)\le \delta^{-1}
\end{align*}
for every $v\in V$ and $x,z\in\bbX$, where 
$q^v:\bbX^v\times\bbX^v\to\mathbb{R}_+$ is a transition density with
respect to $\psi^v$.  Suppose also that we can choose $q\in\mathbb{N}$ and 
$\beta>0$ such that
$$
	c:=
	3q\Delta^2
	e^{\beta(q+2r)}(1-\varepsilon^{2(\Delta+1)}) +
	e^{\beta}(1-\varepsilon^2\delta^2) +
	e^{\beta q}(1-\varepsilon^2\delta^2)^q <1.
$$
Then we have
$$
	\|\pi_n^\sigma-\tilde\pi_n^\sigma\|_J \le
	\frac{2e^{\beta r}}{1-c}\,(1-\varepsilon^{2(q+1)\Delta})\,
	\card J\,
	e^{-\beta d(J,\partial K)}
$$
for every $n\ge 0$, $\sigma\in\bbX$, $K\in\mathcal{K}$ and $J\subseteq K$.
\end{theorem}

In order to use the comparison theorem, we must have a method to construct 
couplings.  Before we proceed to the proof of Theorem \ref{thm:bias}, we 
begin by formulating two elementary results that will provide us with the 
necessary tools for this purpose.

\begin{lem}
\label{lem:minorize}
If probability measures $\mu,\nu,\gamma$ satisfy
$\mu(A)\ge\alpha\gamma(A)$ and 
$\nu(A)\ge\alpha\gamma(A)$ for every measurable set $A$, there 
is a coupling $Q$ of $\mu,\nu$ with
$\int \mathbf{1}_{x\ne z}\,Q(dx,dz)\le 1-\alpha$.
\end{lem}

\begin{proof}
Define $\tilde\mu = (\mu-\alpha\gamma)/(1-\alpha)$,
$\tilde\nu = (\nu-\alpha\gamma)/(1-\alpha)$, and let
$$
	Qf = \alpha\int f(x,x)\,\gamma(dx) +
	(1-\alpha)\int f(x,z)\,\tilde\mu(dx)\,\tilde\nu(dz).
$$
The claim follows readily.
\end{proof}

\begin{lem}
\label{lem:markovminorize}
Let $P_1,\ldots,P_q$ be transition kernels on a measurable space 
$\mathbb{T}$, and define
$$
	\mu_x(d\omega_1,\ldots,d\omega_q) = 
	P_1(x,d\omega_1)P_2(\omega_1,d\omega_2)\cdots
	P_q(\omega_{q-1},d\omega_q).
$$
Suppose that there exist probability measures $\nu_1,\ldots,\nu_q$ on
$\mathbb{T}$ such that $P_i(x,A)\ge\alpha\nu_i(A)$ for every measurable
set $A$, $x\in\mathbb{T}$, and $i\le q$.  Then there exists for 
every $x,z\in\mathbb{T}$ a coupling $Q_{x,z}$ of $\mu_x$ and $\mu_z$ such
that $\int \mathbf{1}_{\omega_i\ne\omega_i'}\,
Q_{x,z}(d\omega,d\omega') \le (1-\alpha)^i$
for every $i\le q$.
\end{lem}

\begin{proof}
Define the transition kernels 
$\tilde P_i = (P_i-\alpha\nu_i)/(1-\alpha)$ and
\begin{align*}
	\tilde Q_if(x,z) =
	\alpha\int f(x',x')\,\nu_i(dx') 
	&+ (1-\alpha)\,\mathbf{1}_{x\ne z}
	\int f(x',z')\,\tilde P_i(x,dx')\,\tilde P_i(z,dz') \\ & +
	(1-\alpha)\,\mathbf{1}_{x=z}\int f(x',x')\,\tilde P_i(x,dx').
\end{align*}
Then $\tilde Q_i(x,z,\mbox{}\cdot\mbox{})$ is a coupling of 
$P_i(x,\mbox{}\cdot\mbox{})$ and $P_i(z,\mbox{}\cdot\mbox{})$.  Now define
$$
	Q_{x,z}(d\omega_1,d\omega_1',\ldots,d\omega_q,d\omega_q') = 
	\tilde Q_1(x,z,d\omega_1,d\omega_1')\cdots
	\tilde Q_q(\omega_{q-1},\omega_{q-1}',d\omega_q,d\omega_q').
$$
The result follows readily once we note that
$\int \mathbf{1}_{x'\ne z'}\,\tilde Q_i(x,z,dx',dz') \le
(1-\alpha)\,\mathbf{1}_{x\ne z}$.
\end{proof}

We can now proceed to the proof of Theorem \ref{thm:bias}.

\begin{proof}[Proof of Theorem \ref{thm:bias}]
We begin by constructing a measure $\mathbf{\tilde P}$ that 
allows to describe the block filter $\tilde\pi_n^\sigma$ as a conditional
distribution, as explained above.  We fix the initial condition 
$\sigma\in\bbX$ throughout the proof (the dependence of various quantities 
on $\sigma$ is implicit).

To construct $\mathbf{\tilde P}$, define for $K\in\mathcal{K}$ and $n\ge 1$
the function
$$
	h^K_n(x,z^{\partial K}) := 
	\int \tilde\pi_{n-1}^\sigma(d\omega)
	\prod_{v\in\partial K}p^v(x^K\omega^{V\backslash K},z^v).
$$
Evidently $h^K_n$ is a transition density with respect to 
$\bigotimes_{v\in\partial K}\psi^v$.  Let
$$
	\tilde p_n(x,z) := \prod_{K\in\mathcal{K}}
	h^K_n(x,z^{\partial K})\prod_{v\in K\backslash\partial K}
	p^v(x,z^v),
$$
and define $\mathsf{\tilde P}_n\mu(dx') := \psi(dx')\int \tilde 
p_n(x,x')\,\mu(dx)$.  Then $\mathsf{\tilde P}_n\tilde\pi_{n-1}^\sigma =
\mathsf{BP}\tilde\pi_{n-1}^\sigma$ by construction for every $n\ge 1$, as
$\tilde\pi_{n-1}^\sigma$ is a product measure across blocks.  Thus we have
$$
	\pi_n^\sigma = \mathsf{C}_n\mathsf{P}\cdots
	\mathsf{C}_1\mathsf{P}\delta_\sigma,\qquad\qquad
	\tilde\pi_n^\sigma = \mathsf{C}_n\mathsf{\tilde P}_n\cdots
	\mathsf{C}_1\mathsf{\tilde P}_1\delta_\sigma.
$$
In particular, the filter and the block filter satisfy the same recursion 
with different transition densities $p$ and $\tilde p_n$.  We can 
therefore interpret the block filter as the filter corresponding to a 
time-inhomogeneous Markov chain with transition densities $\tilde p_n$:
that is, if we set
\begin{multline*}
	\mathbf{\tilde P}[(X_1,\ldots,X_n,Y_1,\ldots,Y_n)\in A] := \\
	\int \mathbf{1}_A(x_1,\ldots,x_n,y_1,\ldots,y_n)\,
	\tilde p_1(\sigma,x_1)
	\prod_{k=2}^n \tilde p_k(x_{k-1},x_k)\,g(x_k,y_k)\,
	\psi(dx_k)\,\varphi(dy_k)
\end{multline*}
(note that $\mathbf{P}^\sigma$ satisfies the same formula where 
$\tilde p_k$ is replaced by $p$), we can write
$$
	\tilde\pi_n^\sigma = \mathbf{\tilde P}[X_n\in\mbox{}\cdot\mbox{}|
	Y_1,\ldots,Y_n].
$$
Let us emphasize that the transition densities $\tilde p_k$ and operators 
$\mathsf{\tilde P}_k$ themselves depend on the initial condition $\sigma$, 
which is certainly not the case for the regular filter.  However, since 
$\sigma$ is fixed throughout the proof, this is irrelevant for our 
computations. 

From now on we fix $n\ge 1$ in the remainder of the proof.  Let
$$
	\rho = \mathbf{P}^\sigma[X_1,\ldots,X_n\in\mbox{}\cdot\mbox{}|
	Y_1,\ldots,Y_n],\qquad\quad
	\tilde\rho = \mathbf{\tilde P}[X_1,\ldots,X_n\in\mbox{}\cdot\mbox{}|
	Y_1,\ldots,Y_n].
$$
Then $\rho$ and $\tilde\rho$ are probability measures on $\bbS=\bbX^n$, 
which is naturally indexed by the set of sites $I=\{1,\ldots,n\}\times V$
(the observation sequence on which we condition is arbitrary and can be 
considered fixed throughout the proof).  The proof now proceeds by 
applying Theorem \ref{thm:main} to $\rho,\tilde\rho$, the main difficulty 
being the construction of a coupled update rule.

Fix $q\ge 1$.  We first specify the cover 
$\mathcal{J}=\{J_l^v:1\le l\le\lceil n/q\rceil,v\in V\}$ as follows:
$$
	J_l^v := \{(l-1)q+1,\ldots,lq\wedge n\}\times\{v\}\qquad
	\mbox{for }1\le l\le \lceil n/q\rceil,~v\in V.
$$
We choose the natural local updates
$\gamma^J_x(dz^J) = \rho(dz^J|x^{I\backslash J})$ and 
$\tilde\gamma^J_x(dz^J) = \tilde\rho(dz^J|x^{I\backslash J})$,
and postpone the construction of the coupled updates $Q_{x,z}^J$ and $\hat 
Q_x^J$ to be done below.  Now note that the cover $\mathcal{J}$ is in fact 
a partition of $I$; thus Theorem \ref{thm:main} yields
$$
	\|\pi_n^\sigma-\tilde\pi_n^\sigma\|_J =
	\|\rho-\tilde\rho\|_{\{n\}\times J} \le
	2\sum_{i\in \{n\}\times J}\sum_{j\in I} D_{ij}\,b_j
$$
provided that $D=\sum_{k=0}^\infty C^k<\infty$ (cf.\ Corollary 
\ref{cor:uniq}), where
$$
	C_{ij} = \sup_{\substack{
				x,z\in\bbS:\\
				x^{I\backslash\{j\}}=z^{I\backslash\{j\}}
		}}
		\int \mathbf{1}_{\omega_i\ne\omega_i'}
		\,Q_{x,z}^{J(i)}(d\omega,d\omega'),\qquad
	b_i = \sup_{x\in\bbS}
	\int \mathbf{1}_{\omega_i\ne\omega_i'}
                \,\hat Q_{x}^{J(i)}(d\omega,d\omega'),
$$
and where we write $J(i)$ for the unique block $J\in\mathcal{J}$ that
contains $i\in I$.  To put this bound to good use, we must introduce 
coupled updates $Q_{x,z}^J$ and $\hat Q_x^J$ and
estimate $C_{ij}$ and $b_j$.

Let us fix until further notice a block $J=J_l^v\in\mathcal{J}$.
We will consider first the case that $1<l<\lceil n/q\rceil$; the
cases $l=1,\lceil n/q\rceil$ will follow subsequently using the identical 
proof.  Let $s=(l-1)q$.  Then we can compute explicitly the local update
rule
\begin{align*}
	&\gamma_x^J(A) = \\
	&\quad \frac{
		\int
		\mathbf{1}_A(x^J)\,p^v(x_s,x_{s+1}^v)
		\prod_{m=s+1}^{s+q}
		g^v(x_m^v,Y_m^v)
		\prod_{w\in N(v)}
		p^w(x_{m},x_{m+1}^w)\,
		\psi^v(dx_m^v)
	}{
		\int
		p^v(x_s,x_{s+1}^v)
		\prod_{m=s+1}^{s+q}
		g^v(x_m^v,Y_m^v)
		\prod_{w\in N(v)}
		p^w(x_{m},x_{m+1}^w)\,
		\psi^v(dx_m^v)
	}
\end{align*}
using Bayes' formula, the definition of $\mathbf{P}^\sigma$ (in the same 
form as the above definition of $\mathbf{\tilde P}$), and 
that $p^v(x,z^v)$ depends only on $x^{N(v)}$.  We now construct 
couplings $Q_{x,z}^J$ of $\gamma^J_x$ and $\gamma^J_z$ where $x,z$ differ 
only at the site $j=(k,w)\in I$.  We distinguish the following cases:
\begin{enumerate}
\item $k=s$, $w\in N(v)\backslash\{v\}$;
\item $k=s$, $w=v$;
\item $k\in \{s+1,\ldots,s+q\}$, $w\in \bigcup_{u\in 
N(v)}N(u)\backslash\{v\}$;
\item $k=s+q+1$, $w\in N(v)\backslash\{v\}$;
\item $k=s+q+1$, $w=v$.
\end{enumerate}
It is easily verified by inspection that $\gamma^J_x$ does not depend on
$x_k^w$ except in one of the above cases.  Thus when $j$ 
satisfies none of the 
above conditions, we can set $C_{ij}=0$ for $i\in J$.

\textbf{Case 1.}  
Note that
\begin{align*}
	&\gamma_x^J(A) \ge \\
	&\quad \varepsilon^2\, \frac{
		\int
		\mathbf{1}_A(x^J)\,q^v(x_s^v,x_{s+1}^v)
		\prod_{m=s+1}^{s+q}
		g^v(x_m^v,Y_m^v)
		\prod_{w\in N(v)}
		p^w(x_{m},x_{m+1}^w)\,
		\psi^v(dx_m^v)
	}{
		\int
		q^v(x_s^v,x_{s+1}^v)
		\prod_{m=s+1}^{s+q}
		g^v(x_m^v,Y_m^v)
		\prod_{w\in N(v)}
		p^w(x_{m},x_{m+1}^w)\,
		\psi^v(dx_m^v)
	},
\end{align*}
and the right hand side does not depend on $x_s^w$ for $w\ne v$.
Thus whenever $x,z\in\bbS$ satisfy $x^{I\backslash\{j\}}=
z^{I\backslash\{j\}}$ for $j=(s,w)$ with $w\in N(v)\backslash\{v\}$, we 
can construct a coupling $Q^J_{x,z}$ using Lemma \ref{lem:minorize} such 
that $C_{ij}\le 1-\varepsilon^2$ for every $i\in J$.

\textbf{Case 2.}  Define the transition kernels on $\bbX^v$
\begin{align*}
	&P_{k,x}(\omega,A) = \\
	&\quad \frac{
		\int
		\mathbf{1}_A(x^v_k)\,
		p^v(\omega x_{k-1}^{V\backslash\{v\}},x_k^v)
		\prod_{m=k}^{s+q}
		g^v(x_m^v,Y_m^v)
		\prod_{w\in N(v)}
		p^w(x_{m},x_{m+1}^w)\,
		\psi^v(dx_m^v)
	}{
		\int
		p^v(\omega x_{k-1}^{V\backslash\{v\}},x_k^v)
		\prod_{m=k}^{s+q}
		g^v(x_m^v,Y_m^v)
		\prod_{w\in N(v)}
		p^w(x_{m},x_{m+1}^w)\,
		\psi^v(dx_m^v)
	}
\end{align*}
for $k=s+1,\ldots,s+q$.  By construction, $P_{k,x}(x_{k-1}^v,dx_k^v)=
\gamma^J_x(dx_k^v|x_{s+1}^v,\ldots,x_{k-1}^v)$, so we are in the
setting of Lemma \ref{lem:markovminorize}.  Moreover, we can estimate
$$
	P_{k,x}(\omega,A) \ge 
	\varepsilon^2\delta^2\,\frac{
		\int
		\mathbf{1}_A(x^v_k)
		\prod_{m=k}^{s+q}
		g^v(x_m^v,Y_m^v)
		\prod_{w\in N(v)}
		p^w(x_{m},x_{m+1}^w)\,
		\psi^v(dx_m^v)
	}{
		\int
		\prod_{m=k}^{s+q}
		g^v(x_m^v,Y_m^v)
		\prod_{w\in N(v)}
		p^w(x_{m},x_{m+1}^w)\,
		\psi^v(dx_m^v)
	},
$$
where the right hand side does not depend on $\omega$.
Thus whenever $x,z\in\bbS$ satisfy $x^{I\backslash\{j\}}=
z^{I\backslash\{j\}}$ for $j=(s,v)$, we can construct a coupling 
$Q^J_{x,z}$ using Lemma \ref{lem:markovminorize} such 
that $C_{ij}\le (1-\varepsilon^2\delta^2)^{k-s}$ for $i=(k,v)$ with
$k=s+1,\ldots,s+q$.

\textbf{Case 3.} Fix $k\in\{s+1,\ldots,s+q\}$ and $u\ne v$. Note that
\begin{align*}
	&\gamma_x^J(A) \ge 
	\varepsilon^{2(\Delta+1)}\times\mbox{}\\
	&\quad 
	\frac{
		\int
		\mathbf{1}_A(x^J)\,p^v(x_s,x_{s+1}^v)
		\prod_{m=s+1}^{s+q}
		g^v(x_m^v,Y_m^v)
		\prod_{w\in N(v)}
		\beta^w_m(x_{m},x_{m+1}^w)\,
		\psi^v(dx_m^v)
	}{
		\int
		p^v(x_s,x_{s+1}^v)
		\prod_{m=s+1}^{s+q}
		g^v(x_m^v,Y_m^v)
		\prod_{w\in N(v)}
		\beta^w_m(x_{m},x_{m+1}^w)\,
		\psi^v(dx_m^v)
	}
\end{align*}
where we set $\beta^w_m(x_m,x^w_{m+1})=q^w(x_m^w,x^w_{m+1})$
if either $m=k$ or $m=k-1$ and $w=u$, and
$\beta^w_m(x_m,x^w_{m+1})=p^w(x_m,x^w_{m+1})$ otherwise.  The right hand
side of this expression does not depend on $x_k^u$ as the terms
$q^w(x_m^w,x^w_{m+1})$ for $w\ne v$ cancel in the numerator and 
denominator.
Thus whenever $x,z\in\bbS$ satisfy $x^{I\backslash\{j\}}=
z^{I\backslash\{j\}}$ for $j=(k,u)$, we
can construct a coupling $Q^J_{x,z}$ using Lemma \ref{lem:minorize} such 
that $C_{ij}\le 1-\varepsilon^{2(\Delta+1)}$ for every $i\in J$.

\textbf{Case 4.} Let $u\in N(v)\backslash v$.  Note that
\begin{align*}
	&\gamma_x^J(A) \ge \\
	&\quad \varepsilon^2\, \frac{
		\int
		\mathbf{1}_A(x^J)\,p^v(x_s^v,x_{s+1}^v)
		\prod_{m=s+1}^{s+q}
		g^v(x_m^v,Y_m^v)
		\prod_{w\in N(v)}
		\beta_m^w(x_{m},x_{m+1}^w)\,
		\psi^v(dx_m^v)
	}{
		\int
		p^v(x_s^v,x_{s+1}^v)
		\prod_{m=s+1}^{s+q}
		g^v(x_m^v,Y_m^v)
		\prod_{w\in N(v)}
		\beta_m^w(x_{m},x_{m+1}^w)\,
		\psi^v(dx_m^v)
	},
\end{align*}
where we set $\beta^w_m(x_m,x^w_{m+1})=q^w(x_m^w,x^w_{m+1})$
if $m=s+q$ and $w=u$, and we set 
$\beta^w_m(x_m,x^w_{m+1})=p^w(x_m,x^w_{m+1})$ 
otherwise.  The right hand side does not depend on $x_{s+q+1}^u$
as the term $q^u(x_{s+q}^u,x^u_{s+q+1})$ cancels in the numerator and
denominator.  Thus whenever $x,z\in\bbS$ satisfy $x^{I\backslash\{j\}}=
z^{I\backslash\{j\}}$ for $j=(s+q+1,u)$, we can construct a coupling 
$Q^J_{x,z}$ using Lemma \ref{lem:minorize} such that $C_{ij}\le 
1-\varepsilon^2$ for every $i\in J$.

\textbf{Case 5.}  Define for $k=s+1,\ldots,s+q$ the transition kernels on 
$\bbX^v$
\begin{align*}
	&P_{k,x}(\omega,A) = \\
	&\quad \frac{
		\int
		\mathbf{1}_A(x^v_k)\,
		p^v(x_s,x_{s+1}^v)
		\prod_{m=s+1}^{k}
		g^v(x_m^v,Y_m^v)
		\prod_{w\in N(v)}
		\beta^w_{m,\omega}(x_{m},x_{m+1}^w)\,
		\psi^v(dx_m^v)
	}{
		\int
		p^v(x_s,x_{s+1}^v)
		\prod_{m=s+1}^{k}
		g^v(x_m^v,Y_m^v)
		\prod_{w\in N(v)}
		\beta^w_{m,\omega}(x_{m},x_{m+1}^w)\,
		\psi^v(dx_m^v)
	},
\end{align*}
where we set $\beta^w_{m,\omega}(x_{m},x_{m+1}^w)=
p^v(x_k,\omega)$ if $m=k$ and $w=v$, and
$\beta^w_{m,\omega}(x_{m},x_{m+1}^w)=p^w(x_{m},x_{m+1}^w)$ otherwise.
By construction, $P_{k,x}(x_{k+1}^v,dx_k^v)=
\gamma^J_x(dx_k^v|x_{k+1}^v,\ldots,x_{s+q}^v)$, so we are in the
setting of Lemma \ref{lem:markovminorize}.  
Moreover, we can estimate
\begin{align*}
	&P_{k,x}(\omega,A) \ge \varepsilon^2\delta^2\times\mbox{} \\
	&\quad \frac{
		\int
		\mathbf{1}_A(x^v_k)\,
		p^v(x_s,x_{s+1}^v)
		\prod_{m=s+1}^{k}
		g^v(x_m^v,Y_m^v)
		\prod_{w\in N(v)}
		\beta^w_{m}(x_{m},x_{m+1}^w)\,
		\psi^v(dx_m^v)
	}{
		\int
		p^v(x_s,x_{s+1}^v)
		\prod_{m=s+1}^{k}
		g^v(x_m^v,Y_m^v)
		\prod_{w\in N(v)}
		\beta^w_{m}(x_{m},x_{m+1}^w)\,
		\psi^v(dx_m^v)
	},
\end{align*}
where $\beta^w_{m}(x_{m},x_{m+1}^w)=1$ if $m=k$ and $w=v$, and
$\beta^w_{m}(x_{m},x_{m+1}^w)=p^w(x_{m},x_{m+1}^w)$ otherwise.
Note that the right hand side does not depend on $\omega$.
Thus whenever $x,z\in\bbS$ satisfy $x^{I\backslash\{j\}}=
z^{I\backslash\{j\}}$ for $j=(s+q+1,v)$, we can construct a coupling 
$Q^J_{x,z}$ using Lemma \ref{lem:markovminorize} such 
that $C_{ij}\le (1-\varepsilon^2\delta^2)^{s+q+1-k}$ for $i=(k,v)$ with
$k=s+1,\ldots,s+q$.

We have now constructed coupled updates $Q^J_{x,z}$ for every pair
$x,z\in\bbS$ that differ only at one point.  Collecting the above bounds 
on $C_{ij}$, we can estimate
\begin{align*}
	&\sum_{(k',v')\in I} e^{\beta\{|k-k'|+d(v,v')\}}C_{(k,v)(k',v')} 
	\\
	&\qquad\mbox{}
	\le
	2e^{\beta(q+r)}(1-\varepsilon^2)\Delta + 
	e^{\beta(q+2r)}(1-\varepsilon^{2(\Delta+1)})\Delta^2q 
	\\ &\qquad\qquad\mbox{}+
	e^{\beta(k-s)}(1-\varepsilon^2\delta^2)^{k-s} +
	e^{\beta(s+q+1-k)}(1-\varepsilon^2\delta^2)^{s+q+1-k}
	\\
	&\qquad\mbox{}
	\le
	3q\Delta^2
	e^{\beta(q+2r)}(1-\varepsilon^{2(\Delta+1)}) +
	e^{\beta}(1-\varepsilon^2\delta^2) +
	e^{\beta q}(1-\varepsilon^2\delta^2)^q =: c
\end{align*}
whenever $(k,v)\in J$.  In the last line, we have used that
$\alpha^{x+1}+\alpha^{q-x}$ is a convex function of $x\in [0,q-1]$, and 
therefore attains its maximum on the endpoints $x=0,q-1$.

Up to this point we have considered an arbitrary block 
$J=J_l^v\in\mathcal{J}$ with $1<l<\lceil n/q\rceil$.  It is however 
evident that the identical proof holds for the boundary blocks 
$l=1,\lceil n/q\rceil$, except that for $l=1$ we only need to consider 
Cases 3--5 above and for $l=\lceil n/q\rceil$ we only need to consider 
Cases 1--3 above.  As all the estimates are otherwise identical, the
corresponding bounds on $C_{ij}$ are at most as large as those in the 
case $1<l<\lceil n/q\rceil$.   Thus
$$
	\|C\|_{\infty,\beta m}:=
	\max_{i\in I}\sum_{j\in I} e^{\beta m(i,j)}C_{ij}\le c,
$$
where we define the metric $m(i,j)=|k-k'|+d(v,v')$ for
$(k,v)\in I$ and $(k',v')\in I$.

Our next order of business is to construct couplings $\hat Q_x^J$ of
$\gamma^J_x$ and $\tilde\gamma^J_x$ and to estimate the  
coefficients $b_i$.  To this end, let us first note that
$h_n^K(x,z^{\partial K})$ depends only on $x^{\partial^2K}$, where
$$
	\partial^2K := \bigcup_{w\in\partial K}N(w)\cap K
$$
is the subset of vertices in $K$ that can interact with vertices outside 
$K$ in two time steps.   It is easily seen that 
$\gamma^J_x=\tilde\gamma^J_x$, and that we can therefore choose $b_i=0$ 
for $i\in J$, unless $J=J_l^v$ with $v\in\partial^2K$ for some 
$K\in\mathcal{K}$.  In the latter case we obtain by Bayes' formula
\begin{align*}
	&\tilde\gamma_x^J(A) = \\
	&\frac{
		\int
		\mathbf{1}_A(x^J)
		\prod_{m=s}^{s+q}
		g^v(x_m^v,Y_m^v) \,
		h^K_{m+1}(x_m,x_{m+1}^{\partial K})
		\prod_{w\in N(v)\cap K\backslash\partial K}
		p^w(x_{m},x_{m+1}^w)\,
		\psi(dx^J)
	}{
		\int
		\prod_{m=s}^{s+q}
		g^v(x_m^v,Y_m^v) \,
		h^K_{m+1}(x_m,x_{m+1}^{\partial K})
		\prod_{w\in N(v)\cap K\backslash\partial K}
		p^w(x_{m},x_{m+1}^w)\,
		\psi(dx^J)
	}
\end{align*}
for $1<l<\lceil n/q\rceil$, where $s=(l-1)q$ and
$\psi(dx^J)=\bigotimes_{(k,v)\in J}\psi^v(dx_k^v)$.  Note that
$$
	\prod_{w\in N(v)\backslash(K\backslash\partial K)}p^w(x,z^w) 
	\ge
	\varepsilon^\Delta
	\prod_{w\in N(v)\backslash(K\backslash\partial K)}q^w(x^w,z^w),
$$
while
$$
	h^K_m(x,z^{\partial K}) \ge
	\varepsilon^\Delta
	\prod_{w\in N(v)\cap\partial K}q^w(x^w,z^w)
	\int \tilde\pi_{m-1}^\sigma(d\omega)
	\prod_{w\in\partial K\backslash N(v)}p^w(x^K\omega^{V\backslash 
	K},z^w).
$$
We can therefore estimate 
$\gamma_x^J(A)\ge \varepsilon^{2(q+1)\Delta}\Gamma(A)$
and
$\tilde\gamma_x^J(A)\ge \varepsilon^{2(q+1)\Delta}\Gamma(A)$ with
\begin{align*}
	&\Gamma(A)=\mbox{}\\
	&\frac{
		\int
		\mathbf{1}_A(x^J)
		\prod_{m=s}^{s+q}
		g^v(x_m^v,Y_m^v) \,
		\beta(x_m^v,x_{m+1}^v)
		\prod_{w\in N(v)\cap K\backslash\partial K}
		p^w(x_{m},x_{m+1}^w)\,
		\psi(dx^J)
	}{
		\int
		\prod_{m=s}^{s+q}
		g^v(x_m^v,Y_m^v) \,
		\beta(x_m^v,x_{m+1}^v)
		\prod_{w\in N(v)\cap K\backslash\partial K}
		p^w(x_{m},x_{m+1}^w)\,
		\psi(dx^J)
	},
\end{align*}
where $\beta(x,z)=q^v(x,z)$ if $v\in\partial K$ and
$\beta(x,z)=1$ if $v\in\partial^2K\backslash\partial K$.
Thus we can construct a coupling $\hat Q_x^J$ using
Lemma \ref{lem:minorize} such that $b_i\le 1-\varepsilon^{2(q+1)\Delta}$
for all $i\in J$ in the case $1<l<\lceil n/q\rceil$.
The same conclusion follows for $l=1,\lceil n/q\rceil$ by 
the identical proof.

We are now ready to put everything together.
As $\|\cdot\|_{\infty,\beta m}$ is a matrix norm, we have
$$
	\|D\|_{\infty,\beta m}\le \sum_{k=0}^\infty\|C\|_{\infty,\beta m}^k
	\le \frac{1}{1-c}<\infty.
$$
Thus $D<\infty$, to we can apply the comparison theorem.  Moreover,
$$
	\sup_{i\in J}\sum_{j\in J'}D_{ij} =
	\sup_{i\in J}e^{-\beta m(i,J')}
	\sum_{j\in J'}e^{\beta m(i,J')}D_{ij} \le
	e^{-\beta m(J,J')}\|D\|_{\infty,\beta m}.
$$
Thus we obtain
\begin{align*}
	\|\pi_n^\sigma-\tilde\pi_n^\sigma\|_J &\le
	2(1-\varepsilon^{2(q+1)\Delta})
	\sum_{i\in \{n\}\times J}
	\sum_{j\in\{1,\ldots,n\}\times\partial^2K} D_{ij} 
	\\ &\le
	\frac{2}{1-c}\,(1-\varepsilon^{2(q+1)\Delta})\,
	\card J\,e^{-\beta d(J,\partial^2K)}.
\end{align*}
But clearly $d(J,\partial^2K)\ge d(J,\partial K)-r$, and the proof
is complete.
\end{proof}

\begin{remark}
In the proof of Theorem \ref{thm:bias} (and similarly for Theorem 
\ref{thm:variance} below), we apply the comparison theorem with a
nonoverlapping cover $\{(l-1)q+1,\ldots,lq\wedge 
n\}$, $l\le\lceil n/q\rceil$. Working instead with overlapping blocks
$\{s+1,\ldots,s+q\}$, $s\le n-q$ would give somewhat better estimates
at the expense of even more tedious computations.
\end{remark}

\subsection{Bounding the variance}
\label{sec:variance}

We now turn to the problem of bounding the variance term 
$\tnorm{\tilde\pi_n^\sigma-\hat\pi_n^\sigma}_J$.  We will follow the basic 
approach taken in \cite{RvH13}, where a detailed discussion of the 
requisite ideas can be found.  In this section we develop the necessary 
changes to the proof in \cite{RvH13}.

At the heart of the proof of the variance bound lies a stability result 
for the block filter \cite[Proposition 4.15]{RvH13}.  This result must be 
modified in the present setting to account for the different assumptions 
on the spatial and temporal correlations.  This will be done next,
using the generalized comparison Theorem \ref{thm:main} much as in the 
proof of Theorem \ref{thm:bias}.

\begin{prop}
\label{prop:bfstab}
Suppose there exist $0<\varepsilon,\delta<1$ such that 
\begin{align*}
	\varepsilon q^v(x^v,z^v) &\le p^v(x,z^v)\le
	\varepsilon^{-1}q^v(x^v,z^v), \\
	\delta &\le q^v(x^v,z^v)\le \delta^{-1}
\end{align*}
for every $v\in V$ and $x,z\in\bbX$, where 
$q^v:\bbX^v\times\bbX^v\to\mathbb{R}_+$ is a transition density with
respect to $\psi^v$.  Suppose also that we can choose $q\in\mathbb{N}$ and 
$\beta>0$ such that
$$
	c:=
	3q\Delta^2e^{\beta q}(1-\varepsilon^{2(\Delta+1)}) +
	e^\beta(1-\varepsilon^2\delta^2) +
	e^{\beta q}(1-\varepsilon^2\delta^2)^q < 1.
$$
Then we have
$$
	\|\mathsf{\tilde F}_n\cdots\mathsf{\tilde F}_{s+1}\delta_\sigma-
	\mathsf{\tilde F}_n\cdots\mathsf{\tilde F}_{s+1}\delta_{\tilde\sigma}
	\|_J \le
	\frac{2}{1-c}\,
	\card J\,
	e^{-\beta (n-s)}
$$
for every $s<n$, $\sigma,\tilde\sigma\in\bbX$, $K\in\mathcal{K}$ and 
$J\subseteq K$.
\end{prop}

\begin{proof}
We fix throughout the proof $n>0$, $K\in\mathcal{K}$, and $J\subseteq K$.
We will also assume for notational simplicity that $s=0$.  As 
$\mathsf{\tilde F}_k$ differ for different $k$ only by their dependence on 
different observations $Y_k$, and as the conclusion of the Proposition is
independent of the observations, the conclusion for $s=0$ extends 
trivially to any $s<n$.

As in Theorem \ref{thm:bias}, the idea behind the proof is to introduce a 
Markov random field $\rho$ of which the block filter is a marginal, 
followed by an application of the generalized comparison theorem.  
Unfortunately, the construction in the proof of Theorem \ref{thm:bias} is 
not appropriate in the present setting, as there all the local 
interactions depend on the initial condition $\sigma$. That was irrelevant 
in Theorem \ref{thm:bias} where the initial condition was fixed, but is 
fatal in the present setting where we aim to understand a perturbation to 
the initial condition.  Instead, we will use a more elaborate construction 
of $\rho$ introduced in \cite{RvH13}, called the \emph{computation tree}. 
We begin by recalling this construction.

Define for $K'\in\mathcal{K}$ the block neighborhood
$N(K') := \{K''\in\mathcal{K}:d(K',K'')\le r\}$
(we recall that $\card N(K')\le\Delta_\mathcal{K}$).  We can evidently
write
$$
	\mathsf{B}^{K'}\mathsf{\tilde F}_s
	\bigotimes_{K''\in\mathcal{K}}\mu^{K''} =
	\mathsf{C}_s^{K'}\mathsf{P}^{K'}
	\bigotimes_{K''\in N(K')}\mu^{K''},
$$
where we define for any probability $\eta$ on $\bbX^{K'}$
$$
	(\mathsf{C}_s^{K'}\eta)(A) :=
	\frac{
		\int \mathbf{1}_A(x^{K'})\prod_{v\in K'}g^v(x^v,Y_s^v)\,
		\eta(dx^{K'})
	}{
		\int \prod_{v\in K'}g^v(x^v,Y_s^v)\,\eta(dx^{K'})
	},
$$
and for any probability $\eta$ on $\bbX^{\bigcup_{K''\in N(K')}K''}$
$$
	(\mathsf{P}^{K'}\eta)(A) :=
	\int \mathbf{1}_A(x^{K'})\prod_{v\in K'}p^v(z,x^v)\,\psi^v(dx^v)\,
	\eta(dz).
$$
Iterating this identity yields
\begin{multline*}
	\mathsf{B}^K\mathsf{\tilde F}_n\cdots\mathsf{\tilde F}_1\delta_\sigma
	= \\
	\mathsf{C}_n^K\mathsf{P}^K\bigotimes_{K_{n-1}\in N(K)}
	\bigg[
	\cdots\,
	\mathsf{C}_2^{K_2}\mathsf{P}^{K_2}
	\bigotimes_{K_1\in N(K_2)}
	\bigg[
	\mathsf{C}_1^{K_1}\mathsf{P}^{K_1}
	\bigotimes_{K_0\in N(K_1)}
	\delta_{\sigma^{K_0}}
	\bigg]
	\cdots
	\bigg].
\end{multline*}
The nested products can be naturally viewed as defining a tree.  

To formalize this idea, define the tree index set (we will 
write $K_n:=K$ for simplicity)
$$
	T := \{[K_u\cdots K_{n-1}]:0\le u<n,~K_s\in N(K_{s+1})
	\mbox{ for }u\le s<n\}\cup\{[\varnothing]\}.
$$
The root of the tree $[\varnothing]$ represents the block $K$ at time $n$, 
while $[K_u\cdots K_{n-1}]$ represents a duplicate of the block $K_u$ at 
time $u$ that affects the root along the branch $K_u\to K_{u+1}\to
\cdots\to K_{n-1}\to K$.  The set of sites corresponding to the 
computation tree is
$$
	I = \{[K_u\cdots K_{n-1}]v:[K_u\cdots K_{n-1}]\in T,~v\in K_u\}
	\cup\{[\varnothing]v:v\in K\},
$$
and the corresponding configuration space is $\bbS=\prod_{i\in I}\bbX^i$
with $\bbX^{[t]v}:=\bbX^v$.  
The following tree notation will be used throughout the proof.
Define for vertices of the 
tree $T$ the depth $d([K_u\cdots K_{n-1}]):=u$ and $d([\varnothing]):=n$.
For every site $[t]v\in I$, we define the associated vertex $v(i):=v$ and 
depth $d(i):=d([t])$.  Define also the sets $I_+:=\{i\in I:d(i)>0\}$
and $T_0:=\{[t]\in T:d([t])=0\}$ of non-leaf sites and leaf vertices, 
respectively.  Define
$$
	c([K_u\cdots K_{n-1}]v) :=
	\{[K_{u-1}\cdots K_{n-1}]v':K_{u-1}\in N(K_u),~v'\in N(v)\},
$$
and similarly for $c([\varnothing]v)$: that is, $c(i)$ is the set of 
children of the site $i\in I$ in the computation tree.  Finally, we will 
frequently identify a tree vertex $[K_u\cdots K_{n-1}]\in T$ with the 
corresponding set of sites $\{[K_u\cdots K_{n-1}]v:v\in K_u\}$, and 
analogously for $[\varnothing]$.

Having introduced the tree structure, we now define probability measures
$\rho,\tilde\rho$ on $\bbS$ by
\begin{align*}
	\rho(A) &= 
	\frac{
	\int \mathbf{1}_A(x)\prod_{i\in I_+}p^{v(i)}(x^{c(i)},x^i)\,
		g^{v(i)}(x^i,Y^i)\,
		\psi^{v(i)}(dx^i)
		\prod_{[t]\in T_0}\delta_{\sigma^{[t]}}(dx^{[t]})
	}{
	\int \prod_{i\in I_+}p^{v(i)}(x^{c(i)},x^i)\,
		g^{v(i)}(x^i,Y^i)\,
		\psi^{v(i)}(dx^i)
		\prod_{[t]\in T_0}\delta_{\sigma^{[t]}}(dx^{[t]})
	},\\
	\tilde\rho(A) &=
	\frac{
	\int \mathbf{1}_A(x)\prod_{i\in I_+}p^{v(i)}(x^{c(i)},x^i)\,
		g^{v(i)}(x^i,Y^i)\,
		\psi^{v(i)}(dx^i)
		\prod_{[t]\in T_0}\delta_{\tilde\sigma^{[t]}}(dx^{[t]})
	}{
	\int \prod_{i\in I_+}p^{v(i)}(x^{c(i)},x^i)\,
		g^{v(i)}(x^i,Y^i)\,
		\psi^{v(i)}(dx^i)
		\prod_{[t]\in T_0}\delta_{\tilde\sigma^{[t]}}(dx^{[t]})
	},
\end{align*}
where we write $\sigma^{[K_0\cdots K_{n-1}]}:=\sigma^{K_0}$ and 
$Y^i:=Y_{d(i)}^{v(i)}$ for simplicity.  Then, by construction, the measure
$\mathsf{B}^K\mathsf{\tilde F}_n\cdots\mathsf{\tilde F}_1\delta_\sigma$ 
coincides with the marginal of $\rho$ on the root of the computation tree, 
while $\mathsf{B}^K\mathsf{\tilde F}_n\cdots\mathsf{\tilde 
F}_1\delta_{\tilde\sigma}$ coincides with the marginal of $\tilde\rho$ 
on the root of the tree: this is easily seen by expanding the above nested 
product identity.  In particular, we obtain
$$
	\|\mathsf{\tilde F}_n\cdots\mathsf{\tilde F}_1\delta_\sigma-
	\mathsf{\tilde F}_n\cdots\mathsf{\tilde F}_1\delta_{\tilde\sigma}\|_J
	= \|\rho-\tilde\rho\|_{[\varnothing]J},
$$
and we aim to apply the comparison theorem to estimate this quantity.

The construction of the computation tree that we have just given is 
identical to the construction in \cite{RvH13}.  We deviate from the proof 
of \cite{RvH13} from this point onward, since we must use Theorem 
\ref{thm:main} instead of the classical Dobrushin comparison theorem to 
account for the distinction between temporal and spatial correlations in 
the present setting.

Fix $q\ge 1$. In analogy with the proof of Theorem \ref{thm:bias}, we 
consider a cover $\mathcal{J}$ consisting of blocks of sites $i\in I$ such 
that $(l-1)q<d(i)\le lq\wedge n$ and $v(i)=v$.  In the present setting, 
however, the same vertex $v$ is duplicated many times in the tree, so that 
we end up with many disconnected blocks of different lengths.  
To keep track of these blocks, define
$$
	I_0 := \{i\in I:d(i)=0\},\qquad
	I_l := \{i\in I:d(i)=(l-1)q+1\}
$$
for $1\le l\le \lceil n/q\rceil$, and let
$$
	\ell([K_u,\ldots,K_{n-1}]v) :=
	\max\{s\ge u:K_u=K_{u+1}=\cdots=K_s\}.
$$
We now define the cover $\mathcal{J}$ as
$$
	\mathcal{J} = \{J_l^i:0\le l\le \lceil n/q\rceil,~i\in I_l\},
$$
where
$$
	J_0^i := \{i\},\qquad
	J_l^i := \{[K_u\cdots K_{n-1}]v:
	(l-1)q+1\le u\le lq\wedge\ell(i)\}
$$
for $i=[K_{(l-1)q+1}\cdots K_{n-1}]v\in I_l$ and $1\le l\le \lceil 
n/q\rceil$.  It is easily seen that $\mathcal{J}$ is in fact a partition 
of of the computation tree $I$ into linear segments.

Having defined the cover $\mathcal{J}$, we must now consider a 
suitable coupled update rule. We will choose the natural local updates 
$\gamma^J_x(dz^J) = \rho(dz^J|x^{I\backslash J})$ and 
$\tilde\gamma^J_x(dz^J) = \tilde\rho(dz^J|x^{I\backslash J})$, with the 
coupled updates $Q_{x,z}^J$ and $\hat Q_x^J$ to be constructed below.  
Then Theorem \ref{thm:main} yields
$$
	\|\mathsf{\tilde F}_n\cdots\mathsf{\tilde F}_1\delta_\sigma-
	\mathsf{\tilde F}_n\cdots\mathsf{\tilde F}_1\delta_{\tilde\sigma}\|_J
	= \|\rho-\tilde\rho\|_{[\varnothing]J} \le
	2\sum_{i\in [\varnothing]J}\sum_{j\in I} D_{ij}\,b_j
$$
provided that $D=\sum_{k=0}^\infty C^k<\infty$ (cf.\ Corollary 
\ref{cor:uniq}), where
$$
	C_{ij} = \sup_{\substack{
				x,z\in\bbS:\\
				x^{I\backslash\{j\}}=z^{I\backslash\{j\}}
		}}
		\int \mathbf{1}_{\omega_i\ne\omega_i'}
		\,Q_{x,z}^{J(i)}(d\omega,d\omega'),\qquad
	b_i = \sup_{x\in\bbS}
	\int \mathbf{1}_{\omega_i\ne\omega_i'}
                \,\hat Q_{x}^{J(i)}(d\omega,d\omega'),
$$
and where we write $J(i)$ for the unique block $J\in\mathcal{J}$ that
contains $i\in I$.  To put this bound to good use, we must introduce 
coupled updates $Q_{x,z}^J$ and $\hat Q_x^J$ and
estimate $C_{ij}$ and $b_j$.

Let us fix until further notice a block $J=J_l^i\in\mathcal{J}$ with
$i=[K_{(l-1)q+1}\cdots K_{n-1}]v\in I_l$ and $1\le l\le\lceil   
n/q\rceil$. From the definition of $\rho$, we can compute explicitly 
\begin{align*}
	&\gamma^J_x(A) = \mbox{}\\
	&\frac{
	\int \mathbf{1}_A(x^J)\,p^v(x^{c(i)},x^i)
	\prod_{a\in I_+: J\cap c(a)\ne\varnothing}
	p^{v(a)}(x^{c(a)},x^a)
	\prod_{b\in J} g^v(x^{b},Y^b)\,
	\psi^v(dx^{b})
	}{
	\int p^v(x^{c(i)},x^i)
	\prod_{a\in I_+: J\cap c(a)\ne\varnothing}
	p^{v(a)}(x^{c(a)},x^a)
	\prod_{b\in J} g^v(x^{b},Y^b)\,
	\psi^v(dx^{b})
	}
\end{align*}
using the Bayes formula.  We now proceed to construct couplings 
$Q^J_{x,z}$ of $\gamma^J_x$ and $\gamma^J_z$ for $x,z\in\bbS$ that differ 
only at the site $j\in I$.  We distinguish the following cases:
\begin{enumerate}
\item $d(j)=(l-1)q$ and $v(j)\ne v$;
\item $d(j)=(l-1)q$ and $v(j)=v$;
\item $(l-1)q+1\le d(j)\le lq\wedge\ell(i)$ and $v(j)\ne v$;
\item $d(j)=lq\wedge\ell(i)+1$ and $v(j)\ne v$;
\item $d(j)=lq\wedge\ell(i)+1$ and $v(j)=v$.
\end{enumerate}
It is easily seen that $\gamma^J_x$ does not depend on $x^j$ except in
one of the above cases.  Thus when $j$ satisfies none of the above 
conditions, we can set $C_{aj}=0$ for $a\in J$.

\textbf{Case 1.}  In this case, we must have $j\in c(i)$ with
$v(j)\ne v$.  Note that
\begin{align*}
	&\gamma^J_x(A) \ge \mbox{}\\
	&\varepsilon^2\,
	\frac{
	\int \mathbf{1}_A(x^J)\,q^v(x^{i_-},x^i)
	\prod_{a\in I_+: J\cap c(a)\ne\varnothing}
	p^{v(a)}(x^{c(a)},x^a)
	\prod_{b\in J} g^v(x^{b},Y^b)\,
	\psi^v(dx^{b})
	}{
	\int q^v(x^{i_-},x^i)
	\prod_{a\in I_+: J\cap c(a)\ne\varnothing}
	p^{v(a)}(x^{c(a)},x^a)
	\prod_{b\in J} g^v(x^{b},Y^b)\,
	\psi^v(dx^{b})
	},
\end{align*}
where we define $i_-\in c(i)$ to be the (unique) child of $i$ such that
$v(i_-)=v(i)$.  As the right hand side does not depend on $x^j$, we can 
construct a coupling $Q^J_{x,z}$ using Lemma \ref{lem:minorize} such that 
$C_{aj}\le 1-\varepsilon^2$ for every $a\in J$ and $x,z\in\bbS$ such that
$x^{I\backslash\{j\}}=z^{I\backslash\{j\}}$.

\textbf{Case 2.} In this case we have $j=i_-$.  Let us write 
$J=\{i_1,\ldots,i_u\}$ where $u=\card J$ and $d(i_k)=(l-1)q+k$ for 
$k=1,\ldots,u$.  Thus $i_1=i$, and we define $i_0=i_-$.  Let us also
write $\tilde J_k = \{i_k,\ldots,i_u\}$.  Then we can define
the transition kernels on $\bbX^v$
\begin{align*}
	&P_{k,x}(\omega,A) = \mbox{}\\
	&\frac{
	\int \mathbf{1}_A(x^{i_k})\,
	p^v(\omega x^{c(i_k)\backslash i_{k-1}},x^{i_k})
	\prod_{\tilde J_k\cap c(a)\ne\varnothing}
	p^{v(a)}(x^{c(a)},x^a)
	\prod_{b\in \tilde J_k} g^v(x^{b},Y^b)\,
	\psi^v(dx^{b})
	}{
	\int 
	p^v(\omega x^{c(i_k)\backslash i_{k-1}},x^{i_k})
	\prod_{\tilde J_k\cap c(a)\ne\varnothing}
	p^{v(a)}(x^{c(a)},x^a)
	\prod_{b\in \tilde J_k} g^v(x^{b},Y^b)\,
	\psi^v(dx^{b})
	}
\end{align*}
for $k=1,\ldots,u$.  By construction, 
$P_{k,x}(x^{i_{k-1}},dx^{i_k})=\gamma^J_x(dx^{i_k}|x^{i_1},\ldots,x^{i_{k-1}})$,
so we are in the setting of Lemma \ref{lem:markovminorize}.
Moreover, we can estimate
$$
	P_{k,x}(\omega,A) \ge \varepsilon^2\delta^2\,
	\frac{
	\int \mathbf{1}_A(x^{i_k})
	\prod_{\tilde J_k\cap c(a)\ne\varnothing}
	p^{v(a)}(x^{c(a)},x^a)
	\prod_{b\in \tilde J_k} g^v(x^{b},Y^b)\,
	\psi^v(dx^{b})
	}{
	\int 
	\prod_{\tilde J_k\cap c(a)\ne\varnothing}
	p^{v(a)}(x^{c(a)},x^a)
	\prod_{b\in \tilde J_k} g^v(x^{b},Y^b)\,
	\psi^v(dx^{b})
	}.
$$
Thus whenever 
$x,z\in\bbS$ satisfy $x^{I\backslash\{j\}}=z^{I\backslash\{j\}}$, we can 
construct a coupling $Q_{x,z}^J$ using Lemma
\ref{lem:markovminorize} such that $C_{i_kj}\le 
(1-\varepsilon^2\delta^2)^k$ for every $k=1,\ldots,u$.

\textbf{Case 3.}  In this case we have $j\in \bigcup_{a\in 
I_+:J\cap c(a)\ne\varnothing}c(a)$ or $J\cap c(j)\ne\varnothing$,
with $v(j)\ne v$.  Let us note for future reference that there are
at most $q\Delta^2$ such sites $j$.  Now note that
\begin{align*}
	&\gamma^J_x(A) \ge  \varepsilon^{2(\Delta+1)}\times\mbox{}\\
	&\frac{
	\int \mathbf{1}_A(x^J)\,p^v(x^{c(i)},x^i)
	\prod_{a\in I_+: J\cap c(a)\ne\varnothing}
	\beta^{a}(x^{c(a)},x^a)
	\prod_{b\in J} g^v(x^{b},Y^b)\,
	\psi^v(dx^{b})
	}{
	\int p^v(x^{c(i)},x^i)
	\prod_{a\in I_+: J\cap c(a)\ne\varnothing}
	\beta^{a}(x^{c(a)},x^a)
	\prod_{b\in J} g^v(x^{b},Y^b)\,
	\psi^v(dx^{b})
	},
\end{align*}
where we have defined $\beta^a(x^{c(a)},x^a)=q^{v(a)}(x^{a_-},x^a)$ 
whenever $j=a$ or $j\in c(a)$, and
$\beta^a(x^{c(a)},x^a)=p^{v(a)}(x^{c(a)},x^a)$ otherwise.
The right hand side of this expression does not depend on $x^j$ as the 
terms $q^{v(a)}(x^{a_-},x^a)$ for $v(a)\ne v$ cancel in the numerator and
denominator. Thus whenever $x,z\in\bbS$ satisfy $x^{I\backslash\{j\}}=
z^{I\backslash\{j\}}$, we can construct a coupling $Q^J_{x,z}$ using
Lemma \ref{lem:minorize} such that $C_{aj}\le 1-\varepsilon^{2(\Delta+1)}$ 
for every $a\in J$.

\textbf{Case 4.}  In this case $J\cap c(j)\ne\varnothing$ with $v(j)\ne v$.
Note that
\begin{align*}
	&\gamma^J_x(A) \ge  \mbox{}\\
	&\varepsilon^2\,\frac{
	\int \mathbf{1}_A(x^J)\,p^v(x^{c(i)},x^i)
	\prod_{a\in I_+: J\cap c(a)\ne\varnothing}
	\beta^{a}(x^{c(a)},x^a)
	\prod_{b\in J} g^v(x^{b},Y^b)\,
	\psi^v(dx^{b})
	}{
	\int p^v(x^{c(i)},x^i)
	\prod_{a\in I_+: J\cap c(a)\ne\varnothing}
	\beta^{a}(x^{c(a)},x^a)
	\prod_{b\in J} g^v(x^{b},Y^b)\,
	\psi^v(dx^{b})
	},
\end{align*}
where $\beta^a(x^{c(a)},x^a)=q^{v(a)}(x^{a_-},x^a)$ 
when $j=a$, and $\beta^a(x^{c(a)},x^a)=p^{v(a)}(x^{c(a)},x^a)$ otherwise.
The right hand side does not depend on $x^j$ as the 
term $q^{v(j)}(x^{j_-},x^j)$ cancels in the numerator and
denominator. Thus whenever $x,z\in\bbS$ satisfy $x^{I\backslash\{j\}}=
z^{I\backslash\{j\}}$, we can construct a coupling $Q^J_{x,z}$ using
Lemma \ref{lem:minorize} such that $C_{aj}\le 1-\varepsilon^2$ 
for every $a\in J$.

\textbf{Case 5.} In this case we have $j_-\in J$.  Note that the existence 
of such $j$ necessarily implies that $\ell(i)>lq$ by the definition of 
$J$.  We can therefore write $J=\{i_1,\ldots,i_q\}$ where $d(i_k)=lq-k+1$ 
for $k=1,\ldots,q$, and we define $i_0=j$.  Let us also define the sets 
$\tilde J_k = \{i_k,\ldots,i_q\}$.  Then we can define the transition 
kernels on $\bbX^v$
\begin{align*}
	&P_{k,x}(\omega,A) = \mbox{}\\
	&\frac{
	\int \mathbf{1}_A(x^{i_k})\,
	p^v(x^{c(i_q)},x^{i_q})
	\prod_{a\in I_+:\tilde J_k\cap c(a)\ne\varnothing}
	\beta^{a}_\omega(x^{c(a)},x^a)
	\prod_{b\in \tilde J_k} g^v(x^{b},Y^b)\,
	\psi^v(dx^{b})
	}{
	\int 
	p^v(x^{c(i_q)},x^{i_q})
	\prod_{a\in I_+:\tilde J_k\cap c(a)\ne\varnothing}
	\beta^{a}_\omega(x^{c(a)},x^a)
	\prod_{b\in \tilde J_k} g^v(x^{b},Y^b)\,
	\psi^v(dx^{b})
	}
\end{align*}
for $k=1,\ldots,q$, where 
$\beta^a_\omega(x^{c(a)},x^a)=p^{v}(x^{c(a)},\omega)$ if $a=i_{k-1}$ and 
$\beta^a_\omega(x^{c(a)},x^a)=p^{v(a)}(x^{c(a)},x^a)$ otherwise. By 
construction
$P_{k,x}(x^{i_{k-1}},dx^{i_k})=\gamma^J_x(dx^{i_k}|x^{i_1},\ldots,x^{i_{k-1}})$, 
so we are in the setting of Lemma \ref{lem:markovminorize}. Moreover, we 
can estimate
\begin{align*}
	&P_{k,x}(\omega,A) \ge \varepsilon^2\delta^2\times \mbox{}\\
	&\frac{
	\int \mathbf{1}_A(x^{i_k})\,
	p^v(x^{c(i_q)},x^{i_q})
	\prod_{a\in I_+:\tilde J_k\cap c(a)\ne\varnothing}
	\beta^{a}(x^{c(a)},x^a)
	\prod_{b\in \tilde J_k} g^v(x^{b},Y^b)\,
	\psi^v(dx^{b})
	}{
	\int 
	p^v(x^{c(i_q)},x^{i_q})
	\prod_{a\in I_+:\tilde J_k\cap c(a)\ne\varnothing}
	\beta^{a}(x^{c(a)},x^a)
	\prod_{b\in \tilde J_k} g^v(x^{b},Y^b)\,
	\psi^v(dx^{b})
	},
\end{align*}
where $\beta^a(x^{c(a)},x^a)=1$ if $a=i_{k-1}$ 
and $\beta^a(x^{c(a)},x^a)=p^{v(a)}(x^{c(a)},x^a)$ otherwise.
Thus whenever 
$x,z\in\bbS$ satisfy $x^{I\backslash\{j\}}=z^{I\backslash\{j\}}$, we can 
construct a coupling $Q_{x,z}^J$ using Lemma
\ref{lem:markovminorize} such that $C_{i_kj}\le 
(1-\varepsilon^2\delta^2)^k$ for every $k=1,\ldots,q$.

We have now constructed coupled updates $Q^J_{x,z}$ for every pair 
$x,z\in\bbS$ that differ only at one point.  Collecting the above bounds 
on the matrix $C$, we can estimate
$$
	\sum_{j\in I}e^{\beta|d(a)-d(j)|}C_{aj} \le
	3q\Delta^2e^{\beta q}(1-\varepsilon^{2(\Delta+1)}) +
	e^\beta(1-\varepsilon^2\delta^2) +
	e^{\beta q}(1-\varepsilon^2\delta^2)^q =: c
$$
whenever $a\in J$, where we have used the convexity of the function
$\alpha^{x+1}+\alpha^{q-x}$.

Up to this point we have considered an arbitrary block 
$J=J_l^i\in\mathcal{J}$ with $1\le l\le\lceil n/q\rceil$.  However, in the 
remaining case $l=0$ it is easily seen that $\gamma^J_x=\delta_{\sigma^J}$ 
does not depend on $x$, so we can evidently set $C_{aj}=0$ for $a\in J$.  
Thus we have shown that
$$
	\|C\|_{\infty,\beta m} := \max_{i\in I}\sum_{j\in I}
	e^{\beta m(i,j)}C_{ij} \le c,
$$
where we define the pseudometric $m(i,j)=|d(i)-d(j)|$.  On the other hand, 
in the present setting it is evident that $\gamma^J_x=\tilde\gamma^J_x$ 
whenever $J=J_l^i\in\mathcal{J}$ with $1\le l\le\lceil n/q\rceil$.  We can 
therefore choose couplings $\hat Q_x^J$ such that
$b_i\le \mathbf{1}_{d(i)=0}$ for all $i\in I$.  Substituting into the
comparison theorem and arguing as in the proof of Theorem \ref{thm:bias}
yields the estimate
$$
	\|\mathsf{\tilde F}_n\cdots\mathsf{\tilde F}_1\delta_\sigma-
	\mathsf{\tilde F}_n\cdots\mathsf{\tilde F}_1\delta_{\tilde\sigma}\|_J
	\le
	\frac{2}{1-c}\card J\,e^{-\beta n}.
$$
Thus the proof is complete.
\end{proof}

Proposition \ref{prop:bfstab} provides control of the block filter as a 
function of time but not as a function of the initial conditions.  The 
dependence on the initial conditions can however be incorporated \emph{a 
posteriori} as in the proof of \cite[Proposition 4.17]{RvH13}.  This 
yields the following result, which forms the basis for the proof of 
Theorem \ref{thm:variance} below.

\begin{cor}[\rm Block filter stability]
\label{cor:bfstab}
Suppose there exist $0<\varepsilon,\delta<1$ such that 
\begin{align*}
	\varepsilon q^v(x^v,z^v) &\le p^v(x,z^v)\le
	\varepsilon^{-1}q^v(x^v,z^v), \\
	\delta &\le q^v(x^v,z^v)\le \delta^{-1}
\end{align*}
for every $v\in V$ and $x,z\in\bbX$, where 
$q^v:\bbX^v\times\bbX^v\to\mathbb{R}_+$ is a transition density with
respect to $\psi^v$.  Suppose also that we can choose $q\in\mathbb{N}$ and 
$\beta>0$ such that
$$
	c:=
	3q\Delta^2e^{\beta q}(1-\varepsilon^{2(\Delta+1)}) +
	e^\beta(1-\varepsilon^2\delta^2) +
	e^{\beta q}(1-\varepsilon^2\delta^2)^q < 1.
$$
Let $\mu$ and $\nu$ be (possibly random) probability measures on
$\bbX$ of the form
$$
	\mu = \bigotimes_{K\in\mathcal{K}}\mu^K,\qquad\qquad
	\nu = \bigotimes_{K\in\mathcal{K}}\nu^K.
$$
Then we have
$$
	\|\mathsf{\tilde F}_n\cdots\mathsf{\tilde F}_{s+1}\mu-
	\mathsf{\tilde F}_n\cdots\mathsf{\tilde F}_{s+1}\nu
	\|_J \le
	\frac{2}{1-c}\,
	\card J\,
	e^{-\beta (n-s)},
$$
as well as
\begin{multline*}
	\mathbf{E}[
	\|\mathsf{\tilde F}_n\cdots\mathsf{\tilde F}_{s+1}\mu-
	\mathsf{\tilde F}_n\cdots\mathsf{\tilde F}_{s+1}\nu
	\|_J^2]^{1/2} \\
	\mbox{}\le
	\frac{2}{1-c}\,
	\frac{1}{(\varepsilon\delta)^{2|\mathcal{K}|_\infty}}\,
	\card J\,
	(e^{-\beta}\Delta_\mathcal{K})^{n-s}
	\max_{K\in\mathcal{K}}
	\mathbf{E}[\|\mu^K-\nu^K\|^2]^{1/2},
\end{multline*}
for every $s<n$, $K\in\mathcal{K}$ and $J\subseteq K$.
\end{cor}

\begin{proof}
The proof is a direct adaptation of \cite[Proposition 4.17]{RvH13}.
\end{proof}

The block filter stability result in \cite{RvH13} is the only place in the 
proof of the variance bound where the inadequacy of the classical 
comparison theorem plays a role.  Having exploited the generalized 
comparison Theorem \ref{thm:main} to extend the stability results in 
\cite{RvH13} to the present setting, we would therefore expect that the 
remainder of the proof of the variance bound follows verbatim from 
\cite{RvH13}.  Unfortunately, however, there is a complication: the result 
of Corollary \ref{cor:bfstab} is not as powerful as the corresponding 
result in \cite{RvH13}.  Note that the first (uniform) bound in Corollary 
\ref{cor:bfstab} decays exponentially in time $n$, but the second (initial 
condition dependent) bound only decays in $n$ if it happens to be the case 
that $e^{-\beta}\Delta_\mathcal{K}<1$.  As in \cite{RvH13} both the 
spatial and temporal interactions were assumed to be sufficiently weak, we 
could assume that the latter was always the case.  In the present setting, 
however, it is possible that $e^{-\beta}\Delta_\mathcal{K}\ge 1$ no matter 
how weak are the spatial correlations.  

To surmount this problem, we will use a slightly different error 
decomposition than was used in \cite{RvH13} to complete the proof of the 
variance bound.  The present approach is inspired by \cite{DG01}.  The 
price we pay is that the variance bound scales in the number of samples as 
$N^{-\gamma}$ where $\gamma$ may be less than the optimal (by the central 
limit theorem) rate $\frac{1}{2}$.  It is likely that a more sophisticated 
method of proof would yield the optimal $N^{\frac{1}{2}}$ rate in the 
variance bound.  However, let us note that in order to put the block 
particle filter to good use we must optimize over the size of the blocks 
in $\mathcal{K}$, and optimizing the error bound in Theorem 
\ref{thm:block} yields at best a rate of order $N^{-\alpha}$ for some 
constant $\alpha$ depending on the constants $\beta_1,\beta_2$.  As the 
proof of Theorem \ref{thm:block} is not expected to yield realistic values 
for the constants $\beta_1,\beta_2$, the suboptimality of the variance 
rate $\gamma$ does not significantly alter the practical conclusions that 
can be drawn from Theorem \ref{thm:block}.

We now proceed to the variance bound.  The following is
the main result of this section.

\begin{theorem}[\rm Variance term]
\label{thm:variance}
Suppose there exist $0<\varepsilon,\delta,\kappa<1$ such that 
\begin{align*}
	\varepsilon q^v(x^v,z^v) &\le p^v(x,z^v)\le
	\varepsilon^{-1}q^v(x^v,z^v), \\
	\delta &\le q^v(x^v,z^v)\le \delta^{-1}, \\
	\kappa &\le g^v(x^v,y^v)\le \kappa^{-1}
\end{align*}
for every $v\in V$, $x,z\in\bbX$, and $y\in\bbY$, where 
$q^v:\bbX^v\times\bbX^v\to\mathbb{R}_+$ is a transition density with
respect to $\psi^v$.  Suppose also that we can choose $q\in\mathbb{N}$ and 
$\beta>0$ such that
$$
	c:=
	3q\Delta^2e^{\beta q}(1-\varepsilon^{2(\Delta+1)}) +
	e^\beta(1-\varepsilon^2\delta^2) +
	e^{\beta q}(1-\varepsilon^2\delta^2)^q < 1.
$$
Then for every $n\ge 0$, $\sigma\in\bbX$, $K\in\mathcal{K}$ and
$J\subseteq K$, the following hold:
\begin{enumerate}
\item If $e^{-\beta}\Delta_\mathcal{K}<1$, we have
$$
	\tnorm{\tilde\pi_n^\sigma-\hat\pi_n^\sigma}_J \le
	\card J\,
	\frac{32\Delta_\mathcal{K}}{1-c}\,
	\frac{2-e^{-\beta}\Delta_\mathcal{K}}{
	1-e^{-\beta}\Delta_\mathcal{K}}\,
	\frac{(\varepsilon\delta\kappa^{\Delta_\mathcal{K}})^{-4|\mathcal{K}|_\infty}
	}{N^{\frac{1}{2}}}.
$$
\item If $e^{-\beta}\Delta_\mathcal{K}=1$, we have
$$
	\tnorm{\tilde\pi_n^\sigma-\hat\pi_n^\sigma}_J \le
	\card J\,
	\frac{16\beta^{-1}\Delta_\mathcal{K}}{1-c}\,
	(\varepsilon\delta\kappa^{\Delta_\mathcal{K}})^{-4|\mathcal{K}|_\infty}
	\,\frac{3+\log N}{N^{\frac{1}{2}}}.
$$
\item If $e^{-\beta}\Delta_\mathcal{K}>1$, we have
$$
	\tnorm{\tilde\pi_n^\sigma-\hat\pi_n^\sigma}_J \le
	\card J\,
	\frac{32\Delta_\mathcal{K}}{1-c}
	\,
	\bigg\{
	\frac{1}{
	e^{-\beta}\Delta_\mathcal{K}-1} + 2\bigg\}
	\,
	\frac{
	(\varepsilon\delta\kappa^{\Delta_\mathcal{K}})^{-4|\mathcal{K}|_\infty}
	}{N^{\frac{\beta}{2\log\Delta_\mathcal{K}}}}.
$$
\end{enumerate}
\end{theorem}

The proof of Theorem \ref{thm:variance} combines the stability bounds of 
Corollary \ref{cor:bfstab} and one-step bounds on the sampling error, 
\cite[Lemma 4.19 and Proposition 4.22]{RvH13}, that can be 
used verbatim in the present setting.  We recall the latter here for the 
reader's convenience.

\begin{prop}[\rm Sampling error]
\label{prop:onestep}
Suppose there exist $0<\varepsilon,\delta,\kappa<1$ such that 
\begin{align*}
	\varepsilon q^v(x^v,z^v) &\le p^v(x,z^v)\le
	\varepsilon^{-1}q^v(x^v,z^v), \\
	\delta &\le q^v(x^v,z^v)\le \delta^{-1}, \\
	\kappa &\le g^v(x^v,y^v)\le \kappa^{-1}
\end{align*}
for every $v\in V$, $x,z\in\bbX$, and $y\in\bbY$.  Then we have
$$
	\max_{K\in\mathcal{K}}\tnorm{\mathsf{\tilde F}_n\hat\pi_{n-1}^\sigma-
	\mathsf{\hat F}_n\hat\pi_{n-1}^\sigma}_K \le
	\frac{2\kappa^{-2|\mathcal{K}|_\infty}}{N^{\frac{1}{2}}}
$$
and
$$
	\max_{K\in\mathcal{K}}\mathbf{E}[
	\|\mathsf{\tilde F}_{s+1}\mathsf{\tilde F}_{s}
	\hat\pi_{s-1}^\sigma-
	\mathsf{\tilde F}_{s+1}\mathsf{\hat F}_s
	\hat\pi_{s-1}^\sigma\|_K^2]^{1/2}
	\le
	\frac{16\Delta_\mathcal{K}
	(\varepsilon\delta)^{-2|\mathcal{K}|_\infty}
	\kappa^{-4|\mathcal{K}|_\infty\Delta_\mathcal{K}}}{N^{\frac{1}{2}}}
$$
for every $0<s<n$ and $\sigma\in\bbX$.
\end{prop}

\begin{proof}
Immediate from \cite[Lemma 4.19 and Proposition 4.22]{RvH13} upon 
replacing $\varepsilon$ by $\varepsilon\delta$.
\end{proof}

We can now prove Theorem \ref{thm:variance}

\begin{proof}[Proof of Theorem \ref{thm:variance}]
We fix for the time being an integer $t\ge 1$ (we will 
optimize over $t$ at the end of the proof).
We argue differently when $n\le t$ and when $n>t$.

Suppose first that $n\le t$.  In this case, we estimate
\begin{align*}
	\tnorm{\tilde\pi_n^\sigma-\hat\pi_n^\sigma}_J &=
	\tnorm{\mathsf{\tilde F}_n\cdots\mathsf{\tilde F}_1\delta_\sigma-
	\mathsf{\hat F}_n\cdots\mathsf{\hat F}_1\delta_\sigma}_J \\ 
	&\le
	\sum_{k=1}^n
	\tnorm{\mathsf{\tilde F}_n\cdots\mathsf{\tilde F}_{k+1}
	\mathsf{\tilde F}_{k}\hat\pi_{k-1}^\sigma -
	\mathsf{\tilde F}_n\cdots\mathsf{\tilde F}_{k+1}
	\mathsf{\hat F}_{k}\hat\pi_{k-1}^\sigma}_J
\end{align*}
using a telescoping sum and the triangle inequality.
The term $k=n$ in the sum is estimated by the first bound in Proposition 
\ref{prop:onestep}, while the remaining terms are estimated by the second 
bound of Corollary \ref{cor:bfstab} and Proposition 
\ref{prop:onestep}, respectively.  This yields
$$
	\tnorm{\tilde\pi_n^\sigma-\hat\pi_n^\sigma}_J \le
	\card J\,
	\frac{32\Delta_\mathcal{K}}{1-c}\,
	\frac{(\varepsilon\delta\kappa^{\Delta_\mathcal{K}})^{-4|\mathcal{K}|_\infty}
	}{N^{\frac{1}{2}}}
	\bigg\{
	\frac{(e^{-\beta}\Delta_\mathcal{K})^{n-1}-1}{
	e^{-\beta}\Delta_\mathcal{K}-1}+1\bigg\}
$$
(in the case $e^{-\beta}\Delta_\mathcal{K}=1$, the quantity between the
brackets $\{\mbox{}\cdot\mbox{}\}$ equals $n$).

Now suppose that $n>t$.  Then we decompose the error as
\begin{align*}
	\tnorm{\tilde\pi_n^\sigma-\hat\pi_n^\sigma}_J &\le
	\tnorm{
	\mathsf{\tilde F}_n\cdots\mathsf{\tilde F}_{n-t+1}\tilde\pi_{n-t}^\sigma -
	\mathsf{\tilde F}_n\cdots\mathsf{\tilde F}_{n-t+1}\hat\pi_{n-t}^\sigma
	}_J \\ &\qquad\mbox{}
	+ \sum_{k=n-t+1}^n
	\tnorm{\mathsf{\tilde F}_n\cdots\mathsf{\tilde F}_{k+1}
	\mathsf{\tilde F}_{k}\hat\pi_{k-1}^\sigma -
	\mathsf{\tilde F}_n\cdots\mathsf{\tilde F}_{k+1}
	\mathsf{\hat F}_{k}\hat\pi_{k-1}^\sigma}_J,
\end{align*}
that is, we develop the telescoping sum for $t$ steps only.
The first term is estimated by the first bound in Corollary 
\ref{cor:bfstab}, while the sum is estimated as in the case $n\le t$.
This yields
$$
	\tnorm{\tilde\pi_n^\sigma-\hat\pi_n^\sigma}_J \le
	\frac{\card J}{1-c}\bigg[
	2e^{-\beta t} +
	\frac{
	32\Delta_\mathcal{K}
	(\varepsilon\delta\kappa^{\Delta_\mathcal{K}})^{-4|\mathcal{K}|_\infty}
	}{N^{\frac{1}{2}}}
	\bigg\{
	\frac{(e^{-\beta}\Delta_\mathcal{K})^{t-1}-1}{
	e^{-\beta}\Delta_\mathcal{K}-1}+1\bigg\}\bigg]
$$
(in the case $e^{-\beta}\Delta_\mathcal{K}=1$, the quantity between the
brackets $\{\mbox{}\cdot\mbox{}\}$ equals $t$).

We now consider separately the three cases in the statement of the 
Theorem.

\textbf{Case 1.} In this case we choose $t=n$, and note that
$$
	\frac{(e^{-\beta}\Delta_\mathcal{K})^{n-1}-1}{
	e^{-\beta}\Delta_\mathcal{K}-1}+1 \le
	\frac{2-e^{-\beta}\Delta_\mathcal{K}}{
	1-e^{-\beta}\Delta_\mathcal{K}}\quad\mbox{for all }n\ge 1.
$$
Thus the result follows from the first bound above.

\textbf{Case 2.} In this case we have
$$
	\tnorm{\tilde\pi_n^\sigma-\hat\pi_n^\sigma}_J \le
	\frac{\card J}{1-c}\bigg[
	2e^{-\beta t} +
	\frac{
	32\Delta_\mathcal{K}
	(\varepsilon\delta\kappa^{\Delta_\mathcal{K}})^{-4|\mathcal{K}|_\infty}
	}{N^{\frac{1}{2}}}\,
	t\bigg]
$$
for all $t,n\ge 1$.  Now choose $t=\lceil (2\beta)^{-1}\log N\rceil$.  Then
$$
	\tnorm{\tilde\pi_n^\sigma-\hat\pi_n^\sigma}_J \le
	\frac{\card J}{1-c}\bigg[
	16\beta^{-1}\Delta_\mathcal{K}
	(\varepsilon\delta\kappa^{\Delta_\mathcal{K}})^{-4|\mathcal{K}|_\infty}
	\,\frac{\log N}{N^{\frac{1}{2}}}
	+\frac{
	34\Delta_\mathcal{K}
	(\varepsilon\delta\kappa^{\Delta_\mathcal{K}})^{-4|\mathcal{K}|_\infty}
	}{N^{\frac{1}{2}}}
	\bigg],
$$
which readily yields the desired bound.

\textbf{Case 3.}  In this case we have
$$
	\tnorm{\tilde\pi_n^\sigma-\hat\pi_n^\sigma}_J \le
	\frac{\card J}{1-c}\bigg[
	2e^{-\beta t} +
	\frac{
	32\Delta_\mathcal{K}
	(\varepsilon\delta\kappa^{\Delta_\mathcal{K}})^{-4|\mathcal{K}|_\infty}
	}{N^{\frac{1}{2}}}
	\bigg\{
	\frac{(e^{-\beta}\Delta_\mathcal{K})^{t-1}-1}{
	e^{-\beta}\Delta_\mathcal{K}-1}+1\bigg\}\bigg]
$$
for all $t,n\ge 1$.  Now choose $t=\Big\lceil \frac{\log 
N}{2\log\Delta_\mathcal{K}}\Big\rceil$.  Then
$$
	\tnorm{\tilde\pi_n^\sigma-\hat\pi_n^\sigma}_J \le
	\card J\,
	\frac{32\Delta_\mathcal{K}}{1-c}
	\,
	\bigg\{
	\frac{1}{
	e^{-\beta}\Delta_\mathcal{K}-1} + 2\bigg\}
	\,
	\frac{
	(\varepsilon\delta\kappa^{\Delta_\mathcal{K}})^{-4|\mathcal{K}|_\infty}
	}{N^{\frac{\beta}{2\log\Delta_\mathcal{K}}}},
$$
and the proof is complete.
\end{proof}

The conclusion of Theorem \ref{thm:block} now follows readily from 
Theorems \ref{thm:bias} and \ref{thm:variance}.  We must only check
that the assumptions Theorems \ref{thm:bias} and \ref{thm:variance} are 
satisfied.  The assumption of Theorem \ref{thm:bias} is slightly stronger 
than that of Theorem \ref{thm:variance}, so it suffices to consider the 
former.  To this end, fix $0<\delta<1$, and choose $q\in\mathbb{N}$ 
such that
$$
	1-\delta^2 + (1-\delta^2)^q < 1.
$$
Then we may evidently choose $0<\varepsilon_0<1$, depending on $\delta$ 
and $\Delta$ only, such that
$$
	3q\Delta^2(1-\varepsilon^{2(\Delta+1)})+
	1-\varepsilon^2\delta^2 +
	(1-\varepsilon^2\delta^2)^q < 1
$$
for all $\varepsilon_0<\varepsilon\le 1$.  This is the constant 
$\varepsilon_0$ that appears in the statement of Theorem \ref{thm:block}.
Finally, it is now clear that we can choose $\beta>0$ sufficiently close 
to zero (depending on $\delta,\varepsilon,r,\Delta$ only) such that $c<1$.
Thus the proof of Theorem \ref{thm:block} is complete.


\begin{thebibliography}{27}

\bibitem{CMR05}
\begin{bbook}[author]
\bauthor{\bsnm{Capp{\'e}},~\bfnm{Olivier}\binits{O.}},
  \bauthor{\bsnm{Moulines},~\bfnm{Eric}\binits{E.}} \AND
  \bauthor{\bsnm{Ryd{\'e}n},~\bfnm{Tobias}\binits{T.}}
(\byear{2005}).
\btitle{Inference in hidden {M}arkov models}.
\bseries{Springer Series in Statistics}.
\bpublisher{Springer}, \baddress{New York}.
\bmrnumber{2159833 (2006e:60002)}
\end{bbook}
\endbibitem

\bibitem{DSS12}
\begin{barticle}[author]
\bauthor{\bsnm{Dai~Pra},~\bfnm{Paolo}\binits{P.}},
  \bauthor{\bsnm{Scoppola},~\bfnm{Benedetto}\binits{B.}} \AND
  \bauthor{\bsnm{Scoppola},~\bfnm{Elisabetta}\binits{E.}}
(\byear{2012}).
\btitle{Sampling from a {G}ibbs measure with pair interaction by means of
  {PCA}}.
\bjournal{J. Stat. Phys.}
\bvolume{149}
\bpages{722--737}.
\bdoi{10.1007/s10955-012-0612-9}.
\bmrnumber{2998598}
\end{barticle}
\endbibitem

\bibitem{RFS08}
\begin{barticle}[author]
\bauthor{\bparticle{de~la} \bsnm{Rue},~\bfnm{T.}\binits{T.}},
  \bauthor{\bsnm{Fern{\'a}ndez},~\bfnm{R.}\binits{R.}} \AND
  \bauthor{\bsnm{Sokal},~\bfnm{A.~D.}\binits{A.~D.}}
(\byear{2008}).
\btitle{How to clean a dirty floor: probabilistic potential theory and the
  {D}obrushin uniqueness theorem}.
\bjournal{Markov Process. Related Fields}
\bvolume{14}
\bpages{1--78}.
\bmrnumber{2433296 (2009g:60099)}
\end{barticle}
\endbibitem

\bibitem{DG01}
\begin{barticle}[author]
\bauthor{\bsnm{Del~Moral},~\bfnm{Pierre}\binits{P.}} \AND
  \bauthor{\bsnm{Guionnet},~\bfnm{Alice}\binits{A.}}
(\byear{2001}).
\btitle{On the stability of interacting processes with applications to
  filtering and genetic algorithms}.
\bjournal{Ann. Inst. H. Poincar\'e Probab. Statist.}
\bvolume{37}
\bpages{155--194}.
\bdoi{10.1016/S0246-0203(00)01064-5}.
\bmrnumber{1819122 (2002k:60013)}
\end{barticle}
\endbibitem

\bibitem{DS85}
\begin{bincollection}[author]
\bauthor{\bsnm{Dobrushin},~\bfnm{R.~L.}\binits{R.~L.}} \AND
  \bauthor{\bsnm{Shlosman},~\bfnm{S.~B.}\binits{S.~B.}}
(\byear{1985}).
\btitle{Constructive criterion for the uniqueness of {G}ibbs field}.
In \bbooktitle{Statistical physics and dynamical systems ({K}\"oszeg, 1984)}.
\bseries{Progr. Phys.}
\bvolume{10}
\bpages{347--370}.
\bpublisher{Birkh\"auser Boston}, \baddress{Boston, MA}.
\bmrnumber{821306 (87d:82006)}
\end{bincollection}
\endbibitem

\bibitem{Dob70}
\begin{barticle}[author]
\bauthor{\bsnm{Dobru{\v{s}}in},~\bfnm{R.~L.}\binits{R.~L.}}
(\byear{1970}).
\btitle{Definition of a system of random variables by means of conditional
  distributions}.
\bjournal{Teor. Verojatnost. i Primenen.}
\bvolume{15}
\bpages{469--497}.
\bmrnumber{0298716 (45 \#\#7765)}
\end{barticle}
\endbibitem

\bibitem{Dud02}
\begin{bbook}[author]
\bauthor{\bsnm{Dudley},~\bfnm{R.~M.}\binits{R.~M.}}
(\byear{2002}).
\btitle{Real analysis and probability}.
\bseries{Cambridge Studies in Advanced Mathematics}
\bvolume{74}.
\bpublisher{Cambridge University Press}, \baddress{Cambridge}.
\bnote{Revised reprint of the 1989 original}.
\bdoi{10.1017/CBO9780511755347}.
\bmrnumber{1932358 (2003h:60001)}
\end{bbook}
\endbibitem

\bibitem{DGJ08}
\begin{barticle}[author]
\bauthor{\bsnm{Dyer},~\bfnm{Martin}\binits{M.}},
  \bauthor{\bsnm{Goldberg},~\bfnm{Leslie~Ann}\binits{L.~A.}} \AND
  \bauthor{\bsnm{Jerrum},~\bfnm{Mark}\binits{M.}}
(\byear{2008}).
\btitle{Dobrushin conditions and systematic scan}.
\bjournal{Combin. Probab. Comput.}
\bvolume{17}
\bpages{761--779}.
\bdoi{10.1017/S0963548308009437}.
\bmrnumber{2463409 (2009m:60247)}
\end{barticle}
\endbibitem

\bibitem{DGJ09}
\begin{barticle}[author]
\bauthor{\bsnm{Dyer},~\bfnm{Martin}\binits{M.}},
  \bauthor{\bsnm{Goldberg},~\bfnm{Leslie~Ann}\binits{L.~A.}} \AND
  \bauthor{\bsnm{Jerrum},~\bfnm{Mark}\binits{M.}}
(\byear{2009}).
\btitle{Matrix norms and rapid mixing for spin systems}.
\bjournal{Ann. Appl. Probab.}
\bvolume{19}
\bpages{71--107}.
\bdoi{10.1214/08-AAP532}.
\bmrnumber{2498672 (2010c:60205)}
\end{barticle}
\endbibitem

\bibitem{Fol79}
\begin{barticle}[author]
\bauthor{\bsnm{F{\"o}llmer},~\bfnm{H.}\binits{H.}}
(\byear{1979}).
\btitle{Tail structure of {M}arkov chains on infinite product spaces}.
\bjournal{Z. Wahrsch. Verw. Gebiete}
\bvolume{50}
\bpages{273--285}.
\bdoi{10.1007/BF00534151}.
\bmrnumber{554547 (82c:60176)}
\end{barticle}
\endbibitem

\bibitem{Fol82}
\begin{barticle}[author]
\bauthor{\bsnm{F{\"o}llmer},~\bfnm{H.}\binits{H.}}
(\byear{1982}).
\btitle{A covariance estimate for {G}ibbs measures}.
\bjournal{J. Funct. Anal.}
\bvolume{46}
\bpages{387--395}.
\bdoi{10.1016/0022-1236(82)90053-2}.
\bmrnumber{661878 (84d:60142)}
\end{barticle}
\endbibitem

\bibitem{Fol88}
\begin{bincollection}[author]
\bauthor{\bsnm{F{\"o}llmer},~\bfnm{Hans}\binits{H.}}
(\byear{1988}).
\btitle{Random fields and diffusion processes}.
In \bbooktitle{\'{E}cole d'\'{E}t\'e de {P}robabilit\'es de {S}aint-{F}lour
  {XV}--{XVII}, 1985--87}.
\bseries{Lecture Notes in Math.}
\bvolume{1362}
\bpages{101--203}.
\bpublisher{Springer}, \baddress{Berlin}.
\bdoi{10.1007/BFb0086180}.
\bmrnumber{983373 (90f:60099)}
\end{bincollection}
\endbibitem

\bibitem{Geo11}
\begin{bbook}[author]
\bauthor{\bsnm{Georgii},~\bfnm{Hans-Otto}\binits{H.-O.}}
(\byear{2011}).
\btitle{Gibbs measures and phase transitions},
\bedition{second} ed.
\bseries{de Gruyter Studies in Mathematics}
\bvolume{9}.
\bpublisher{Walter de Gruyter \& Co.}, \baddress{Berlin}.
\bdoi{10.1515/9783110250329}.
\bmrnumber{2807681 (2012d:82015)}
\end{bbook}
\endbibitem

\bibitem{GZ03}
\begin{bincollection}[author]
\bauthor{\bsnm{Guionnet},~\bfnm{A.}\binits{A.}} \AND
  \bauthor{\bsnm{Zegarlinski},~\bfnm{B.}\binits{B.}}
(\byear{2003}).
\btitle{Lectures on logarithmic {S}obolev inequalities}.
In \bbooktitle{S\'eminaire de {P}robabilit\'es, {XXXVI}}.
\bseries{Lecture Notes in Math.}
\bvolume{1801}
\bpages{1--134}.
\bpublisher{Springer}, \baddress{Berlin}.
\bdoi{10.1007/978-3-540-36107-7_1}.
\bmrnumber{1971582 (2004b:60226)}
\end{bincollection}
\endbibitem

\bibitem{Kal02}
\begin{bbook}[author]
\bauthor{\bsnm{Kallenberg},~\bfnm{Olav}\binits{O.}}
(\byear{2002}).
\btitle{Foundations of modern probability},
\bedition{second} ed.
\bseries{Probability and its Applications (New York)}.
\bpublisher{Springer-Verlag}, \baddress{New York}.
\bmrnumber{1876169 (2002m:60002)}
\end{bbook}
\endbibitem

\bibitem{Kul03}
\begin{barticle}[author]
\bauthor{\bsnm{K{\"u}lske},~\bfnm{Christof}\binits{C.}}
(\byear{2003}).
\btitle{Concentration inequalities for functions of {G}ibbs fields with
  application to diffraction and random {G}ibbs measures}.
\bjournal{Comm. Math. Phys.}
\bvolume{239}
\bpages{29--51}.
\bdoi{10.1007/s00220-003-0841-5}.
\bmrnumber{1997114 (2004i:60069)}
\end{barticle}
\endbibitem

\bibitem{LMS90}
\begin{barticle}[author]
\bauthor{\bsnm{Lebowitz},~\bfnm{Joel~L.}\binits{J.~L.}},
  \bauthor{\bsnm{Maes},~\bfnm{Christian}\binits{C.}} \AND
  \bauthor{\bsnm{Speer},~\bfnm{Eugene~R.}\binits{E.~R.}}
(\byear{1990}).
\btitle{Statistical mechanics of probabilistic cellular automata}.
\bjournal{J. Statist. Phys.}
\bvolume{59}
\bpages{117--170}.
\bdoi{10.1007/BF01015566}.
\bmrnumber{1049965 (91i:82012)}
\end{barticle}
\endbibitem

\bibitem{Lig05}
\begin{bbook}[author]
\bauthor{\bsnm{Liggett},~\bfnm{Thomas~M.}\binits{T.~M.}}
(\byear{2005}).
\btitle{Interacting particle systems}.
\bseries{Classics in Mathematics}.
\bpublisher{Springer-Verlag}, \baddress{Berlin}.
\bnote{Reprint of the 1985 original}.
\bmrnumber{2108619 (2006b:60003)}
\end{bbook}
\endbibitem

\bibitem{RvH13}
\begin{bmisc}[author]
\bauthor{\bsnm{Rebeschini},~\bfnm{Patrick}\binits{P.}} \AND \bauthor{\bsnm{{van
  Handel}},~\bfnm{Ramon}\binits{R.}}
(\byear{2013}).
\btitle{Can local particle filters beat the curse of dimensionality?}
\bnote{Preprint arxiv:1301.6585}.
\end{bmisc}
\endbibitem

\bibitem{Roy88}
\begin{bbook}[author]
\bauthor{\bsnm{Royden},~\bfnm{H.~L.}\binits{H.~L.}}
(\byear{1988}).
\btitle{Real analysis}, \bedition{Third} ed.
\bpublisher{Macmillan Publishing Company}, \baddress{New York}.
\bmrnumber{1013117 (90g:00004)}
\end{bbook}
\endbibitem

\bibitem{Sim93}
\begin{bbook}[author]
\bauthor{\bsnm{Simon},~\bfnm{Barry}\binits{B.}}
(\byear{1993}).
\btitle{The statistical mechanics of lattice gases. {V}ol. {I}}.
\bseries{Princeton Series in Physics}.
\bpublisher{Princeton University Press}, \baddress{Princeton, NJ}.
\bmrnumber{1239893 (95a:82001)}
\end{bbook}
\endbibitem

\bibitem{Tat03}
\begin{binproceedings}[author]
\bauthor{\bsnm{Tatikonda},~\bfnm{Sekhar~C.}\binits{S.~C.}}
(\byear{2003}).
\btitle{Convergence of the Sum-Product Algorithm}.
In \bbooktitle{Information Theory Workshop, 2003. Proceedings. 2003 IEEE}
\bpages{222-225}.
\bdoi{10.1109/ITW.2003.1216735}
\end{binproceedings}
\endbibitem

\bibitem{VW96}
\begin{bbook}[author]
\bauthor{\bparticle{van~der} \bsnm{Vaart},~\bfnm{Aad~W.}\binits{A.~W.}} \AND
  \bauthor{\bsnm{Wellner},~\bfnm{Jon~A.}\binits{J.~A.}}
(\byear{1996}).
\btitle{Weak convergence and empirical processes}.
\bseries{Springer Series in Statistics}.
\bpublisher{Springer-Verlag}, \baddress{New York}.
\bnote{With applications to statistics}.
\bmrnumber{1385671 (97g:60035)}
\end{bbook}
\endbibitem

\bibitem{Vil09}
\begin{bbook}[author]
\bauthor{\bsnm{Villani},~\bfnm{C{\'e}dric}\binits{C.}}
(\byear{2009}).
\btitle{Optimal transport}.
\bseries{Grundlehren der Mathematischen Wissenschaften [Fundamental Principles
  of Mathematical Sciences]}
\bvolume{338}.
\bpublisher{Springer-Verlag}, \baddress{Berlin}.
\bnote{Old and new}.
\bdoi{10.1007/978-3-540-71050-9}.
\bmrnumber{2459454 (2010f:49001)}
\end{bbook}
\endbibitem

\bibitem{Wei05}
\begin{barticle}[author]
\bauthor{\bsnm{Weitz},~\bfnm{Dror}\binits{D.}}
(\byear{2005}).
\btitle{Combinatorial criteria for uniqueness of {G}ibbs measures}.
\bjournal{Random Structures Algorithms}
\bvolume{27}
\bpages{445--475}.
\bdoi{10.1002/rsa.20073}.
\bmrnumber{2178257 (2006k:82036)}
\end{barticle}
\endbibitem

\bibitem{Wu06}
\begin{barticle}[author]
\bauthor{\bsnm{Wu},~\bfnm{Liming}\binits{L.}}
(\byear{2006}).
\btitle{Poincar\'e and transportation inequalities for {G}ibbs measures under
  the {D}obrushin uniqueness condition}.
\bjournal{Ann. Probab.}
\bvolume{34}
\bpages{1960--1989}.
\bdoi{10.1214/009117906000000368}.
\bmrnumber{2271488 (2008e:60308)}
\end{barticle}
\endbibitem

\bibitem{You89}
\begin{barticle}[author]
\bauthor{\bsnm{Younes},~\bfnm{Laurent}\binits{L.}}
(\byear{1989}).
\btitle{Parametric inference for imperfectly observed {G}ibbsian fields}.
\bjournal{Probab. Theory Related Fields}
\bvolume{82}
\bpages{625--645}.
\bdoi{10.1007/BF00341287}.
\bmrnumber{1002904 (91a:62072)}
\end{barticle}
\endbibitem

\end{thebibliography}

\end{document}